\documentclass[reqno,12pt]{amsart}
\usepackage[colorlinks=true, linkcolor=blue, citecolor=blue]{hyperref}

\usepackage{amssymb}
\usepackage{amsmath, graphicx, rotating}
\usepackage{color}
\usepackage{soul}
\usepackage[dvipsnames]{xcolor}

\allowdisplaybreaks
\usepackage{ifthen}
\usepackage{xkeyval}
\usepackage{todonotes}
\usepackage[T1]{fontenc}
\usepackage{lmodern}
\usepackage[english]{babel}

\usepackage{ upgreek }
\usepackage{stmaryrd}
\SetSymbolFont{stmry}{bold}{U}{stmry}{m}{n}
\usepackage{amsthm}
\usepackage{float}

\usepackage{ bbm }
\usepackage{ stmaryrd }
\usepackage{ mathrsfs }
\usepackage{ frcursive }
\usepackage{ comment }

\usepackage{pgf, tikz}
\usetikzlibrary{shapes}
\usepackage{varioref}
\usepackage{enumitem}

\definecolor{rouge}{rgb}{0.7,0.00,0.00}
\definecolor{vert}{rgb}{0.00,0.5,0.00}
\definecolor{bleu}{rgb}{0.00,0.00,0.8}
\usepackage[margin=1in]{geometry}
\newtheorem{theorem}{Theorem}[section]
\newtheorem*{theorem*}{Theorem}
\newtheorem{lemma}[theorem]{Lemma}

\newtheorem{proposition}[theorem]{Proposition}

\labelformat{hypothesis}{\textbf{M\kern-0.1mm#1}}

\newtheorem{condition}{Condition}

\newtheorem{conditionA}{A\kern-0.1mm}
\labelformat{conditionA}{\textbf{A\kern-0.1mm#1}}

\labelformat{conditionH}{\textbf{H\kern-0.1mm#1}}

\theoremstyle{definition}
\newtheorem{example}[theorem]{Example}
\newtheorem{remark}[theorem]{Remark}

\def \eref#1{\hbox{(\ref{#1})}}

\numberwithin{equation}{section}

\def\bf#1{\mathbf{#1}}

\def\geq{\geqslant}
\def\leq{\leqslant}

\def\RR{\mathbb{R}}
\def\PP{\mathbb{P}}
\def\EE{\mathbb{E}}

\def\vare{{\varepsilon}}
\def \eref#1{\hbox{(\ref{#1})}}

\def\EE{\mathbb{ E}}

\begin{document}

\title[Poisson equation and SDEs with State-Dependent Switching]
{Poisson Equation and Application to Multi-Scale SDEs with State-Dependent Switching}

\author{Xiaobin Sunsun}
\curraddr[Sun, X.]{ School of Mathematics and Statistics/RIMS, Jiangsu Normal University, Xuzhou, 221116, China}
\email{xbsun@jsnu.edu.cn}

\author{Yingchao Xie}
\curraddr[Xie, Y.]{ School of Mathematics and Statistics/RIMS, Jiangsu Normal University, Xuzhou, 221116, China}
\email{ycxie@jsnu.edu.cn}

\begin{abstract}
In this paper, we study the averaging principle and central limit theorem for multi-scale stochastic differential equations with state-dependent switching. To accomplish this, we first study the Poisson equation associated with a Markov chain and the regularity of its solutions. As applications of the results on the Poisson equations, we prove three averaging principle results and two central limit theorems results. The first averaging principle result is a strong convergence of order $1/2$ of the slow component $X^{\varepsilon}$ in the space $C([0,T],\RR^n)$. The second averaging principle result is a weak convergence of $X^{\varepsilon}$ in $C([0,T],\RR^n)$. The third averaging principle result is a weak convergence of order $1$ of $X^{\varepsilon}_t$ in $\RR^n$ for any fixed $t\ge 0$. The first central limit theorem type result is a weak convergence of $(X^{\varepsilon}-\bar{X})/\sqrt{\varepsilon}$ in $C([0,T],\RR^n)$, where $\bar{X}$ is the solution of the averaged equation. The second central limit theorem type result is a  weak convergence of order $1/2$ of $(X^{\varepsilon}_t-\bar{X}_t)/\sqrt{\varepsilon}$ in $\RR^n$ for fixed $t\ge 0$. Several examples are given to show that all the achieved orders are optimal.
\end{abstract}

\subjclass[2000]{ Primary 60H10}
\keywords{Poisson equation; Stochastic differential equations with state-dependent switching; Central limit theorem;  Averaging principle; Strong convergence; Weak convergence}

\maketitle
\tableofcontents

\section{Introduction}
Stochastic differential equations (SDEs) with state-dependent switching are used to describe the evolution of the following complicated system: There are countably many scenarios  $\mathbb{S}$ for the system; in each scenario, the system evolves according to a certain diffusion; the system can switch between scenarios  and the switching depends on the current location or state of the system. The switching process is typically modeled by a right-continuous pure jump process. More precisely, we will use $X_t$ to denote the spatial location (or state) of the system is in at time $t$ and $\alpha_t$ to denote the scenario that the system is in at time $t$. SDEs with state-dependent switching consist of the following system of stochastic equations:
\begin{align}\label{e:eqn1}
\begin{cases} dX_t=b(X_t, \alpha_t)dt+\sigma(X_t, \alpha_t)dW_t,\\
\mathbb{P}\{\alpha_{t+\Delta}=j|\alpha_{t}=i, X_s, \alpha_s, s\leq
t\}=\left\{\begin{array}{l}
\displaystyle q_{ij}(X_t)\Delta+o(\Delta),~~~~~~i\neq j\\
1+q_{ii}(X_t)\Delta+o(\Delta),~~~i=j,\end{array}\right.\\
(X_0,\alpha_0)=(x,\alpha)\in \RR^n \times \mathbb{S},
\end{cases}
\end{align}
where $\{W_t\}_{t\geq 0}$ is a standard $d$-dimensional Brownian motion on $ (\Omega,\mathcal{F},\mathbb{P})$ with natural filtration $\{\mathcal{F}_t\}_{t\geq 0}$, $b$ is a function from $\mathbb{R}^n\times\mathbb{S}$ to $\mathbb{R}^n$, $\sigma$ is a function from $\mathbb{R}^n\times\mathbb{S}$ to $\mathbb{R}^n\times \mathbb{R}^d$ and $Q(x)= (q_{ij}(x))_{1 \le i,j \le m_0}$ is a $Q$-matrix depending on $x\in \mathbb{R}^n$.

SDEs with state-dependent switching are widely used in various fields, including finance, engineering, biology, and physics. In finance, this kind of systems are used to model the behavior of financial assets and derivatives under the influence of both random market fluctuations and state-dependent trading strategies. In engineering, this kind of systems are used to study the dynamics of systems subjected to switching control policies that depend on the system's current state. Due to the intricate interplay between the diffusions and the switching mechanism, SDEs with state-dependent switching are more difficult to study than traditional SDEs.

Recently, SDEs with state-dependent switching have attracted  the attention of many researchers, with ongoing efforts aimed at deepening our understanding of these complex dynamical systems and  exploring their applications. For recent results in this area, we refer our readers to \cite{HNSX2019,MY2006,NYZ2017,S2015,SX2014,SX2018,XZ2017,XZW2021,YZ1998,YZ2010,YM2003} and references therein.

In this paper, we focus on situations where the switching process $\alpha_t$ evolves at a much faster rate compared to the diffusion process $X_t$. This scenario is particularly relevant in some  neuron models (see \cite[Section 3]{PTW2012}).
We introduce a small positive parameter $\varepsilon$ that represents the ratio of the time scales between the slow and fast components. The switching process exhibits rapid jumps on a time scale of $O(\varepsilon^{-1})$, which is described by the state-dependent transition matrix $Q^{\varepsilon}(x)=\frac{1}{\varepsilon}Q(x)$, while the diffusion process $X_t$ evolves on a slower time scale of $O(1)$. So we are interested in the following system of stochastic equations:
\begin{align}\label{e:eqn2}
\begin{cases}
dX^{\varepsilon}_t=b(X^{\varepsilon}_{t},\alpha^{\varepsilon}_{t})dt+\sigma( X^{\varepsilon}_t,\alpha^{\varepsilon}_t)dW_{t},\\
\mathbb{P}\left(\alpha^{\varepsilon}_{t+\Delta}=j|\alpha^{\varepsilon}_{t}=i,X^{\vare}_s,\alpha^{\varepsilon}_{s},s\leq t\right)=\left\{\begin{array}{l}
\displaystyle \varepsilon^{-1}q_{ij}(X^{\vare}_t)\Delta+o(\Delta),\quad i\neq j,\\
1+\varepsilon^{-1}q_{ii}(X^{\vare}_t)\Delta+o(\Delta),\quad i=j,\end{array}\right.\\
(X^{\varepsilon}_0,\alpha^{\varepsilon}_{0})=(x,\alpha)\in \mathbb{R}^n\times \mathbb{S}.
\end{cases}
\end{align}

More generally, one can consider the case $Q^{\varepsilon}(x)=\tilde{Q}(x)+\varepsilon^{-1}Q(x)$, in which the switching process exhibits two distinct time scales (see \cite[Part IV]{YZ2010} or \cite[Subsection 4.3]{PTW2012}). However, in this paper, we will only consider the case mentioned in the paragraph above, that is, we  will only consider the case when $\tilde{Q}(x)\equiv 0$.

Due to the different time scales and the  interactions between the fast and slow components, multi-scale SDEs with state-dependent switching are difficult to study directly.  It is very helpful to find a simpler equationthat captures the long-term evolution of the system. This leads to the study of the limit behavior of the solution $X^{\varepsilon}$ of \eqref{e:eqn2} as $\varepsilon\to 0$. This study is usually called the averaging principle. The limit process $\bar X$ is the unique solution of the following averaged equation:
\begin{eqnarray*}\left\{\begin{array}{l}
\displaystyle d\bar{X}_{t}=\bar{b}( \bar{X}_{t})dt+\bar{\sigma}(\bar X_t)dW_t,\\
\bar{X}_{0}=x\in \mathbb{R}^n,\end{array}\right. \label{OAR}
\end{eqnarray*}
where $\bar{b}$ and $\bar\sigma$ are the corresponding averaged coefficients (see Eq.\eref{AR21} below).

In the literature, two kinds of limits are usually considered: strong and weak convergences. Let $(\mathbb{H},\|\cdot\|_{\mathbb{H}})$ be a Banach space and $\mathbb{C}$ be a class of functions $f: \mathbb{H}\rightarrow \RR$.
Suppose $Y$ and $\{Y_{n}\}_{n\geq 1}$ are $\mathbb{H}$-valued random variables.
\begin{itemize}
\item $\{Y_{n}\}_{n\geq 1}$ is said to converges strongly to $Y$ in $\mathbb{H}$ if
\begin{eqnarray}
\lim_{n\rightarrow \infty}\mathbb{E}\|Y_{n}-Y\|_{\mathbb{H}}=0.\label{StrongA1}
\end{eqnarray}
Furthermore, we say the strong convergence is of order $k>0$ if
\begin{eqnarray}
\mathbb{E}\|Y_{n}-Y\|_{\mathbb{H}}\leq Cn^{-k}, \quad \forall n\geq 1,\label{StrongA2}
\end{eqnarray}
where $C>0$ is a constant.

\item $\{Y_{n}\}_{n\geq 1}$ is said to converges weakly to $Y$ in $\mathbb{H}$ under the set $\mathbb{C}$ if
\begin{eqnarray}
\lim_{n\rightarrow \infty}|\mathbb{E}\phi(Y_{n})-\mathbb{E}\phi(Y)|=0,\quad \forall \phi\in\mathbb{C}.\label{WeakA1}
\end{eqnarray}
Furthermore, we say the weak convergence is of order $k>0$ if
\begin{eqnarray}
|\mathbb{E}\phi(Y_{n})-\mathbb{E}\phi(Y)|\leq C_{\phi}n^{-k}, \quad \forall n\geq 1,\phi\in\mathbb{C},\label{WeakA2}
\end{eqnarray}
where $C_\phi>0$ is a constant which might depends on $\phi$.
\end{itemize}

The parameter $n\to\infty$ can be interpreted in different ways. In our case, we will use  the parameter $\varepsilon\to 0$. It is clear that if $\mathbb{C}$ is the set of bounded continuous functions on $\mathbb{H}$, weak convergence corresponds to convergence in law. On the other hand, if $\mathbb{C}$ represents the set of Lipschitz functions on $\mathbb{H}$, the weak convergence order is at least as good as the strong convergence order. In general, the weak order depends on the specific class $\mathbb{C}$ being considered.

Multi-scale systems find extensive applications in various fields such as nonlinear oscillations, chemical kinetics, biology, and climate dynamics, see e.g. \cite{BR,FW,PS2008,WTRY16}. For classical multi-scale SDEs, an often-used to achieve strong convergence \eref{StrongA1} or \eref{StrongA2} is Khasminskii's time discretization technique, see e.g. \cite{B2012,F2022,GP1,L2010,LRSX2020}. For weak convergence \eref{WeakA2}, a commonly used technique is the method of asymptotic expansion of solutions of Kolmogorov equations in the parameter $\vare$, see e.g. \cite{B2012,DSXZ, FWLL, KY2004}. In recent years, the Poisson equation technique has proven to be highly efficient in achieving strong convergence
\eref{StrongA2}, while a combination of the Poisson equation with the Kolmogorov equation has been pretty successful in the study of weak convergence \eref{WeakA2}, see e.g. \cite{B2020, RSX2021, RX2021, SSWX2023, SX2023, SXX2022}. For more related results on the averaging principle for multi-scale stochastic systems, see e.g. \cite{C2009,HL2020,LS2022}.


In the state-independent switching setting, i.e., the case with $Q(x)\equiv Q$, several papers have studied the averaging principle (e.g., \cite{BYY2017,PXY2018,PXYZ2018,YY2004}). In the state-dependent switching case,  particularly for models described by \eref{e:eqn2}, Faggionato et al. \cite{FGR2010} studied the averaging principle in terms of convergence in probability in the case  $\sigma\equiv 0$ and  $\mathbb{S}$ is finite. Mao and Shao \cite{MS2022} studied the averaging principle for the case $\mathbb{S}$ is countably infinite. However, the optimal convergence order was not established in the papers mentioned above.

It is well known that finding the optimal convergence rate is very important numerical applications, functional limit theorems and homogenization problems (see e.g. \cite{B2020, PS2008}). In this paper, we will use the Poisson equation technique to study the strong and weak convergences of solutions to \eref{e:eqn2} and try to find the optimal convergence rates.

~\\
\textbf{Poisson equation}: We first study the following Poisson equation on a countable space $\mathbb{S}$:
$$
-Q(x)\Phi(x,\cdot)(i)= F(x,i), \quad x\in\RR^n, i\in\mathbb{S},
$$
where $Q(x)$ is the generator of an $\mathbb{S}$-valued Markov chain $\{\alpha^{x}_t\}_{t\geq 0}$.
Assume that $F(x,i)$ satisfies the "\emph{central condition}":
\begin{eqnarray*}
\sum_{i\in\mathbb{S}}F(x,i)\mu^x_i=0,\quad \forall x\in \RR^n,
\end{eqnarray*}
where $\mu^x$ is the unique invariant measure of $\{\alpha^{x}_t\}_{t\geq 0}$. Under some suitable conditions, the above Poisson equation admits a solution
$$
\Phi(x,i)=\int_{0}^{\infty}\EE F(x,\alpha_{t}^{x,i})dt,
$$
where $\{\alpha^{x,i}_t\}_{t\geq 0}$  stands for the Markov chain $\{\alpha^{x}_t\}_{t\geq 0}$ with initial value $\alpha^{x}_0=i$.

In the last few decades, the Poisson equation and its applications in diffusion approximation and central limit type theorems in various stochastic multi-scale systems, as well as in  finding optimal convergence orders in both strong and weak senses, have been intensively studied.
For example, Pardoux and Veretennikov \cite{PV1,PV2}  studied the regularity of solutions to the following Poisson equation:
$$
-\mathcal{L}(x)\Phi(x,\cdot)(y)= b(x,y),
$$
where $b$ satisfies the similar "\emph{central condition}" and
$$\mathcal{L}(x)h(x,\cdot)(y)=\sum_{i,j}a_{ij}(x,y)\frac{\partial^2 h(x,y)}{\partial{y_i}\partial{y_j}}+\sum_{i}f_i(x,y)\frac{\partial h(x,y)}{\partial{y_i}}.$$
The regularity results were further applied to a study on the diffusion approximation for two-scaled diffusion processes. However, the papers mentioned above mainly focused on Poisson equations associated with diffusion processes and the techniques of these papers do not work in our discrete setting. The first purpose of this paper is to establish the regularity of the solution $\Phi(x,i)$ of the Poisson equation associated with a Markov chain. Specifically, under suitable conditions on $Q(x)$ and $F(x,i)$, we will study the derivatives $\partial_{x}\Phi(x,i)$ and $\partial^2_{x}\Phi(x,i)$ of $\Phi(x,i)$ with respect to $x$.

~~\\
\textbf{Averaging principle}: As an application of
the Poison equation, we study the strong and weak convergences of solutions of \eref{e:eqn2}.
For strong convergence, we get that when $\sigma(x,\alpha)=\sigma(x)$
\begin{eqnarray*}
\mathbb{E}\left(\sup_{0\leq t\leq T}\left |X_{t}^{\varepsilon}-\bar{X_{t}}\right |^{p}\right)\le C_{p,T}\varepsilon^{p/2},\quad \forall \varepsilon\in (0,1], \label{Omain}
\end{eqnarray*}
where $\bar X$ is the solution of the corresponding averaged equation as $\varepsilon\to0$. Note that condition $\sigma(x,\alpha)=\sigma(x)$ is necessary, otherwise the strong convergence may fails (see Remark \ref{R5.4}).

In Example \ref{Ex2.9} we will show that the strong order $1/2$ above is optimal. We focus on the case when the drift coefficient $b(x,\alpha)$ only satisfies the monotonicity condition with respect to $x$, particularly the case when $b(x,\alpha)=-|x|^2 x +\tilde{b}(\alpha)$, where $\tilde{b}$ is a bounded function on $\mathbb{S}$.
The  papers \cite{SSWX2023} and \cite{CLR2023} have examined the optimal strong convergence order when the switching process is also a diffusion. However, the results and techniques of these papers don't fit into our setting.

For weak convergence,  it is important to consider the dependence of $\sigma(x,\alpha)$ on $\alpha$. We will study the weak convergences of $\{X^{\varepsilon}\}_{\varepsilon\in (0,1]}$ in $C([0,T],\RR^n)$ and $\{X^{\varepsilon}_t\}_{\varepsilon\in (0,1]}$ in $\RR^n$ for fixed $t\ge 0$, respectively. For the former, we will use  the Poisson equation technique and the martingale problem method to identify the limiting process. For the latter, we will use combine the  Poisson equation and Kolmogorov equation to get a weak convergence of order $1$ under the set $C^4_p(\RR^n)$, the space of functions on $\RR^n$ with all partial derivatives up to order 4 being continuous and of polynomial growth. This order $1$ will be shown to be optimal by a concrete example (see Example \ref{Ex2.11}).

\vspace{-2mm}
~\\
\textbf{Central limit theorem}:
It is also very important to study the fluctuations of $X^\varepsilon$ around the averaged process $\bar{X}$ as $\varepsilon\to 0$. For this, it is crucial to examine the asymptotic behavior of the deviation between $X^{\varepsilon}$ and $\bar{X}$. $$Z_t^\varepsilon:=\frac{X^{\varepsilon}_t-\bar{X}_t}{\sqrt{\varepsilon}},\ \ \text{as}~\varepsilon\to 0.$$
Consequently, the second goal of this paper is to establish weak convergence of $\{Z^{\varepsilon}\}_{\varepsilon>0}$ to $Z$ in $C([0, T];\RR^n)$ as $\varepsilon\to0$, where $Z$ is the unique solution of a new stochastic differential equation that incorporates an additional explicit stochastic integral term.

 The central limit theorem (CLT), also known as normal deviations in classical multi-scale SDEs, has been extensively studied. For instance, Khasminskii \cite{K2} investigated normal deviations in a specific class of slow-fast stochastic systems with vanishing noise in the slow equation. Cerrai \cite{C09} extended Khasminskii's results to the infinite-dimensional case. However, the techniques used in \cite{C09, K2} are not applicable to fully coupled stochastic systems. In a different context, Wang and Roberts \cite{WR2012} examined the normal deviations of a stochastic system driven by additive noise in infinite dimensions, where a martingale problem approach was employed to characterize the limiting process. Recently, Hong et al. \cite{HLLS2023} investigated the CLT for a class of multi-scale McKean-Vlasov SDEs, where the martingale representation theorem was used to identify the corresponding limiting process.

In the case when $\sigma\equiv 0$ Pakdaman et al. \cite{PTW2012} used an asymptotic expansion technique to derive a CLT. Subsequently, the averaging principle and CLT for a specific class of multi-scale piecewise-deterministic Markov processes in infinite dimensions were investigated in \cite{GT2012} and \cite{GT2014} respectively. The martingale problem approach was used in \cite{PTW2012, GT2012, GT2014} to identify the limiting process. However, the corresponding weak convergence order has not been addressed in these papers.

In this paper, we will use the Poisson equation technique and the associated martingale problem to characterize the limiting process. Furthermore, we will combine  the Kolmogorov backward equation and the Poisson equation to study weak convergence order for $\{Z^{\varepsilon}_t\}_{\varepsilon>0}$ under the space $C^4_p(\RR^n)$. Specifically, we will show that for any $\phi\in C^4_p(\RR^n)$,
$$
|\mathbb{E}\phi(Z^{\varepsilon}_t)-\mathbb{E}\phi(Z_t)|\leq C_{\phi}\varepsilon^{1/2},
$$
where $C_\phi>0$ is a constant dependent on $\phi$. We will show that the order  $1/2$  above is optimal  by a specific example (see Example \ref{Ex2.9}). As far as we know, these are the first results on the optimal weak convergence order, erms of the CLT, for the stochastic system \eref{e:eqn2}.

~\\
\textbf{Structure of the paper:}
The rest of this paper is organized as follows. In Section 2, we first introduce the notation and assumptions, then we state our main results and give concrete examples to show that the obtained convergence orders are optimal.
In Section 3, we study the Poisson equation associated with a Markov chain. In Section 4, we give the proofs of the three averaging principles Theorems \ref{main result 1},  \ref{main result 1.2} and  \ref{main result 1.3}. The proofs of the  two types of convergences in CLT are given in Section 5.

The usage of $k_p$, $C_{T}$, and $C_{p,T}$ is specifically intended to highlight the dependence on the subscript parameters.

\section{Assumptions and main results}

We give some notations and assumptions firstly, then the main results are stated, some interesting examples are given to show that all the achieved orders are optimal finally.

\subsection{Notations}
In this paper, $\mathbb{S}$ is either the set of all positive integers or $\{1, 2, \ldots, m_0\}$ for some $m_0<\infty$. We  will use $\|\cdot\|_{\infty}$ to denote the supremum norm on $\mathbb{S}$ and $\|\mu-\nu\|_{\rm{var}}$ to denote the total variation distance between probability measures $\mu$ and $\nu$ on $\mathbb{S}$.
We will use $|\cdot|$ and  $\langle\cdot, \cdot\rangle$ to denote  Euclidean  norm and the inner product, respectively. We will use $\|\cdot\|$ to denote  matrix norm or the operator norm if there is no danger of confusion. For any positive integers $k, l$, we will use $\RR^k\otimes\RR^l$ to denote the set of $k\times l$ matrices.
$\RR^\infty\otimes \RR^\infty$ stands for he set of matrices $M=(M_{ij})_{1\leq i,j\leq m_0}$ with  $\|M\|_{\ell}:=\sup_{i\in\mathbb{S}}\sum_{j\in \mathbb{S}}|m_{ij}|<\infty$.

For positive integers $l_1,l_2,k$, the following function spaces will be used later in this paper:

(1) $C^k(\RR^n,\RR^{l_1}\otimes\RR^{l_2})$ (resp. $C^k_b(\RR^n,\RR^{l_1}\otimes\RR^{l_2})$ or $C^k_p(\RR^n,\RR^{l_1}\otimes\RR^{l_2})$) is the space of $\RR^{l_1}\otimes\RR^{l_2}$-valued functions $\varphi(x)$ on $\RR^n$ with all partial derivatives of up to order $k$ being continuous (resp. continuous and bounded or continuous and polynomial growth). $C^k(\RR^n,\RR^{\infty}\otimes\RR^{\infty})$ (resp. $C^k_b(\RR^n, \RR^{\infty}\otimes\RR^{\infty})$ or $C^k_p(\RR^n, \RR^{\infty}\otimes\RR^{\infty})$) stands for space of $\RR^{\infty}\otimes\RR^{\infty}$-valued functions $\varphi(x)$ on $\RR^n$ with  all partial derivatives of up to order $k$ being continuous (resp. continuous and bounded or continuous and polynomial growth).

(2) $C^k(\RR^n\times \mathbb{S},\RR^{l_1}\otimes \RR^{l_2})$ (resp. $C^k_{b}(\RR^n\times \mathbb{S},\RR^{l_1}\otimes \RR^{l_2})$ or $C^k_{p}(\RR^n\times \mathbb{S},\RR^{l_1}\otimes \RR^{l_2})$) is the space of all mappings $\varphi(x,i): \RR^n\times \mathbb{S}\rightarrow \RR^{l_1}\otimes \RR^{l_2}$ with $\varphi(\cdot,i)\in C^k(\RR^n,\RR^{l_1}\otimes \RR^{l_2})$ (resp. $C^k_{b}(\RR^n,\RR^{l_1}\otimes \RR^{l_2})$ or $C^k_p(\RR^n,\RR^{l_1}\otimes \RR^{l_2})$) for any $i\in\mathbb{S}$, and $\|\partial^j_x\varphi(x,\cdot)\|_{\infty}$ are finite (resp. bounded or polynomial growth) with respective to $x$, for any $j=0,\ldots,k$.

\subsection{Assumptions}  Suppose that $b$, $\sigma$ and $Q$ satisfy the following conditions:
\begin{conditionA}\label{A1} There exist $C,k>0$ such that for $x,y\in\RR^n, i,j\in\mathbb{S}$,
\begin{eqnarray}
&&\left\langle b(x,i)-b(y,j),x-y\right\rangle \leq C|x-y|^{2}+C|x-y|1_{i\neq j},\label{ConA11}\\
&&\|b(x,\cdot)-b(y,\cdot)\|_{\infty} \le C\left(1+|x|^{k}+|y|^k\right)|x-y|,\quad\|b(x,\cdot)\|_{\infty}\leq C(1+|x|^{k}),\label{ConA13}\\
&&\|\sigma (x,\cdot)-\sigma (y,\cdot)\|_{\infty}\leq C|x-y|,\quad\|\sigma(x,\cdot)\|_{\infty}\leq C(1+|x|).\label{ConA14}
\end{eqnarray}
	\end{conditionA}

\begin{conditionA}\label{A2}
(i) $Q(x)=(q_{ij}(x))_{i,j\in\mathbb{S}}:\RR^n\rightarrow \RR^{m_0}\otimes\RR^{m_0}$ is measurable and conservative, i.e., for $x\in\RR^n$,
$$
q_{ij}(x)\geq0 ~~\text{for}~~\forall i\neq j\in\mathbb{S},\quad\quad\quad
\sum_{j\in \mathbb{S}}q_{ij}(x)=0~~\text{for}~~\forall i\in \mathbb{S}.
$$

(ii) $Q(x)$ is irreducible, that is, for any $x\in\RR^n$, the equations
$$\mu^{x} Q(x)=0,\quad \text{with}\quad \sum_{i\in\mathbb{S}}\mu^x_i=1$$
have a unique solution $\mu^x=(\mu^x_1,\mu^x_2,\ldots,\mu^x_{m_0})$ with $\mu^x_i>0$ for all $i\in \mathbb{S}$.

(iii) Let $P^{x}_t:=(p^{x}_{ij}(t))_{i,j\in\mathbb{S}} $ be the transition probability matrix associated with $Q(x)$. $P^{x}_t$ is exponential ergodicity uniformly in $x$, i.e., there exist  $C>0,\lambda>0$ such that
\begin{eqnarray}
\sup_{i\in \mathbb{S},x\in\RR^n}\|p^{x}_{i\cdot}(t)-\mu^x\|_{\text{var}}\leq Ce^{-\lambda t},\quad \forall t>0.\label{ergodicity}
\end{eqnarray}
\end{conditionA}

\begin{conditionA}\label{A3}
Assume there exists $C>0$ such that
	\begin{eqnarray}
		\|Q(x)-Q(y)\|_{\ell}\leq C|x-y|,\quad \forall x,y\in\RR^n.\label{LipQ}
	\end{eqnarray}
Moreover, there exists $k>0$ such that
	\begin{eqnarray}
		K(x):=\sum_{i\in \mathbb{S}}\sum_{j\in\mathbb{S}\backslash\{i\}}q_{ij}(x)\leq C(1+|x|^k),\quad \forall x\in\RR^n.\label{Finte K}
\end{eqnarray}
\end{conditionA}

\begin{conditionA}\label{A4}
Assume that $Q\in C^1(\RR^n,\RR^{m_0}\otimes\RR^{m_0})$, $\sigma\in C^1_b(\RR^n\times\mathbb{S}, \RR^n\otimes\RR^d)$ and satisfies
\begin{eqnarray}
\inf_{x\in\RR^n,i\in \mathbb{S},z\in \RR^n\backslash\{0\}}\frac{\langle(\sigma(x,i)\sigma^{\ast}(x,i))\cdot z, z\rangle} {|z|^2}>0. \label{NonD}
		\end{eqnarray}
\end{conditionA}

The analysis of various types of convergence requires different regularity on $b$, $\sigma$ and $Q$. For the sake of simplicity, the following assumptions are presented for $k=2,3,4$.

\vspace{2mm}
\noindent
\textbf{${\bf H}_{k}.$} $b\in C^{k}_{p}(\RR^n\times\mathbb{S},\RR^n)$, $\sigma\in C^{k}_{p}(\RR^n\times\mathbb{S},\RR^n\otimes\RR^{d})$ and $Q\in C^{k}_p(\RR^n,\RR^{m_0}\otimes\RR^{m_0})$.

\begin{remark}
(i) Condition \eref{ConA11} implies that for any $x\in\RR^n$, $i\in \mathbb{S}$,
\begin{eqnarray}
\left\langle b(x,i),x\right\rangle \leq C(1+|x|^{2}). \label{ConA12}
\end{eqnarray}
\eref{ConA12} and \eref{ConA14} implies \cite[Assumption 3.4]{NY2016} holds with $V(x)=|x|^2$.  \eref{Finte K} yields \cite[Assumption 3.2 (ii)]{NY2016} holds. Hence, Eq.\eref{e:eqn2} has a unique strong solution by \cite[Theorem 3.5]{NY2016}.

(ii) Assumption \ref{A2} ensures that $Q(x)$ generates a unique Markov process $\{\alpha^{x}_t\}_{t\geq 0}$. The condition \eref{ergodicity} is called strong ergodicity (as discussed in \cite{C2005}) for $Q(x)$ uniformly in $x$. This is equivalent to the following expression:
$$
\lim_{t\rightarrow \infty}\sup_{i\in \mathbb{S},x\in\RR^n}\|p^{x}_{i\cdot}(t)-\mu^x\|_{\text{var}}=0.
$$

(iii) The condition \eref{ConA11} is stronger than the classical monotonicity condition:
\begin{eqnarray}
		\left\langle b(x_{1},i)-b(x_{2},i),x_{1}-x_{2}\right\rangle \leq C|x_{1}-x_{2}|^{2}. \label{REM1}
	\end{eqnarray}
If $b$ is uniformly bounded, \eref{ConA11} and \eref{REM1} are equivalent. Furthermore, \eref{ConA11}, \eref{ergodicity} and \eref{LipQ} play a crucial role in studying the monotonicity condition of the averaged coefficient $\bar{b}$, as explained in detail in Lemma \ref{PMbarX}.

(iv) When $m_0$ is finite,  \eref{LipQ} implies \eref{Finte K}. However, in the case $m_0=\infty$,  \eref{Finte K} is necessary. \ref{A4} is used to study the well-posedness of averaged equation (see Eq.\eref{AR21} below).

\vspace{2mm}
(v) ${\color{blue} {\bf H}_{k}}$ is used to study the regularity of the averaged coefficients $\bar{b}$ and $\bar{\sigma}$ (see Remark \ref{HighR} below).
\end{remark}

\subsection{Main results}
Our main results are stated for Poisson equation, averaging principle and central limit theorem in this subsection.

\subsubsection{Poisson equation} In order to state our first result, we present some notations.

Suppose that $F(x,i)=(F^1(x,i), \ldots, F^n(x,i))$ satisfies the "\emph{central condition}":
\begin{eqnarray}
\sum_{i\in\mathbb{S}}F(x,i)\mu^x_i=0,\quad \forall x\in \RR^n\label{CenCon}
\end{eqnarray}
and $Q(x)=(q_{ij}(x))_{i,j\in\mathbb{S}}$ satisfies \ref{A2}. Considering the following Poisson equation on $\mathbb{S}$:
\begin{eqnarray}\label{PE1}
-Q(x)\Phi(x,\cdot)(i)=F(x,i),
\end{eqnarray}
which is equivalent to
\begin{eqnarray*}
-Q(x)\Phi^l(x,\cdot)(i)=F^l(x,i), \quad l=1,2,\ldots,n,
\end{eqnarray*}
where $\Phi(x,i)=(\Phi^1(x,i), \ldots, \Phi^n(x,i))$.

For Poisson equation \eref{PE1}, we have the followings:

\begin{theorem}\label{Poisson}
Suppose that \ref{A2} holds, $F$ satisfies \eref{CenCon} with $\|F(x,\cdot)\|_{\infty}\!<\!\infty$, $\forall x\in\RR^n$. Put
\begin{eqnarray}
\Phi(x,i)=\int_{0}^{\infty}\EE F(x,\alpha_{t}^{x,i})dt. \label{SPE}
\end{eqnarray}
Then $\Phi(x,i)$ is a solution of Poisson equation (\ref{PE1}) and satisfies
\begin{eqnarray}
\|\Phi(x,\cdot)\|_{\infty}\le C\|F(x,\cdot)\|_{\infty},\label{E1}
\end{eqnarray}
where $\{\alpha^{x,i}_t\}_{t\geq 0}$ is the unique $\mathbb{S}$-valued Markov chain generated by generator $Q(x)$ with initial value $\alpha^{x,i}_0=i\in \mathbb{S}$. Moreover, if $Q\in C^1(\RR^n,\RR^{m_0}\otimes\RR^{m_0})$ and $F\in C^1(\RR^n\times \mathbb{S},\RR^n)$, then there exists a constant $C>0$ such that for any $x\in \RR^n$,
\begin{eqnarray}
\|\partial_{x} \Phi(x,\cdot)\|_{\infty}\le C\left[\|F(x,\cdot)\|_{\infty}\|\nabla Q(x)\|_{\ell}+\|\partial_{x}F(x,\cdot)\|_{\infty}\right].\label{E2}
\end{eqnarray}
Furthermore, if $Q\in C^2(\RR^n,\RR^{m_0}\otimes\RR^{m_0})$ and $F\in C^2(\RR^n\times \mathbb{S},\RR^n)$, then there exists a constant $C>0$ such that for any $x\in \RR^n$,
\begin{eqnarray}
&&\|\partial^2_x\Phi(x,\cdot)\|_{\infty}\le  C\left[\|F(x,\cdot)\|_{\infty}\left (\|\nabla Q(x)\|_{\ell}+\|\nabla Q(x)\|^2_{\ell}+\|\nabla^2 Q(x)\|_{\ell}\right)\right.\nonumber\\
&&\quad\quad\quad\quad\quad\quad\quad\quad\quad+\left.\|\partial_{x}F(x,\cdot)\|_{\infty}\|\nabla Q(x)\|_{\ell}+\|\partial^2_x F(x,\cdot)\|_{\infty}\right].\label{E3}
\end{eqnarray}
\end{theorem}

\subsubsection{Averaging principle}
As an application of Th.\ref{Poisson}, we study three types of convergences for averaging principle. Firstly, the strong convergence \eref{StrongA2} is researched. To ensure the validity of the convergence, it is essential to suppose that $\sigma(x,i)$ is independent of $i$, that is $\sigma(x,i)=\sigma(x)$,  otherwise the strong convergence may fails (see a counter-example in Remark \ref{R5.4}). Consequently, the corresponding averaged equation is the following:
\begin{eqnarray}\left\{\begin{array}{l}
\displaystyle d\bar{X}_{t}=\bar{b}( \bar{X}_{t})dt+\sigma(\bar X_t)dW_t,\\
\bar{X}_{0}=x\in \mathbb{R}^n,\end{array}\right. \label{AR1}
\end{eqnarray}
where $\bar{b}(x):=\sum_{j\in\mathbb{S}}b(x,j)\mu^{x}_{j}$.

We have the following strong convergence.

\begin{theorem}\label{main result 1} (\textbf{Strong convergence of $X^{\varepsilon}$ in $C([0,T],\RR^n)$})
Suppose $\sigma(x,i)=\sigma(x)$ and that \ref{A1}-\ref{A3} hold. Then for $x\in\mathbb{R}^n, \alpha\in\mathbb{S}$, $T > 0$ and $p\geq 2$, there exist constants $C_{p,T}>0,k_p>0$ such that
\begin{eqnarray}
\mathbb{E}\left\|X^{\varepsilon}-\bar{X}\right\|^{p}_{C([0,T],\RR^n)}\le C_{p,T}\left(1+|x|^{k_p}\right)\varepsilon^{p/2},\quad \forall \varepsilon\in (0,1], \label{main 1}
\end{eqnarray}
where $\|\cdot\|_{C([0,T],\RR^n)}$ denotes the supremum norm on $[0,T]$, $\bar{X_{t} }$ is the unique solution of Eq.\eref{AR1}.
\end{theorem}

Secondly, the weak convergence \eref{WeakA1} is researched. It is important to note that $\sigma(x,i)$ may depend on $i\in\mathbb{S}$. In this case, the averaged equation is the following:
\begin{eqnarray}\left\{\begin{array}{l}
\displaystyle d\bar{X}_{t}=\bar{b}( \bar{X}_{t})dt+\bar{\sigma}(\bar X_t)dW_t,\\
\bar{X}_{0}=x\in \mathbb{R}^n,\end{array}\right. \label{AR21}
\end{eqnarray}
where $\bar{b}(x):=\sum_{j\in\mathbb{S}}b(x,j)\mu^{x}_{j}$ and
$$\bar{\sigma}(x):=\left[\overline{\sigma\sigma^{\ast}}(x)\right]^{1/2}:=\Big[\sum_{j\in\mathbb{S}}\sigma(x,j)\sigma^{\ast}(x,j)\mu^{x}_{j}\Big]^{1/2},$$
i.e., $\bar{\sigma}(x)\bar{\sigma}(x)=\overline{\sigma\sigma^{\ast}}(x)$.

We present the following two weak convergence results.

\begin{theorem}(\textbf{Weak convergence of $X^{\varepsilon}$ in $C([0,T],\RR^n)$})\label{main result 1.2}
Suppose that \ref{A1}-\ref{A4} hold. Then for $x\in\RR^n,\alpha\in\mathbb{S}$ and $T>0$,  $\{X^{\varepsilon}\}_{\varepsilon>0}$ converges weakly to the solution $\{\bar{X}_t\}_{t\leq T}$ of Eq.\eref{AR21} in $C([0,T];\RR^n)$ as $\varepsilon\to 0$.
\end{theorem}

\begin{theorem}(\textbf{Weak convergence of $X^{\varepsilon}_t$ in $\RR^n$})\label{main result 1.3}
Suppose that \ref{A1}-\ref{A4} and ${\color{blue} \bf H}_{4}$ hold. Then for $x\in\RR^n,\alpha\in\mathbb{S}$, $T>0$ and $\phi\in C^4_p(\RR^n)$, there exists $k>0$ such that
\begin{eqnarray}
\sup_{0\leq t\leq T}|\EE \phi(X^{\varepsilon}_t)-\EE \phi(\bar{X}_t)|\leq C\left(1+|x|^{k}\right)\varepsilon, \quad \forall \varepsilon\in(0,1],\label{AVWO}
\end{eqnarray}
where $\bar{X}$ is the solution of Eq.\eref{AR21}, $C>0$ depends on $T$ and $\phi$.
\end{theorem}

\subsubsection{Central limit theorem}
The convergences considered in this part are based on the strong convergence \eref{main 1}, it requires $\sigma(x,i)=\sigma(x)$. Since the order of $X^{\varepsilon}$ to $\bar{X}$ is $1/2$ in $C([0,T],\RR^n)$, we can consider the following two types of convergences of  $(X^{\varepsilon}-\bar{X})/\sqrt{\varepsilon}$ by using this and Poisson equation technique, which can be interpreted in the context of the central limit theorem.

\begin{theorem} (\textbf{Weak convergence of $(X^{\varepsilon}-\bar{X})/\sqrt{\varepsilon}$ in $C([0,T],\RR^n)$})\label{main result 4}
Suppose $\sigma(x,i)=\sigma(x)$ and \ref{A1}-\ref{A3} and that ${\color{blue}{\bf H}_{2}}$ hold. Then for $x\in\RR^n,\alpha\in\mathbb{S}$ and $T>0$, $Z^{\varepsilon}$ converges weakly to the solution $Z$ of the following equation
\begin{eqnarray}\label{limE}
dZ_t= \!\!\!\!\!\!\!\!&& \nabla\bar{b}(\bar{X}_t)\cdot Z_tdt+\left[\nabla\sigma(\bar{X}_t)\cdot Z_t\right]dW_t
+\Theta(\bar{X}_t)d\tilde{W}_t,~~~Z_0=0
\end{eqnarray}
in $C([0,T];\RR^n)$  as $\varepsilon\to 0$, where
$$\Theta(x):=\left[\overline{F\otimes \Phi+(F\otimes \Phi)^*}(x)\right]^{1/2}:=
\Big[\sum_{j\in\mathbb{S}}((F\otimes \Phi)(x,j)+(F\otimes \Phi)^*(x,j))\mu^{x}_j\Big]^{1/2}$$
 with $F(x,j):=b(x,j)-\bar{b}(x)$, $\Phi$ is defined by \eref{SPE} and
$$(F\otimes\Phi)(x,j):=((F\otimes\Phi)_{kl}(x,j))_{\{1\leq k,l\leq n\}}:=\left(F_k(x,\alpha)\Phi^l(x,j)\right)_{\{1\leq k,l\leq n\}}.$$
$\bar{X}$ is the solution of Eq.\eref{AR1} and $\tilde{W}$ is a $n$-dimensional standard Brownian motion on probability space $(\Omega, \mathcal{F},\mathbb{P})$ and independent of $W$.
\end{theorem}

\begin{theorem}(\textbf{Weak convergence of $(X^{\varepsilon}_t-\bar{X}_t)/\sqrt{\varepsilon}$ in $\RR^n$})\label{main result 5}
Suppose that $\sigma(x,i)=\sigma(x)$ and \ref{A1}-\ref{A3} and ${\color{blue}{\bf H}_{3}}$ hold. Then for $x\in\RR^n,\alpha\in\mathbb{S}$,$T>0$, and $\phi\in C^4_p(\RR^n)$, there exists $k>0$ such that
\begin{eqnarray}
\sup_{0\leq t\leq T}|\EE \phi(Z^{\varepsilon}_t)-\EE \phi(Z_t)|\leq C\left(1+|x|^{k}\right)\varepsilon^{1/2},\quad \varepsilon\in(0,1],\label{CLTWO}
\end{eqnarray}
where $Z$ is the unique solution of \eref{limE}, constant $C$ depends on $T$ and $\phi$.
\end{theorem}

\subsection{Examples}

In this subsection, we give several concrete examples to illustrate our convergence orders are optimal. For simplicity, we only consider the 1D case.

\begin{example}\label{Ex2.9}
Considering:
\begin{eqnarray*}
dX^{\varepsilon}_t=b(\alpha^{\varepsilon}_t)dt+dW_t,\quad X^{\varepsilon}_0=x\in\mathbb{R},
\end{eqnarray*}
$\alpha^{\varepsilon}$ is a $\{1,2\}$-valued Markov chain with the  generator
$$
Q^{\varepsilon}=\varepsilon^{-1}Q=\varepsilon^{-1}\left(
                       \begin{array}{cc}
                         -1 & 1 \\
                         1 & -1 \\
                       \end{array}
                     \right).
$$
For $\alpha^{\varepsilon}_0=1$, $b(1)=1$ and $b(2)=-1$, It is easy to check that $Q$ admits a unique invariant probability measure $\mu=\left(1/2,1/2\right)$. The transition probability $P^\varepsilon_t$ of $\alpha^{\varepsilon}$ is the following:
$$
P^{\varepsilon}_t=\frac12\left(
  \begin{array}{cc}
    1+e^{-2t/\varepsilon} & 1-e^{-2t/\varepsilon}  \\
   1-e^{-2t/\varepsilon}  &  1+e^{-2t/\varepsilon} \\
  \end{array}
\right).
$$
Since the averaged coefficient $\bar{b}(x)=0$, the corresponding averaged equation is given by
$$\bar{X}_t=x+W_t.$$
Refer to \cite[Example 3.4]{LLSW2023}, we have
\begin{eqnarray*}
\EE\left|X^{\varepsilon}_t-\bar{X}_t\right|^2=\varepsilon t+\frac{\varepsilon^2}{2}\left(e^{-\frac{2t}{\varepsilon}}-1\right),
\end{eqnarray*}
which implies that the strong convergence order $1/2$ in \eref{main 1} is optimal.

Furthermore, we have
$$
Z^{\varepsilon}_t=\frac{X^{\varepsilon}_t-\bar{X}_t}{\sqrt{\varepsilon}}=\frac{1}{\sqrt{\varepsilon}}\int^t_0 b(\alpha^{\varepsilon}_s) ds.
$$
It is easy to check that the corresponding limiting  process $Z_t=\tilde{W}_t$ via the limiting equation \eref{limE}, where $\tilde{W}$ is a $1$D standard Brownian motion. Taking $\phi(x)=x$, we have
\begin{eqnarray*}
|\EE\phi(Z^{\varepsilon}_t)-\EE\phi(Z_t)|=\!\!\!\!\!\!\!\!&&|\EE Z^{\varepsilon}_t-\EE Z_t|=\Big|\frac{1}{\sqrt{\varepsilon}}\int^t_0 \EE b(\alpha^{\varepsilon}_s) ds\Big|\\
=\!\!\!\!\!\!\!\!&&\Big|\frac{1}{\sqrt{\varepsilon}}\int^t_0 \left[\PP(\alpha^{\varepsilon}_s=1)-\PP(\alpha^{\varepsilon}_s=2) \right]ds\Big|\\
=\!\!\!\!\!\!\!\!&&\Big|\frac{1}{\sqrt{\varepsilon}}\int^t_0 \left[p^{\varepsilon}_{11}(s)-p^{\varepsilon}_{12}(s) \right]ds\Big|\\
=\!\!\!\!\!\!\!\!&&\frac{1}{\sqrt{\varepsilon}}\int^t_0 e^{-2s/\varepsilon}ds=\frac{\sqrt{\varepsilon}}{2}\left(1-e^{-\frac{2t}{\varepsilon}}\right),
\end{eqnarray*}
which implies that the weak convergence order $1/2$ in \eref{CLTWO} is optimal.
\end{example}

\begin{remark}
Choosing $\phi(x)=x^2$, we have
\begin{eqnarray*}
|\EE\phi(Z^{\varepsilon}_t)-\EE\phi(Z_t)|=\left|\EE (Z^{\varepsilon}_t)^2-\EE Z_t^2\right|=\Big|\EE\Big(\frac{X^{\varepsilon}_t-\bar{X}_t}{\sqrt{\varepsilon}}\Big)^2-t\Big|=\frac{\varepsilon}{2}\left(1-e^{-\frac{2t}{\varepsilon}}\right).
\end{eqnarray*}
This implies the weak convergence order is $1$ which is bigger than $1/2$. Hence, the weak order may vary depending on the class $\mathbb{C}$ under consideration.
\end{remark}

\begin{example}\label{Ex2.11}
Considering:
\begin{eqnarray}
dX^{\varepsilon}_t=b(\alpha^{\varepsilon}_t)dt+\sigma(\alpha^{\varepsilon}_t)dW_t,\quad X^{\varepsilon}_0=x\in\mathbb{R},\label{Ex2}
\end{eqnarray}
where $\alpha^{\varepsilon}$ is the same Markov chain as in Example \ref{Ex2.9}, and assume $\sigma(1)=b(1)=1$, $\sigma(2)=b(2)=-1$. Then the averaged coefficients $\bar{b}(x)=0$ and $\bar{\sigma}(x)=1$, the corresponding averaged equation is given by
$$\bar{X}_t=x+W_t.$$
Taking $\phi(x)=x$, we have
\begin{eqnarray*}
\left|\EE\phi(X^{\varepsilon}_t)-\EE\phi(\bar{X}_t)\right|=\left|\EE X^{\varepsilon}_t-\EE \bar{X}_t\right|=\Big|\int^t_0 \EE b(\alpha^{\varepsilon}_s) ds\Big|=\frac{\varepsilon}{2}\left(1-e^{-\frac{2t}{\varepsilon}}\right).
\end{eqnarray*}
This means that the weak convergence order $1$ in \eref{AVWO} is optimal.
\end{example}

\begin{remark}\label{R5.4}
Let $b\equiv0$ in Eq.\eref{Ex2}, then
\begin{eqnarray*}
\EE\left|X^{\varepsilon}_t-\bar{X}_t\right|^2=\!\!\!\!\!\!\!\!&&\EE\Big|\int^t_0 \left[\sigma(\alpha^{\varepsilon}_s)-1\right] dW_s\Big|^2\\
=\!\!\!\!\!\!\!\!&&\EE\int^t_0 |\sigma(\alpha^{\varepsilon}_s)-1|^2ds=\int^t_0 2\left(1-e^{-\frac{2s}{\varepsilon}}\right)ds=2t+\varepsilon\left(e^{-\frac{2t}{\varepsilon}}-1\right).
\end{eqnarray*}
This shows that there is no strong convergence in this case.
\end{remark}

\section{Poisson equation associated to a Markov chain}

In this section, we focus on the study of Poisson equation \eref{PE1}. To do this, we will now investigate the differentiability of $P^{x}_t$ with respect to $x$, which is crucial in examining the regularity of the Poisson equation's solution.

\begin{proposition}\label{DQ}
Suppose \ref{A2} and $Q\in C^{2}(\RR^n,\RR^{m_0}\otimes\RR^{m_0})$. Then $P^x_t$ is twice differentiable with respective to $x$, more precisely, for any $f$ on $\mathbb{S}$ with $\|f\|_{\infty}\leq 1$ and $k,l\in\{1,2\ldots,n\}$,
\begin{equation}
\sup_{t\geq 0,i\in\mathbb{S}}|\partial_{x_k}P^{x}_tf(i)|\leq C\|\partial_{x_k}Q(x)\|_{\ell},\label{FDP}
\end{equation}
\begin{equation}
\sup_{t\geq 0,i\in\mathbb{S}}|\partial_{x_k}\partial_{x_l}P^{x}_tf(i)|\leq C\left(\|\partial_{x_k}Q(x)\|_{\ell}\|\partial_{x_l}Q(x)\|_{\ell}+\|\partial_{x_k}\partial_{ x_l}Q(x)\|_{\ell}\right).\label{FDP2}
\end{equation}
Moreover, there exists $C>0$ such that for any $x\in\RR^n$, $i,i'\in\mathbb{S}$ and $t\geq 0$,
\begin{equation}
|\partial_{x_k}P^x_tf(i)-\partial_{x_k}P^x_tf(i')|\leq C  \|\partial_{x_k}Q(x)\|_{\ell}e^{-\lambda t}t ,\label{SDP1}
\end{equation}
\begin{eqnarray}
|\partial_{x_k}\partial_{x_l}P^x_tf(i)\!-\!\partial_{x_k}\partial_{x_l}P^x_tf(i')|
\!\le\! C\left (\|\partial_{x_k}Q(x)\|_{\ell} \|\partial_{x_l}Q(x)\|_{\ell}\!+\!\|\partial_{x_k}\partial_{ x_l}Q(x)\|_{\ell}\right) e^{-\lambda t}(1\!+\!t^2). \label{SDP2}
\end{eqnarray}
\end{proposition}

\begin{proof}
By the integration by parts formula for the semigroups $P^x_t$ and $P^y_t$ (see \cite[Theorem 13.40]{C2004}), for any $f$ on $\mathbb{S}$ with $\|f\|_{\infty}\leq 1$, $x,y\in\RR^n$ and $t\geq 0$, we have
\begin{eqnarray*}
P^y_tf-P^x_tf=\int^t_0 P^{y}_s[Q(y)-Q(x)]P^{x}_{t-s}fds.
\end{eqnarray*}
$q_{ij}(x)\in C^2_{b}(\RR^n)$ yields that, for any $k\in\{1,2\ldots,n\}$ and $i\in\mathbb{S}$,
\begin{eqnarray}
\partial_{x_k}P^{x}_tf(i)=\int^t_0 P^{x}_s\left[\partial_{x_k}Q(x)\right]P^{x}_{t-s}f(i) ds.\label{Ppartial x}
\end{eqnarray}
By the condition \eref{ergodicity}, we can obtain
\begin{eqnarray*}
|\partial_{x_k}P^x_tf(i)|=\!\!\!\!\!\!&&\Big|\int^t_0 P^{x}_s\left[\partial_{x_k}Q(x)\right]\left[P^{x}_{t-s}(f-\mu^x(f) )\right](i)ds\Big|\\
\leq\!\!\!\!\!\!&&\int^t_0 \sup_{j\in\mathbb{S}}|\partial_{x_k}Q(x)\left[P^{x}_{t-s}(f-\mu^x(f))\right](j) |ds\\
\leq\!\!\!\!\!\!&&\|\partial_{x_k}Q(x)\|_{\ell}\int^t_0 \sup_{j\in\mathbb{S}}|P^{x}_{t-s}f(j)-\mu^x(f) |ds\\
\leq\!\!\!\!\!\!&&\|\partial_{x_k}Q(x)\|_{\ell}\int^t_0 \sup_{j\in\mathbb{S}}\|p^{x}_{j\cdot}(t-s)-\mu^x\|_{\text{var}}ds\\
\leq\!\!\!\!\!\!&&C \|\partial_{x_k}Q(x)\|_{\ell}\int^t_0e^{-\lambda (t-s)} ds\leq \frac{C}{\lambda}\|\partial_{x_k}Q(x)\|_{\ell}
\end{eqnarray*}
and for any $i,i'\in\mathbb{S}$,
\begin{eqnarray}
&&|\partial_{x_k}P^x_tf(i)-\partial_{x_k}P^x_tf(i')|\nonumber\\
=\!\!\!\!\!\!&&\Big|\int^t_0 \big\{P^{x}_s[\partial_{x_k}Q(x)][P^{x}_{t-s}(f-\mu^x(f)) ](i)-P^{x}_s\left[\partial_{x_k}Q(x)\right][P^{x}_{t-s}(f-\mu^x(f))](i')\big\}ds\Big|\nonumber\\
\leq\!\!\!\!\!\!&&\int^t_0 \|p^x_{i\cdot}(s)-p^x_{i'\cdot}(s)\|_{\text{var}}\sup_{j\in \mathbb{S}}|\partial_{x_k}Q(x)\left[P^{x}_{t-s}(f-\mu^x(f))\right](j) |ds\nonumber\\
\leq\!\!\!\!\!\!&&\|\partial_{x_k}Q(x)\|_{\ell}\int^t_0 \|p^x_{i\cdot}(s)-p^x_{i'\cdot}(s)\|_{\text{var}}\sup_{j\in\mathbb{S}}\|p^{x}_{j\cdot}(t-s)-\mu^x \|_{\text{var}}ds\nonumber\\
\leq\!\!\!\!\!\!&&C\|\partial_{x_k}Q(x)\|_{\ell}\int^t_0 e^{-\lambda s}e^{-\lambda (t-s)} ds\leq C\|\partial_{x_k}Q(x)\|_{\ell} e^{-\lambda t}t.\label{DPpartialx}
\end{eqnarray}

Similarly, by \eref{Ppartial x} we have for any $k,l\in\{1,2\ldots,n\}$,
\begin{eqnarray*}
\partial_{x_k} \partial_{ x_l}P^{x}_tf(i)=\!\!\!\!\!\!&&\int^t_0 \partial_{x_l}P^{x}_s\left[\partial_{x_k}Q(x)\right]P^{x}_{t-s}f(i)ds+\int^t_0 P^{x}_s\left[\partial_{x_k}\partial_{ x_l}Q(x)\right]P^{x}_{t-s}f(i)ds\nonumber\\
&&+\int^t_0 P^{x}_s\left[\partial_{x_k}Q(x)\right] \partial_{x_l}P^{x}_{t-s}f(i) ds.\label{Ppartial xx}
\end{eqnarray*}
Thus using \eref{ergodicity} and \eref{DPpartialx}, we can obtain
\begin{eqnarray*}
|\partial_{x_k}\partial_{ x_l }P^x_tf(i)|\leq\!\!\!\!\!\!&&C \|\partial_{x_l}Q(x)\|_{\ell}\int^t_0 \sup_{j\in \mathbb{S}}\left|\left[\partial_{x_k}Q(x)\right]\left[P^{x}_{t-s}(f -\mu^x(f))\right](j)\right|ds\\
&&+C\int^t_0 \sup_{j\in\mathbb{S}}\left|\left[\partial_{x_k}\partial_{ x_l}Q(x)\right]\left[P^{x}_{t-s}(f-\mu^x(f))\right](j)\right| ds\\
&&+C\int^t_0 \sup_{j\in\mathbb{S}}\left|\left[\partial_{x_k}Q(x)\right] \left[\partial_{x_l}P^{x}_{t-s}f-\partial_{x_l}P^{x}_{t-s}f(i))\right](j)\right|ds\\
\leq\!\!\!\!\!\!&&C\left (\|\partial_{x_l}Q(x)\|_{\ell} \|\partial_{x_k}Q(x)\|_{\ell}+\|\partial_{x_k}\partial_{ x_l}Q(x)\|_{\ell}\right)\int^t_0 \!\!\sup_{j\in \mathbb{S}}\|p^{x}_{j\cdot}(t-s) -\mu^x\|_{\text{var}}ds\\
&&+C \|\partial_{x_k}Q(x)\|_{\ell}\int^t_0 \sup_{j\in\mathbb{S}}|\partial_{x_l}P^{x}_{t-s}f(j) -\partial_{x_l}P^{x}_{t-s}f(i)|ds\\
\leq\!\!\!\!\!\!&&C\left (\|\partial_{x_l}Q(x)\|_{\ell} \|\partial_{x_k}Q(x)\|_{\ell}+\|\partial_{x_k}\partial_{ x_l}Q(x)\|_{\ell}\right)\int^t_0 e^{-\lambda s} (1+ s)ds\\
\leq \!\!\!\!\!\!&& \frac{2C}{\lambda}\left (\|\partial_{x_l}Q(x)\|_{\ell} \|\partial_{x_k}Q(x)\|_{\ell}+\|\partial_{x_k}\partial_{ x_l}Q(x)\|_{\ell}\right)
\end{eqnarray*}
and for any $i,i'\in\mathbb{S}$,
\begin{eqnarray*}
&&|\partial_{x_k} \partial_{ x_l}P^x_tf(i)-\partial_{x_k} \partial_{ x_l}P^x_tf(i')|\nonumber\\
\leq\!\!\!\!\!\!&&C\|\partial_{x_l}Q(x)\|_{\ell}\int^t_0 s e^{-\lambda s}\sup_{j\in\mathbb{S}}\big|\left[\partial_{x_k}Q(x)\right]\big[P^{x}_{t-s}(f -\mu^x(f))\big](j)\big|ds\\
&&+C\int^t_0 \|p^x_{i\cdot}(s)-p^x_{i'\cdot}(s)\|_{\text{var}}\sup_{j\in\mathbb{S}}\left|\left[\partial_{x_k}\partial_{ x_l}Q(x)\right]\left[P^{x}_{t-s}(f -\mu^x(f))\right](j)\right| ds\\
&&+C\int^t_0 \|p^x_{i\cdot}(s)-p^x_{i'\cdot}(s)\|_{\text{var}}\sup_{j\in \mathbb{S}}\left|\left[\partial_{x_k}Q(x)\right]\left[\partial_{x_l}P^{x}_{t-s}f -\partial_{x_l}P^{x}_{t-s}f(i))\right](j)\right| ds\\
\leq\!\!\!\!\!\!&&C\|\partial_{x_l}Q(x)\|_{\ell} \|\partial_{x_k}Q(x)\|_{\ell}\int^t_0 s e^{-\lambda s}\sup_{j\in \mathbb{S}}\|p_{j\cdot}^{x}(t-s)-\mu^x\|_{\text{var}}ds\\
&&+C\|\partial_{x_k}\partial_{ x_l}Q(x)\|_{\ell}\int^t_0 \|p^x_{i\cdot}(s)-p^x_{i'\cdot}(s)\|_{\text{var}}\sup_{j\in \mathbb{S}}\|p_{j\cdot}^{x}(t-s)-\mu^x\|_{\text{var}}ds\\
&&+C\|\partial_{x_k}Q(x)\|_{\ell}\int^t_0 \|p^x_{i\cdot}(s)-p^x_{i'\cdot}(s)\|_{\text{var}}\sup_{j\in \mathbb{S}}\left|\partial_{x_l}P^{x}_{t-s}f(j) -\partial_{x_l}P^{x}_{t-s}f(i)\right| ds\\
\leq\!\!\!\!\!\!&&C\left (\|\partial_{x_l}Q(x)\|_{\ell} \|\partial_{x_k}Q(x)\|_{\ell}+\|\partial_{x_k}\partial_{ x_l}Q(x)\|_{\ell}\right)\int^t_0 (1+s)e^{-\lambda s}e^{-\lambda(t-s)} \left[1+(t-s)\right] ds\\
\leq\!\!\!\!\!\!&&C\left (\|\partial_{x_l}Q(x)\|_{\ell} \|\partial_{x_k}Q(x)\|_{\ell}+\|\partial_{x_k}\partial_{ x_l}Q(x)\|_{\ell}\right) e^{-\lambda t}\left(1+t^2\right).
\end{eqnarray*}
The proof is complete.
\end{proof}

\begin{remark}
Suppose that \ref{A2} holds, $Q\in C^m(\RR^n,\RR^{m_0}\otimes\RR^{m_0})$ with $\sum^m_{k=1}\|\nabla^k Q(x)\|_{\ell}\leq C$ for some $m\in\mathbb{N}_{+}$. Then $P^x_t$ is $m$-th differentiable with respective to $x$, more precisely, there exists $C>0$ such that for $f$ on $\mathbb{S}$ with $\|f\|_{\infty}\leq 1$,
\begin{eqnarray}
\sup_{t\geq 0,i\in\mathbb{S}}\|\partial_x^j P^x_tf(i)\|\leq C,\quad j=1,\ldots,m,\label{Partialb}
\end{eqnarray}
Similarly, if $Q\in C^m_p(\RR^n,\RR^{m_0}\otimes\RR^{m_0})$ for some $m\in\mathbb{N}_{+}$. Then there exist $C,k>0$ such that
\begin{eqnarray}
\sup_{t\geq 0,i\in\mathbb{S}}\|\partial_x^j P^x_tf(i)\|\leq C(1+|x|^k),\quad j=1,\ldots,m.\label{PartialP}
\end{eqnarray}
\end{remark}

Now, we give the proof of our first main result on Poisson equation.

\vspace{-2mm}
~~\\
\textbf{Proof of Theorem \ref{Poisson}}: The detailed proofs are divided into three steps.

\textbf{Step 1:} We shall prove that $\Phi(x,i)$ is a solution of the Poisson equation (\ref{PE1}).
Put $P^x_t$ is the corresponding transition probability matrix of $Q(x)$. It is sufficient to show that for any $l=1,2,\ldots,n$,
$$
\lim_{s\rightarrow 0}\frac{P^x_s\Phi^{l}(x,\cdot)(i)-\Phi^l(x,i)}{s}=-F^l(x,i).
$$
The Chapman-Kolmogorov equation implies that
\begin{eqnarray*}
P^x_s\Phi^{l}(x,\cdot)(i)-\Phi^l(x,i)=\!\!\!\!\!\!\!\!&&\sum_{j\in\mathbb{S}}p^x_{ij}(s)\Phi^{l}(x,j)-\Phi^{l}(x,i)\\
=\!\!\!\!\!\!\!\!&&\sum_{j\in\mathbb{S}}p^x_{ij}(s)\int^{\infty}_0 \sum_{k\in\mathbb{S}}F^l(x, k)p^x_{jk}(t)dt-\sum_{k\in\mathbb{S}}\int^{\infty}_0 F^l(x, k)p^x_{ik}(t)dt\\
=\!\!\!\!\!\!\!\!&&\sum_{k\in\mathbb{S}}\int^{\infty}_0 F^l(x, k)p^x_{ik}(t+s)dt-\sum_{k\in\mathbb{S}}\int^{\infty}_0 F^l(x, k)p^x_{ik}(t)dt\\
=\!\!\!\!\!\!\!\!&&\sum_{k\in\mathbb{S}}\int^{\infty}_s F^l(x, k)p^x_{ik}(t)dt-\sum_{k\in\mathbb{S}}\int^{\infty}_0 F^l(x, k)p^x_{ik}(t)dt\\
=\!\!\!\!\!\!\!\!&&-\sum_{k\in\mathbb{S}}\int^{s}_0 F^l(x, k)p^x_{ik}(t)dt.
\end{eqnarray*}
This yields
\begin{eqnarray*}
\lim_{s\rightarrow 0}\frac{P^x_s\Phi^{l}(x,\cdot)(i)-\Phi^l(x,i)}{s}=-\lim_{s\rightarrow 0}\sum_{k\in\mathbb{S}}F^l(x, k)p^x_{ik}(s)=-F^l(x,i).
\end{eqnarray*}

By \eref{CenCon} and \eref{ergodicity}, we have
\begin{eqnarray*}
|\Phi(x,i)|\leq\!\!\!\!\!\!\!\!&&\int_{0}^{\infty}|\EE F(x,\alpha_{t}^{x,i})|dt\\
=\!\!\!\!\!\!\!\!&&\int_{0}^{\infty}\left|P^x_t F(x,\cdot)(i)-\mu^x(F(x,\cdot))\right|dt\\
\leq\!\!\!\!\!\!\!\!&&\int_{0}^{\infty}\|F(x,\cdot)\|_{\infty}\|p^x_{i\cdot}(t)-\mu^x\|_{\rm{var}}dt\\
\leq\!\!\!\!\!\!\!\!&&C\|F(x,\cdot)\|_{\infty}\int_{0}^{\infty}e^{-\lambda t}dt\\
\leq\!\!\!\!\!\!\!\!&&\frac{C}{\lambda}\|F(x,\cdot)\|_{\infty}.
\end{eqnarray*}

Put $\hat F(x, i, t):=\EE F(x, \alpha^{x,i}_t)$, $\tilde F_{t_0}(x, i, t):=\hat F(x,i, t)-\hat F(x, i, t+t_0)$. The ergodicity condition \eref{ergodicity} implies
$$\lim_{t_0\rightarrow +\infty} \tilde F_{t_0}(x, i, t)=\EE F(x,\alpha^{x,i}_t).$$

In order to prove \eref{E2} and \eref{E3}, it is sufficient to prove that there exists $C>0$ such that for any $t_0>0$, $ t> 0$, $x\in\RR^{n}$, $k,l\in\{1,2,\ldots,n\}$ and $i\in\mathbb{S}$,
\begin{eqnarray}
| \partial _{x_k}\tilde{F}_{t_0}(x,i,t)|\le C\left[\|F(x,\cdot)\|_{\infty}\|\partial_{x}Q(x)\|_{\ell}+\|\partial_{x}F(x,\cdot)\|_{\infty}\right]e^{-\lambda t}t, \label{E2.}
\end{eqnarray}
\begin{eqnarray}
\!\!\!\!\!\!| \partial_{x_k}\partial_{x_l}\tilde{F}_{t_0}(x,i,t)|\le \!\!\!\!\!\!\!\!&& C\left[\|F(x,\cdot)\|_{\infty}\left (\|\partial_{x}Q(x)\|_{\ell}+\|\partial_{x}Q(x)\|^2_{\ell}+\|\partial^2_{x}Q(x)\|_{\ell}\right)\right.\nonumber\\
&&\quad\left.+\|\partial_{x}F(x,\cdot)\|_{\infty}\|\partial_{x}Q(x)\|_{\ell}+\|\partial^2_x F(x,\cdot)\|_{\infty}\right]e^{-\lambda t}(1\vee t^2), \label{E3.}
\end{eqnarray}
which will be proved in \text{Step 2} and \text{Step 3}, respectively.
\vspace{0.3cm}

\textbf{Step 2:} {\it The proof of \eref{E2.}}. By the Markov property,
\begin{eqnarray*}
\tilde F_{t_0}(x,i, t)=\!\!\!\!\!\!\!\!&& \hat F(x, i, t)\!-\!\EE F(x, \alpha^{x,i}_{t+t_0})\!=\! \hat F(x,i, t)\!-\!\EE \!\left[\EE\!\left(F(x,\alpha^{x,i}_{t+t_0})|\mathcal{F}_{t_0}\right)\right]\\
=\!\!\!\!\!\!\!\!&&\hat F(x, i, t)\!-\!\EE \hat F(x, \alpha^{x,i}_{t_0},t),
\end{eqnarray*}
which implies that, for any $k\in \{1,2,\ldots,n\}$,
\begin{eqnarray}
\partial_{x_k}\tilde F_{t_0}(x,i, t)= \partial_{x_k} \hat F(x, i, t)-\EE \partial_{x_k}\hat {F}(x, \alpha^{x,i}_{t_0},t)- \left\{\partial_{x_k}\EE \left[\hat F(z,\alpha^{x,i}_{t_0},t) \right]\right\}|_{z=x}.\label{F4.12}
\end{eqnarray}

For any $i,j\in\mathbb{S}$,
$$\partial_{x_k}\hat F(x,i, t)=\partial_{x_k}\EE F(x,\alpha^{x,i}_t)=\EE\partial_{x_k}F(x,\alpha^{x,i}_t)+\left\{\partial_{x_k}\EE F(z,\alpha^{x,i}_t)\right\}|_{z=x}$$
implies
\begin{eqnarray*}
&&|\partial_{x_k} \hat{F}(x,i,t)-\partial_{x_k} \hat{F}(x,j,t)|\\
\leq\!\!\!\!\!\!\!\!&&|\EE\partial_{x_k}F(x,\alpha^{x,i}_t)-\EE\partial_{x_k}F(x,\alpha^{x,j}_t)|+\left|\partial_{x_k}\EE F(z,\alpha^{x,i}_t)|_{z=x}-\partial_{x_k}\EE F(z,\alpha^{x,j}_t)|_{z=x}\right|.
\end{eqnarray*}
Then by \eref{SDP1} and \eref{ergodicity}, we have
\begin{eqnarray}
&&|\partial_{x_k} \hat{F}(x,i,t)-\partial_{x_k} \hat{F}(x,j,t)|\nonumber\\
\leq\!\!\!\!\!\!\!\!&&\|\partial_{x}F(x,\cdot)\|_{\infty} \|p^{x}_{i\cdot}(t)-p^{x}_{j\cdot}(t)\|_{\text{var}}+\left|\partial_{x_k}P^x_t F(z,\cdot)(i)|_{z=x}-\partial_{x_k}P^x_tF(z,\cdot)(j)|_{z=x}\right|\nonumber\\
\leq\!\!\!\!\!\!\!\!&&C\|\partial_{x}F(x,\cdot)\|_{\infty}e^{-\lambda t}+C\|F(x,\cdot)\|_{\infty}\|\partial_{x_k}Q(x)\|_{\ell} e^{-\lambda t}t \nonumber\\
\leq\!\!\!\!\!\!\!\!&& C\!\left[\|F(x,\cdot)\|_{\infty}\|\partial_{x_k}Q(x)\|_{\ell}+\|\partial_{x}F(x,\cdot)\|_{\infty}\!\right]e^{-\lambda t}(1\vee t).\label{phatF}
\end{eqnarray}
Hence, we have
\begin{eqnarray}
|\partial_{x_k} \hat F(x, i, t)\!-\!\EE \partial_{x_k}\hat {F}(x, \alpha^{x,i}_{t_0},t)|\!=\!C\left[\|F(x,\cdot)\|_{\infty}\|\partial_{x_k}Q(x)\|_{\ell}\!+\!\|\partial_{x}F(x,\cdot)\|_{\infty}\!\right]e^{-\lambda t}(1\!\vee\! t).\nonumber\\\label{I11}
\end{eqnarray}

By the central condition \eref{CenCon}, we get
\begin{eqnarray}
|\hat F(x,l,t)|\!=\!|P^x_t F(x,\cdot)(l)\!-\!\mu^x(F(x,\cdot)|\!\le\! \|F(x,\cdot)\|_{\infty}\|p^x_{l\cdot}(t)\!-\!\mu^{x}\|_{\text{var}}\!\le\! C\|F(x,\cdot)\|_{\infty}e^{-\lambda t}.\nonumber\\ \label{hatF}
\end{eqnarray}
\eref{FDP} and \eref{hatF} imply that, for any $t_0\geq 0$,
\begin{eqnarray}
\left|\left\{\partial_{x_k}\EE \left[\hat F(z,\alpha^{x,i}_{t_0},t) \right]\right\}|_{z=x}\right|=\!\!\!\!\!\!\!\!&&|\partial_{x_k}P^x(t_0)\hat F(x,\cdot,t)(i)|\leq C\|\hat F(x,\cdot,t)\|_{\infty}\|\partial_{x_k}Q(x)\|_{\ell}\nonumber\\
\leq\!\!\!\!\!\!\!\!&&C\|F(x,\cdot)\|_{\infty}\|\partial_{x_k}Q(x)\|_{\ell}e^{-\lambda t}.\label{I12}
\end{eqnarray}
\eref{F4.12}, \eref{I11} and \eref{I12} yield that \eref{E2.} holds.

\textbf{Step 3:}  { \it The proof of \eref{E3.}}. \eref{F4.12} and the chain rule yields that, for $k,l\in \{1,2,\ldots,n\}$,
\begin{eqnarray*}
\partial_{x_k}\partial_{ x_l}\tilde{F}_{t_0}(x,i,t)\!\!\!\!\!\!\!\!&&=\left[\partial_{x_k}\partial_ {x_l}\hat{F}(x,i,t)-\EE\partial_{x_k}\partial_{ x_l}\hat{F}(x,\alpha_{t_0}^{x,i},t)\right]\\
&&~~~~-\left\{\partial_{x_l}\EE[\partial_{x_k}\hat{F}(z,\alpha_{t_0}^{x,i},t)]\right\}|_{z=x}-\partial_{x_l}\left\{\partial_{x_k}\EE \left[\hat F(z,\alpha^{x,i}_{t_0},t) \right]|_{z=x}\right\}\\
\!\!\!\!\!\!\!\!&&=:J_1+J_2+J_3.
\end{eqnarray*}

(i) For $J_1$,  note that
$$\partial_{x_k}\partial_{x_l}\hat{F}(x,i,t)\!=\!\EE\partial_{x_k}\partial_{ x_l} F(x,\alpha_{t}^{x,i})\!+\!\left\{\partial_{x_l}\EE\left[\partial_{x_k}F(z,\alpha_{t}^{x,i})\right]\right\}|_{z=x}\!+\!\partial_{x_l}\left[\partial_{x_k}\EE F(z,\alpha_{t}^{x,i})|_{z=x}\right],$$
which implies that, for any $i,j\in \mathbb{S}$,
\begin{eqnarray*}
&&|\partial_{x_k}\partial_{ x_l}\hat{F}(x,i,t)-\partial_{x_k}\partial_{ x_l}\hat{F}(x,j,t)|\nonumber\\
\leq \!\!\!\!\!\!\!&&|\EE\partial_{x_k}\partial_{ x_l}F(x,\alpha_{t}^{x,i})-\EE\partial_{x_k}\partial_{ x_l}F(x,\alpha_{t}^{x,j})|\nonumber\\
&&+\left|\left\{\partial_{x_l}\EE\left[\partial_{x_k}F(z,\alpha_{t}^{x,i})\right]\right\}|_{z=x}-\left\{\partial_{x_l}\EE\left[\partial_{x_k}F(z,\alpha_{t}^{x,j})\right]\right\}|_{z=x}\right|\nonumber\\
&&+\left|\partial_{x_l}\left[\partial_{x_k}\EE F(z,\alpha_{t}^{x,i})|_{z=x}\right]-\partial_{x_l}\left[\partial_{x_k}\EE F(z,\alpha_{t}^{x,j})|_{z=x}\right]\right|\nonumber\\
=:\!\!\!\!\!\!\!\!\!&&J_{11}+J_{12}+J_{13} .
\end{eqnarray*}

The ergodicity condition \eref{ergodicity} yields that
\begin{eqnarray}
J_{11}\leq\|\partial_{x_k}\partial_{ x_l} F(x,\cdot)\|_{\infty} \|p^{x}_{i\cdot}(t)-p^{x}_{j\cdot}(t)\|_{\text{var}}\le C\|\partial^2_{x}F(x,\cdot)\|_{\infty} e^{-\lambda t}.\label{J11}
\end{eqnarray}

Using \eref{SDP1} we have
\begin{eqnarray}
J_{12}\leq\|\partial_{x_k} F(x,\cdot)\|_{\infty}\|\partial_{x_k}Q(x)\|_{\ell}e^{-\lambda t}t\le C\|\partial_{x}F(x,\cdot)\|_{\infty}\|\partial_{x_k}Q(x)\|_{\ell} e^{-\lambda t}t.\label{J12}
\end{eqnarray}

Applying \eref{SDP1} and \eref{SDP2}, we obtain
\begin{eqnarray}
J_{13}\leq\!\!\!\!\!\!\!\!&&\left|\partial_{x_k}P^x_t\partial_{x_l}F(x,\cdot)(i)\!-\!\partial_{x_k}P^x_t\partial_{x_l}F(x,\cdot)(j)\right|\!+\!\left|\partial_{x_l}\partial_{x_k}P^x_t F(x,\cdot)(i)\!-\!\partial_{x_l}\partial_{x_k}P^x_t F(x,\cdot)(j)\right|\nonumber\\
\leq\!\!\!\!\!\!\!&&C\|\partial_{x_l}F(x,\cdot)\|_{\infty}\|\partial_{x_k}Q(x)\|_{\ell} e^{-\lambda t}t\nonumber\\
&&+C\|F(x,\cdot)\|_{\infty}\left [\|\partial_{x_k}Q(x)\|_{\ell} \|\partial_{x_l}Q(x)\|_{\ell}+\|\partial_{x_k}\partial_{ x_l}Q(x)\|_{\ell}\right]e^{-\lambda t}(1\vee t^2).\label{J13}
\end{eqnarray}

Thus by \eref{J11}-\eref{J13}, we obtain
\begin{eqnarray}
|J_1|\leq\!\!\!\!\!\!\!&&C\left[\|F(x,\cdot)\|_{\infty}\left (\|\partial_{x_k}Q(x)\|_{\ell} \|\partial_{x_l}Q(x)\|_{\ell}+\|\partial_{x_k}\partial_{ x_l}Q(x)\|_{\ell}\right)\right.\nonumber\\
&&\quad+\|\partial_{x}F(x,\cdot)\|_{\infty}(1+\|\partial_{x_k}Q(x)\|_{\ell} )+\|\partial^2_x F(x,\cdot)\|_{\infty}]e^{-\lambda t}(1\vee t^2).\label{J1.}
\end{eqnarray}

(ii) For $J_2$,  for any coupling process $(\tilde{\alpha}_{t_0}^{x,i}, \tilde{\alpha}_{t_0}^{y,i})$ of $(\alpha_{t_0}^{x,i}, \alpha_{t_0}^{y,i})$,  using \eref{phatF}, we have that, for any $x,y\in\RR^n$,
\begin{eqnarray*}
&&\left|\EE[\partial_{x_k}\hat{F}(z,\alpha_{t_0}^{x,i},t)]-\EE[\partial_{x_k}\hat{F}(z,\alpha_{t_0}^{y,i},t)]\right|\\
=\!\!\!\!\!\!\!&&\left|\EE[\partial_{x_k}\hat{F}(z,\tilde{\alpha}_{t_0}^{x,i},t)]-\EE[\partial_{x_k}\hat{F}(z,\tilde{\alpha}_{t_0}^{y,i},t)]\right|\nonumber\\
\leq\!\!\!\!\!\!\!&& C\left[\|F(z,\cdot)\|_{\infty}\|\partial_{x_k}Q(z)\|_{\ell}+\|\partial_{x}F(z,\cdot)\|_{\infty}\right]e^{-\lambda t}(1\vee t) \mathbb{P}(\tilde{\alpha}_{t_0}^{x,i}\neq \tilde{\alpha}_{t_0}^{y,i}),
\end{eqnarray*}
thus we get
\begin{eqnarray*}
&&\left|\EE[\partial_{x_k}\hat{F}(z,\alpha_{t_0}^{x,i},t)]-\EE[\partial_{x_k}\hat{F}(z,\alpha_{t_0}^{y,i},t)]\right|\\
\leq\!\!\!\!\!\!\! &&C\left[\|F(z,\cdot)\|_{\infty}\|\partial_{x_k}Q(z)\|_{\ell}+\|\partial_{x}F(z,\cdot)\|_{\infty}\right]e^{-\lambda t}(1\vee t)\mathbb{W}(p^x_{i\cdot}(t_0), p^y_{i\cdot}(t_0)),
\end{eqnarray*}
where
$$\mathbb{W}(\mu,\nu):=\inf_{\pi} \sum^{m_0}_{i=1}\sum^{m_0}_{j=1}1_{\{i\neq j\}}\pi_{ij},$$
here the infimum is taken over all coupling measures $\pi=(\pi_{ij})$ on $\mathbb{S}\times \mathbb{S}$ of $\mu=(\mu_1,\ldots,\mu_{m_0})$ and $\nu=(\nu_1,\ldots,\nu_{m_0})$.

By $\mathbb{W}(\mu,\nu)=1/2\|\mu-\nu\|_{\text{var}}$ and \eref{FDP}, we get
\begin{eqnarray*}
&&\left|\EE[\partial_{x_k}\hat{F}(z,\alpha_{t_0}^{x,i},t)]-\EE[\partial_{x_k}\hat{F}(z,\alpha_{t_0}^{y,i},t)]\right|\nonumber\\
\leq\!\!\!\!\!\!\!&& C\left[\|F(z,\cdot)\|_{\infty}\|\partial_{x_k}Q(z)\|_{\ell}+\|\partial_{x}F(z,\cdot)\|_{\infty}\right]e^{-\lambda t}(1\vee t)\|p^x_{i\cdot}(t_0)-p^y_{i\cdot}(t_0)\|_{\text{var}}\nonumber\\
\leq\!\!\!\!\!\!\!&& C\left[\|F(z,\cdot)\|_{\infty}\|\partial_{x_k}Q(z)\|_{\ell}+\|\partial_{x}F(z,\cdot)\|_{\infty}\right]e^{-\lambda t}(1\vee t)|x-y|,
\end{eqnarray*}
which implies that
\begin{eqnarray}
|J_2|\leq C\left[\|F(x,\cdot)\|_{\infty}\|\partial_{x_k}Q(x)\|_{\ell}+\|\partial_{x}F(x,\cdot)\|_{\infty}\right]e^{-\lambda t}(1\vee t).\label{J2.}
\end{eqnarray}

(iii) For $J_3$, note that
\begin{eqnarray*}
\partial_{x_l}\!\left\{\partial_{x_k}\EE \hat F(z,\alpha^{x,i}_{t_0},t) |_{z=x}\right\}\!=\!\!\!\!\!\!\!\!\!&&\left\{\partial_{x_k}\EE \left[\partial_{x_l}\hat F(z,\alpha^{x,i}_{t_0},t) \right]\right\}|_{z=x}\!+\!\left\{\partial_{x_k}\partial_ {x_l}\EE \left[\hat F(z,\alpha^{x,i}_{t_0},t) \right]\right\}|_{z=x}\nonumber\\
\!=:\!\!\!\!\!\!\!\!&&J_{31}+J_{32}.
\end{eqnarray*}

As the same argumente in  $J_2$, it is easy to get
\begin{eqnarray}
|J_{31}|\leq  C\left[\|F(x,\cdot)\|_{\infty}\|\partial_{x_l}Q(x)\|_{\ell}+\|\partial_{x}F(x,\cdot)\|_{\infty}\right]e^{-\lambda t}(1\vee t).\label{J31}
\end{eqnarray}

By \eref{FDP2} and \eref{hatF}, we have
\begin{eqnarray}
|J_{32}|\!=\!\!\!\!\!\!\!\!\!&&\left|\partial_{x_k}\partial_{x_l}P^x_{t_0}\hat F(x,\cdot,t)\right|\nonumber\\
\!\le\!\!\!\!\!\!\!\!\!&& C\|\partial_{x_k}\partial_{ x_l}Q(x)\|_{\ell}\|\hat F(x,\cdot,t)\|_{\infty}
\!\le\! C\|\partial_{x_k}\partial_{ x_l}Q(x)\|_{\ell}\|F(x,\cdot)\|_{\infty}e^{-\lambda t}.\label{J32}
\end{eqnarray}
Hence, \eref{J31} and \eref{J32} yield that
\begin{eqnarray}
|J_{3}|\!\leq \!  C\left[\|F(x,\cdot)\|_{\infty}(\|\partial_{x_l}Q(x)\|_{\ell}\!+\!\|\partial_{x_k}\partial_{ x_l}Q(x)\|_{\ell})+\|\partial_{x}F(x,\cdot)\|_{\infty}\right]e^{-\lambda t}(1\vee t).\label{J3.}
\end{eqnarray}
By \eref{J1.}, \eref{J2.} and \eref{J3.}, we get \eref{E3.}. The proof is complete. $\square$

\begin{remark}
Suppose that \ref{A2} holds, $F$ satisfies \eref{CenCon} with $\|F(x,\cdot)\|_{\infty}\!<\!\infty$, $\forall x\in\RR^n$.
Then following a similar argument as used in \textbf{Step 2} in the proof of Theorem \ref{Poisson}, there exists a constant $C>0$ such that for any $x,y\in \RR^n$,
\begin{eqnarray}
\|\Phi(x,\cdot)-\Phi(y,\cdot)\|_{\infty}\le C\left[\|F(x,\cdot)\|_{\infty}\|Q(x)-Q(y)\|_{\ell}+\|F(x,\cdot)-F(y,\cdot)\|_{\infty}\right].\label{Elip}
\end{eqnarray}
\end{remark}

\begin{remark}
Assuming the "\emph{central condition}" \eref{CenCon}, \eref{SPE} is a solution of Poisson equation \eref{PE1}. However, this solution is not unique usually, but it is enough for our purpose. In fact, if the solution $\Phi$ also satisfies the "\emph{central condition}", i.e., $\sum^{m_0}_{i=1}\Phi(x,i)\mu^{x}_i=0$ for any $x\in\RR^n$, then the solution is unique (cf. \cite[Lemma 4.2.6]{PTW2012}).
\end{remark}

\section{Proofs of averaging principle}
In this section, we will use Poisson equation to study the averaging principle of  Eq.\eref{e:eqn2}. Specifically, we will investigate three convergences: (1) strong convergence of $X^{\varepsilon}$ in $C([0,T],\RR^n)$; (2) weak convergence of $X^{\varepsilon}$ in $C([0,T],\RR^n)$; (3) weak convergence of $X^{\varepsilon}_t$ in $\RR^n$ for any $t>0$.

Prior to proving the aforementioned results, it is necessary to establish a priori estimates of the solution $\{X^{\varepsilon}_t\}_{t \geq 0}$ of Eq. \eref{e:eqn2}, which will be utilized throughout the paper. Since the proof follows a standard argument, we omit the detailed proof here.

\begin{lemma} \label{PMX}
Suppose that  \ref{A1} and \ref{A2} hold. For any $T>0$ and $p>0$, there exists $C_{p,T}>0$ such that
\begin{eqnarray}
\mathbb{E}\big(\sup_{0\leq t\leq T}| X_t^{\varepsilon }|^{p}\big)\le C_{p,T}\left(1+|x|^p\right).  \label{X}
\end{eqnarray}
\end{lemma}

\subsection{Strong convergence of $X^{\varepsilon}$ in $C([0,T],\RR^n)$}
\vspace{0.2cm}

In this subsection, we study the strong convergence \eref{StrongA2} of $X^{\varepsilon}$ in $C([0,T],\RR^n)$. In this case, recall the corresponding averaged Eq. \eref{AR1}.

\subsubsection{Well-posedness of equation \eref{AR1}} We first study the existence and uniqueness of the solution for Eq. \eref{AR1} and its a priori estimates.

\begin{lemma}     \label{PMbarX}
Equation \eref{AR1} admits a unique solution $\{\bar{X}_t\}_{t\geq 0}$. Moreover, for any $T>0$ and $p>0$, there exists $C_{p,T}>0$ such that	
\begin{eqnarray}
\mathbb{E}\big(\sup_{0\leq t\leq T}| \bar{X}_t|^{p}\big)\le C_{p,T}(1+|x|^p).  \label{barX}
\end{eqnarray}
\end{lemma}

\begin{proof}
For any $x_1,x_2\in\RR^n$, we have
\begin{eqnarray*}
&&\left\langle \bar b(x_{1})-\bar b(x_{2}),x_{1}-x_{2}\right\rangle\\
=\!\!\!\!\!\!\!\!&&\left\langle \mu^{x_1}(b(x_1,\cdot))-\mu^{x_2}(b(x_2,\cdot)),x_{1}-x_{2}\right\rangle\\
\leq\!\!\!\!\!\!\!\!&&\left|\mu^{x_1}(b(x_1,\cdot))-\EE b(x_{1},\alpha^{x_1,i}_t)\right||x_1-x_2|+\left|\mu^{x_2}(b(x_2,\cdot))-\EE b(x_{2},\alpha^{x_2,i}_t)\right||x_1-x_2|\\
&&+\left\langle \EE b(x_{1},\alpha^{x_1,i}_t)-\EE b(x_{2},\alpha^{x_2,i}_t), x_1-x_2\right\rangle\\
\le\!\!\!\!\!\!\!\!&&\|b(x_1,\cdot)\|_{\infty}\|p^{x_1}_{i\cdot}-\mu^{x_1}\|_{\text{var}}|x_1-x_2|
+\|b(x_2,\cdot)\|_{\infty}\|p^{x_2}_{i\cdot}-\mu^{x_2}\|_{\text{var}}|x_1-x_2|\\
&&+\left\langle \EE b(x_{1},\alpha^{x_1,i}_t)-\EE b(x_{2},\alpha^{x_2,i}_t), x_1-x_2\right\rangle\\
\le\!\!\!\!\!\!\!\!&&C(1+|x_1|^k+|x_2|^k) e^{-\lambda t}|x_1-x_2|+\left\langle \EE b(x_{1},\alpha^{x_1,i}_t)-\EE b(x_{2},\alpha^{x_2,i}_t), x_1-x_2\right\rangle.
\end{eqnarray*}
Using condition \eref{ConA11}, for any coupling process $(\tilde{\alpha}_{t}^{x_1,i}, \tilde{\alpha}_{t}^{x_2,i})$ of $(\alpha_{t}^{x_1,i}, \alpha_{t}^{x_2,i})$,  we have
\begin{eqnarray*}
\left\langle \EE b(x_{1},\alpha^{x_1,i}_t)-\EE b(x_{2},\alpha^{x_2,i}_t), x_1-x_2\right\rangle
=\!\!\!\!\!\!\!\!\!\!&&\left\langle \EE b(x_{1},\tilde\alpha^{x_1,i}_t)-\EE b(x_{2},\tilde\alpha^{x_2,i}_t), x_1-x_2\right\rangle\nonumber\\
\le\!\!\!\!\!\!\!\!\!\!&& C|x_1-x_2|^2+C\PP(\tilde{\alpha}_{t}^{x_1,i}\neq\tilde{\alpha}_{t}^{x_2,i})|x_1-x_2|,
\end{eqnarray*}
which implies
\begin{eqnarray}
\left\langle \EE b(x_{1},\alpha^{x_1,i}_t)-\EE b(x_{2},\alpha^{x_2,i}_t), x_1-x_2\right\rangle
\le C|x_1-x_2|^2+C\mathbb{W}^2(p^{x_1}_{i\cdot}(t), p^{x_2}_{i\cdot}(t)).\label{Wd}
\end{eqnarray}

As the proof of Proposition \ref{DQ}, \eref{LipQ} implies that
\begin{eqnarray}
\sup_{t>0}\|p^{x_1}_{i\cdot}(t)-p^{x_2}_{i\cdot}(t)\|_{\text{var}}\leq C\|Q(x_1)-Q(x_2)\|_{\ell}\leq C|x_1-x_2|.\label{Var}
\end{eqnarray}
\eref{Wd}, \eref{Var} and $\mathbb{W}(\mu,\nu)=1/2\|\mu-\nu\|_{\text{var}}$ imply that
\begin{eqnarray*}
\left\langle \bar b(x_{1})-\bar b(x_{2}),x_{1}-x_{2}\right\rangle \leq C\left(1+|x_1|^k+|x_2|^k\right) e^{-\lambda t}+C|x_1-x_2|^2.
\end{eqnarray*}
Let $t\to \infty$, we final obtain
 \begin{eqnarray}
\left\langle \bar b(x_{1})-\bar b(x_{2}),x_{1}-x_{2}\right\rangle\leq C|x_{1}-x_{2}|^{2}, \label{barc1}
\end{eqnarray}
which implies
\begin{eqnarray}
\langle \bar b(x),x\rangle \leq C(1+|x|^{2}).\label{Eobarb}
\end{eqnarray}
Hence, Eq.\eref{AR1} admits a unique solution (see \cite[Theorem 3.1.1]{LR2015}). Moreover,  \eref{barX} holds by using \eref{Eobarb} and \eref{ConA14}. The proof is complete.		
\end{proof}

\subsubsection{Mollifying approximation} Using \eref{Elip}, \ref{A1}-\ref{A3} imply the Poisson equation's solution $\Phi(x,\alpha)$ is polynomial growth and local Lipschitz continuous with respective to $x$, which is not smooth enough for applying It\^{o}'s formula on $\Phi$ later. Thus a mollifying convergence argument will be introduced firstly.

Let $\rho:\RR^n\to[0,1]$ be a smooth function such that for $m\in\mathbb{N}_{+}$,
\begin{eqnarray}
\int_{\mathbb{R}^{n}}\rho (z)dz=1,\quad \int_{\mathbb{R}^{n}} \left |z\right |^{m}\rho (z)dz\le C_m,\quad \left \|\nabla^{m}\rho (z) \right\|\le C_m\rho (z).\label{F6.1}
\end{eqnarray}
Define a mollifying approximation of $H:\RR^n\times \mathbb{S}\rightarrow \RR^n$  as follows:
\begin{eqnarray*}
H_m(x,i):=\int_{\mathbb{R}^{n}}H(x-z,i)\rho ^{m}(z)dz.
\end{eqnarray*}

We have the following lemma:

\begin{lemma}\label{App}
Suppose there exist $C,k>0$ such that for $x,y\in\RR^n$,
\begin{eqnarray}
\|H(x,\cdot)-H(y,\cdot)\|_{\infty}\leq C(1+|x|^k+|y|^k)|x-y|,\quad \|H(x,\cdot)\|_{\infty}\leq C(1+|x|^k).\label{locLip}
\end{eqnarray}
Then for $m\in\mathbb{N}_{+}$ and $x\in\RR^n$, we have
\begin{eqnarray}
\|H(x,\cdot)-H_m(x,\cdot)\|_\infty\leq C(1+|x|^k) m^{-1},\quad \|H_m(x,\cdot)\|_{\infty}\leq C(1+|x|^k),\label{App1}
\end{eqnarray}
\begin{eqnarray}
\| \partial_x H_m(x,\cdot)\|_{\infty}\leq C(1+|x|^k), \quad\| \partial^2_x H_m(x,\cdot)\|_{\infty}\leq C(1+|x|^k) m,\label{App2}
\end{eqnarray}
where $C>0$ is a constant independent of $m$ and $x$.
\end{lemma}

\begin{proof}
(i) {\it The proof of \eref{App1}. } For $m\in\mathbb{N}_{+}$, we have
\begin{eqnarray*}
\left |H_m(x,i)-H(x,i)\right |=\!\!\!\!\!\!\!&&\Big |\int_{\mathbb{R}^{n}}H(x-z,i)\rho ^{m}(z)dz-\int_{\mathbb{R}^{n}}H(x,i)\rho ^{m}(z)dz\Big |\nonumber\\
\leq\!\!\!\!\!\!\!&&C\int_{\mathbb{R}^{n}}(1+|x|^k+|z|^k)||z|\rho ^{m}(z)dz\nonumber \\
=\!\!\!\!\!\!\!&&C(1+|x|^k)m^{-1}\int_{\mathbb{R}^{n}}\left |mz\right |m^{n}\rho (mz)dz+Cm^{-k}\int_{\mathbb{R}^{n}}|mz|^k m^{n}\rho (mz)dz\nonumber \\
=\!\!\!\!\!\!\!&&C(1+|x|^k)m^{-1}\int_{\mathbb{R}^{n}}\left |y\right |\rho (y)dy+Cm^{-k}\int_{\mathbb{R}^{n}}\left |y\right |^k\rho (y)dy\nonumber\\
\leq\!\!\!\!\!\!\!&&C(1+|x|^k)m^{-1},
\end{eqnarray*}
which implies
$$
\|H(x,\cdot)-H_m(x,\cdot)\|_\infty\leq C(1+|x|^k) m^{-1}.
$$
Consequently, we have for $x\in\mathbb{R}^n$,
\begin{eqnarray*}
\|H_m(x,\cdot)\|_{\infty}\leq \|H(x,\cdot)-H_m(x,\cdot)\|_{\infty}+\|H(x,\cdot)\|_{\infty}\leq C(1+|x|^k).
\end{eqnarray*}

(ii) {\it The proof of \eref{App2}. } For $x_1,x_2\in\RR^n$ and $i\in\mathbb{S}$,
\begin{eqnarray*}
\left |H_m(x_1,i)-H_m(x_2,i)\right |=\!\!\!\!\!\!\!&&\Big |\int_{\mathbb{R}^{n}}H(x_1-z,i)\rho ^{m}(z)dz-\int_{\mathbb{R}^{n}}H(x_2-z,i)\rho ^{m}(z)dz\Big |\nonumber\\
&&\le C\int_{\mathbb{R}^{n}}(1+|x_1|^k+|x_2|^k+|z|^k)|x_1-x_2|\rho ^{m}(z)dz\nonumber\\
&&\le (1+|x_1|^k+|x_2|^k)|x_1-x_2|
\end{eqnarray*}
 implies
\begin{eqnarray*}
\left \|\partial_x H_m(x,\cdot)\right \|_{\infty}\le C(1+|x|^k).
\end{eqnarray*}
Since
\begin{eqnarray*}
\partial_x H_m(x,i)=\int_{\mathbb{R}^{n}}H(z,i)\partial_x\left[\rho ^{m}(x-z)\right]dz
=\int_{\mathbb{R}^{n}}H(x-z,i)m^{n+1}\nabla\rho(mz)dz.
\end{eqnarray*}
Then by \eref{locLip} and \eref{F6.1}, we get
\begin{eqnarray*}
\left \|\partial_x H_m(x_1,i)-\partial_x H_m(x_2,i)\right \|\le\!\!\!\!\!\!\!&& \int_{\mathbb{R}^{n}}\left |H(x_1-z,i)-H(x_2-z,i)\right |\left |\nabla\rho (mz)\right |\cdot m^{n+1}dz\\
\le\!\!\!\!\!\!\!&& C\left |x_1-x_2\right |\int_{\mathbb{R}^{n}}(1+|x_1|^k+|x_2|^k+|z|^k)\rho (mz)m^{n+1}dz\\
\le\!\!\!\!\!\!\!&& C(1+|x_1|^k+|x_2|^k)|x_1-x_2|m,
\end{eqnarray*}
which implies
\begin{eqnarray*}
\left \|\partial_x^{2} H_m(x,\cdot)\right \|_{\infty}\le C(1+|x|^k) m.
\end{eqnarray*}
The proof is complete.
\end{proof}

\subsubsection{Proof of Theorem \ref{main result 1}}  We prove our first result in the averaging principle.

\vspace{-2mm}
~~\\
\textbf{Proof of Theorem \ref{main result 1}:} We divide the proof into three steps.

\vspace{-3mm}
~~\\
\textbf{Step 1:} Note that
\begin{eqnarray*}
X_t^\varepsilon-\bar{X_t}=\!\!\!\!\!\!&&\int_0^{t}\left[b(X_s^\varepsilon,\alpha_s^\varepsilon)-\bar{b}(\bar{X}_s)\right]ds+\int_0^t\left[\sigma(X_s^\varepsilon)-\sigma(\bar{X}_s)\right]dW_s.
\end{eqnarray*}
By It\^{o}'s formula, we have  that, for $p\geq 4$,
		\begin{eqnarray*}
			|X_t^\varepsilon-\bar{X_t}|^{p}=\!\!\!\!\!\!&&p\int_0^t|X_s^\varepsilon-\bar{X}_s|^{p-2}\left\langle X_s^\varepsilon-\bar{X}_s,b(X_s^{\varepsilon},\alpha_s^{\varepsilon})-\bar{b}({X_s^{\varepsilon }})\right\rangle ds\\
			&&+p\int_0^t|X_s^\varepsilon-\bar{X}_s|^{p-2}\left\langle X_s^\varepsilon-\bar{X}_s,\bar{b}({X_s^{\varepsilon }})-\bar{b}(\bar{X}_s)\right\rangle ds\\
			 &&+{\frac{p(p-2)}{2}}\int_0^t|X_s^\varepsilon-\bar{X}_s|^{p-4}\big|\!\left(\sigma(X_s^\varepsilon)-\sigma(\bar{X}_s)\right)^{*}\!\cdot\!\left(X_s^\varepsilon-\bar{X}_s \right) \!\big|^{2}ds\\
			&&+{\frac{p}{2}}\int_0^t|X_s^\varepsilon-\bar{X}_s|^{p-2} \|\sigma(X_s^\varepsilon)-\sigma(\bar{X}_s)\|^2ds\\
			&&+p\int_0^t|X_s^\varepsilon-\bar{X_s}|^{p-2}\left\langle X_s^\varepsilon-\bar{X_s},\left(\sigma(X_s^\varepsilon)-\sigma(\bar{X}_s)\right)dW_s\right\rangle\\
=:\!\!\!\!\!\!&&\sum^{5}_{i=1}B_i(t).
		\end{eqnarray*}
Young's inequality yields that
\begin{eqnarray}
			&&\mathbb{E}\big(\sup_{0\leq t\leq T}|B_1(t)|\big)\nonumber\\
 \le\!\!\!\!\!\!\!\!&& C_{p} \mathbb{E}\Big[\sup_{0\leq s\leq T}|X_s^\varepsilon-\bar{X}_s|^{p-2}\cdot\sup_{0\leq t\leq T}\Big|\int_0^t\left\langle X_s^\varepsilon-\bar{X}_s,b(X_s^{\varepsilon},\alpha_s^{\varepsilon})-\bar{b}({X_s^{\varepsilon }})\right\rangle ds\Big|\Big]\nonumber \\
			\le\!\!\!\!\!\!\!\!&& C_p\mathbb{E}\Big(\sup_{0\leq t\leq T}\Big|\int_0^t\!\!\left\langle X_s^\varepsilon\!-\!\bar{X}_s,b(X_s^{\varepsilon},\alpha_s^{\varepsilon})\!-\!\bar{b}({X_s^{\varepsilon }})\right\rangle ds\Big|^{\frac{p}{2}}\Big) +\frac{1}{3}\mathbb{E}\big(\sup_{0\leq t\leq T}|X_t^\varepsilon\!-\!\bar{X}_t|^{p}\big)\!.\label{S5.1}
		\end{eqnarray}
By \eref{ConA14} and \eref{barc1}, it is easy to prove
		\begin{eqnarray}
			\sum^{4}_{i=2}\EE\big(\sup_{0\leq t\leq T}|B_i(t)|\big)\le C_{p}\int_0^T\EE|X_s^\varepsilon-\bar{X}_s|^{p}ds.\label{S5.2}
		\end{eqnarray}

The Burkholder-Davis-Gundy (BDG) inequality and Young's inequality imply that
		\begin{eqnarray}
			\mathbb{E}\big(\sup_{0\leq t\leq T}|B_5(t)|\big)\!\!\!\!\!\!\!\!&&\le C_{p}\mathbb{E}\Big[\int_{0}^{T}|X_s^\varepsilon-\bar{X}_s|^{2(p-1)} \|\sigma(X_s^\varepsilon)-\sigma(\bar{X}_s)\|^2ds \Big]^{1/2}\nonumber\\
			\!\!\!\!&&\le \frac{1}{3}\mathbb{E}\big(\sup_{0\leq t\leq T}|X_t^\varepsilon-\bar{X}_t|^{p}\big)+\!  C_p\mathbb{E}\int_{0}^{T}|X_s^\varepsilon-\bar{X}_s|^{p}ds\label{S5.4}
		\end{eqnarray}
Combining  \eref{S5.1}-\eref{S5.4}, we get
		\begin{eqnarray*}
			\mathbb{E}\big(\sup_{0\leq t\leq T}|X^{\varepsilon}_t-\bar{X}_t|^p\big)\!\!\!\!\!\!&&
			\leq C_{p}\mathbb{E}\Big[\sup_{0\leq t\leq T}\Big|\int^t_0\langle X_{s}^{\varepsilon}-\bar{X}_{s}, b(X^{\varepsilon}_{s},\alpha^{\varepsilon}_{s})-\bar{b}(X^{\varepsilon}_{s})\rangle ds\Big|^{p/2}\Big]\\
			\!\!\!\!\!\!&&+C_p\int_{0}^{T}\mathbb{E}|X_s^\varepsilon-\bar{X}_s|^{p}ds.
		\end{eqnarray*}
Then by Gronwall's inequality, it follows
		\begin{eqnarray}
			\mathbb{E}\big(\sup_{0\leq t\leq T}|X^{\varepsilon}_t-\bar{X}_t|^p\big)
			\leq C_{p,T}\mathbb{E}\Big[\sup_{0\leq t\leq T}\big|\int^t_0 \langle X_{s}^{\varepsilon}-\bar{X}_{s}, b(X^{\varepsilon}_{s},\alpha^{\varepsilon}_{s})-\bar{b}(X^{\varepsilon}_{s})\rangle ds\big|^{p/2}\Big]. \label{S5.61}
		\end{eqnarray}

\vspace{-3mm}
~~\\
\textbf{Step 2:} In order to estimate the right term in \eref{S5.61}, we consider the Poisson equation
\begin{eqnarray}
-Q(x)\Phi(x,\cdot)(i)=F(x,i), \label{PEQ}
\end{eqnarray}
where $F(x,i):=b(x,i)-\bar{b}(x)$. Obviously, $F$ satisfies "\emph{central condition}" \eref{CenCon}. Then Th.\ref{Poisson} yields that \eref{PEQ} has a solution $\Phi$. Moreover, using \eref{Elip}, \eref{ConA13} and \eref{LipQ} imply that
\begin{eqnarray*}
\|\Phi(x,\cdot)-\Phi(y,\cdot)\|_{\infty}\leq C\left(1+|x|^k+|y|^k\right)|x-y|,\quad \|\Phi(x,\cdot)\|_{\infty}\leq C\left(1+|x|^k\right).
\end{eqnarray*}
Lemma \ref{App} implies that there exist smooth sequences $\{\Phi_m(x,i)\}_{m\geq 1}$ such that
\begin{eqnarray}
\|\Phi(x,\cdot)-\Phi_m(x,\cdot)\|_\infty\leq C\left(1+|x|^k\right) m^{-1},\quad \|\Phi_m(x,\cdot)\|_{\infty}\leq C\left(1+|x|^k\right),\label{PhiApp1}
\end{eqnarray}
\begin{eqnarray}
\| \partial_x \Phi_m(x,\cdot)\|_{\infty}\leq C\left(1+|x|^k\right), \quad\| \partial^2_x \Phi_m(x,\cdot)\|_{\infty}\leq C\left(1+|x|^k\right) m.\label{PhiApp2}
\end{eqnarray}
Using Young's inequality and \eref{PhiApp1}, there exists $k_p>0$ such that
\begin{eqnarray}
&&\mathbb{E}\big(\sup_{0\leq t\leq T}|X_t^\varepsilon-\bar{X_t}|^{p}\big)\nonumber\\
\le\!\!\!\!\!\!&&C_{p,T}\mathbb{E}\Big[\sup_{0\leq t\leq T}\Big|\int^t_0\langle X_{s}^{\varepsilon}-\bar{X}_{s},Q(X_{s}^{\varepsilon})\Phi(X_{s}^{\varepsilon},\cdot)(\alpha_s^{\varepsilon})
-Q(X_{s}^{\varepsilon})\Phi_m(X_{s}^{\varepsilon},\cdot)(\alpha_s^{\varepsilon})\rangle ds\Big|^{p/2}\Big]\nonumber\\
&&+C_{p,T}\mathbb{E}\Big[\sup_{0\leq t\leq T}\Big|\int^t_0 \langle X_{s}^{\varepsilon}-\bar{X}_{s},Q(X_{s}^{\varepsilon})\Phi_m(X_{s}^{\varepsilon},\cdot)(\alpha_s^{\varepsilon})\rangle ds\Big|^{p/2}\Big]\nonumber\\
\le\!\!\!\!\!\!&&\frac{1}{2}\mathbb{E}\big(\sup_{0\leq t\leq T}|X_t^\varepsilon-\bar{X_t}|^{p}\big)+C_{p,T}\mathbb{E}\int^T_0\|Q(X_{s}^{\varepsilon})\|^p_{\ell}
\|\Phi(X_{s}^{\varepsilon},\cdot)-\Phi_m(X_{s}^{\varepsilon},\cdot)\|^p_{\infty} ds\nonumber\\
&&+C_{p,T}\mathbb{E}\Big[\sup_{0\leq t\leq T}\Big|\int^t_0 \langle X_{s}^{\varepsilon}-\bar{X}_{s},Q(X_{s}^{\varepsilon})\Phi_m(X_{s}^{\varepsilon},\cdot)(\alpha_s^{\varepsilon})\rangle ds\Big|^{p/2}\Big]\nonumber\\
\le\!\!\!\!\!\!&&\frac{1}{2}\mathbb{E}\big(\sup_{0\leq t\leq T}|X_t^\varepsilon-\bar{X_t}|^{p}\big)+C_{p,T}(1+|x|^{k_p})m^{-p}\nonumber\\
&&+C_{p,T}\mathbb{E}\Big[\sup_{0\leq t\leq T}\Big|\int^t_0 \langle X_{s}^{\varepsilon}-\bar{X}_{s}, Q(X_{s}^{\varepsilon})\Phi_m(X_{s}^{\varepsilon},\cdot)(\alpha_s^{\varepsilon})\rangle ds\Big|^{p/2}\Big].\label{F4.19}
\end{eqnarray}

Recall that $\alpha^{\varepsilon}$ can be described as following (cf. \cite{BBG1999}),
\begin{equation*} \label{2.3}
d\alpha^{\varepsilon}_t=\int_{[0, \infty)}g^{\varepsilon}(X^{\vare}_{t-},\alpha^{\varepsilon}_{t-}, z)N(dt, dz),
\end{equation*}
where $N(dt, dz)$ is a Poisson random measure defined on $\Omega\times\mathcal{B}(\mathbb{\RR_{+}})\times\mathcal{B}(\mathbb{\RR_{+}})$ with Lebesgue measure as its intensity measure, $N(dt, dz)$ is independent of $W$, and
$$
g^{\varepsilon}(x,i, z)=\sum_{j\in \mathbb{S}\backslash\{ i\} }(j-i)1_{z\in\triangle^{\varepsilon}_{ij}(x)},\quad   i\in \mathbb{S}\,,
$$
$\triangle^{\varepsilon}_{ij}(x)$ are the consecutive (with respect to the lexicographic ordering on $\mathbb{S}\times\mathbb{S}$) left-closed, right-open intervals of $\mathbb{R}_{+}$, each having length
$q_{ij}(x)\varepsilon^{-1}$, with $\Delta^{\varepsilon}_{12}= [0,\varepsilon^{-1}q_{12}(x))$.

Applying It\^{o}'s formula, see e.g. \cite[(2.7)]{YZ2010}, we obtain
\begin{eqnarray*}
&&\big\langle\Phi_m(X_t^{\varepsilon},\alpha_t^{\varepsilon}),X^{\varepsilon}_t-\bar{X}_t\big\rangle\\
=\!\!\!\!\!\!\!\!\!&&\int_0^t\!\!\left \langle\partial_x\Phi_m(X_{s}^{\varepsilon},\alpha_s^{\varepsilon})\cdot b(X_{s}^{\varepsilon},\alpha_s^{\varepsilon}), X^{\varepsilon}_s\!-\!\bar{X}_s\right\rangle ds+\int_0^t\langle\partial_x\Phi_m(X_{s}^{\varepsilon},\alpha_s^{\varepsilon})\cdot \sigma(X_{s}^{\varepsilon})dW_s, X^{\varepsilon}_s-\bar{X}_s\rangle\\
&&+\frac12\!\int_0^t\!\!\left\langle \text{Tr}\big[\partial_x^{2}\Phi_m(X_{s}^{\varepsilon},\alpha_s^{\varepsilon})\cdot \left(\sigma\sigma^{*}\right)(X_{s}^{\varepsilon})\big], X^{\varepsilon}_s\!-\!\bar{X}_s\right\rangle ds\\
&&+\int_0^t\!\int_{[0,\infty)}\left\langle \Phi_m(X_{s-}^{\varepsilon},\alpha^{\varepsilon}_{s-}\!+\!g^{\varepsilon}(X_{s-}^{\varepsilon},\alpha_{s-}^{\varepsilon},z))-\Phi_m(X_{s-}^{\varepsilon}\!,\alpha_{s-}^{\varepsilon}),X^{\varepsilon}_s-\bar{X}_s\right\rangle \tilde{N}(ds,dz)\\
&&+\int_0^{t}\left\langle \Phi_m(X_s^{\varepsilon},\alpha_s^{\varepsilon}),\left[b(X_s^\varepsilon,\alpha_s^\varepsilon)-\bar{b}(\bar{X}_s)\right]\right\rangle ds+\int_0^t\left\langle\Phi_m(X_s^{\varepsilon},\alpha_s^{\varepsilon}),\left[\sigma(X_s^\varepsilon)-\sigma(\bar{X}_s)\right]dW_s\right\rangle\\
&&+\int^t_0\text{Tr}\left[\sigma(X_s^{\varepsilon})\big(\sigma^{\ast}(X_s^{\varepsilon})-\sigma^{\ast}(\bar {X}_s)\big)\cdot \partial_x\Phi_m(X_s^{\varepsilon},\alpha_s^{\varepsilon})\right] ds\\
&&+\frac{1}{\varepsilon}\int_0^t\left\langle Q(X_{s}^{\varepsilon})\Phi_m(X_{s}^{\varepsilon},\cdot)(\alpha_s^{\varepsilon}), X^{\varepsilon}_s-\bar{X}_s\right\rangle ds,
\end{eqnarray*}
where
$$\text{Tr}\big[\partial_x^{2}\Phi_m(x,i)\cdot \left(\sigma\sigma^{*}\right)(x)\big]\!:=\!\left(\text{Tr}\big[\partial_x^{2}\Phi^1_m(x,i)\cdot \left(\sigma\sigma^{*}\right)(x)\big],\ldots,\text{Tr}\big[\partial_x^{2}\Phi^n_m(x,i)\cdot \left(\sigma\sigma^{*}\right)(x)\big]\right).$$
Therefore, it is easy to see
\begin{eqnarray}
&&-\int_0^t \langle Q(X_{s}^{\varepsilon})\Phi_m(X_{s}^{\varepsilon},\cdot)(\alpha_s^{\varepsilon}),X^{\varepsilon}_s-\bar{X}_s\rangle ds\nonumber\\
=\!\!\!\!\!\!\!\!\!&&\varepsilon\Big\{-\langle\Phi_m(X_t^{\varepsilon},\alpha_t^{\varepsilon}),X^{\varepsilon}_t-\bar{X}_t\rangle+\int_0^t \langle\partial_x\Phi_m(X_{s}^{\varepsilon},\alpha_s^{\varepsilon})\cdot b(X_{s}^{\varepsilon},\alpha_s^{\varepsilon}), X^{\varepsilon}_s-\bar{X}_s\rangle ds\nonumber\\
&&+\int_0^t\langle\partial_x\Phi_m(X_{s}^{\varepsilon},\alpha_s^{\varepsilon})\cdot \sigma(X_{s}^{\varepsilon})dW_s, X^{\varepsilon}_s-\bar{X}_s\rangle+\int_0^{t}\langle \Phi_m(X_s^{\varepsilon},\alpha_s^{\varepsilon}),\left[b(X_s^\varepsilon,\alpha_s^\varepsilon)-\bar{b}(\bar{X}_s)\right]\rangle ds\nonumber\\
&&+\!\!\int_0^t\!\!\langle\Phi_m(X_s^{\varepsilon},\alpha_s^{\varepsilon}),[\sigma(X_s^\varepsilon)\!-\!\sigma(\bar{X}_s)]dW_s\rangle\!+\!\!\int^t_0\!\!\text{Tr}\!\left[\sigma(X_s^{\varepsilon})(\sigma^{\ast}(X_s^{\varepsilon})\!-\!\sigma^{\ast}(\bar {X}_s))\!\cdot\! \partial_x\Phi_m(X_s^{\varepsilon},\alpha_s^{\varepsilon})\right]\! ds\Big\}\nonumber\\
&&+\varepsilon\int_0^t \frac{1}{2}\langle \text{Tr}\big[\partial_x^{2}\Phi_m(X_{s}^{\varepsilon},\alpha_s^{\varepsilon})\cdot \left(\sigma\sigma^{*}\right)(X_{s}^{\varepsilon})\big], X^{\varepsilon}_s-\bar{X}_s\rangle ds\nonumber\\
&&+\varepsilon\int_0^t\!\int_{[0,\infty)}\langle \Phi_m(X_{s-}^{\varepsilon},\alpha^{\varepsilon}_{s-}+\!g^{\varepsilon}(X_{s-}^{\varepsilon},\alpha_{s-}^{\varepsilon},z))-\Phi_m(X_{s-}^{\varepsilon}\!,\alpha_{s-}^{\varepsilon}),X^{\varepsilon}_s-\bar{X}_s\rangle \tilde{N}(ds,dz)\nonumber\\
=:\!\!\!\!\!\!\!\!\!&&\sum^3_{i=1}V^{\varepsilon}_i(t).\label{F4.11}
\end{eqnarray}

According to \eref{F4.19} and \eref{F4.11}, we get
\begin{eqnarray}
\mathbb{E}\big(\sup_{0\leq t\leq T}|X_t^\varepsilon-\bar{X_t}|^{p}\big)\le\!\!\!\!\!\!&& C_{p,T}\left(1\!+\!|x|^{kp}\right)m^{-p}+C_{p,T}\sum^3_{i=1}\mathbb{E}\big(\sup_{0\leq t\leq T}\left|V^{\varepsilon}_i(t)\right|^{p/2}\big).\label{F4.9}
\end{eqnarray}
	
\vspace{-3mm}
~\\
\textbf{Step 3:} In this step, we show the estimate of $V^{\varepsilon}_1(t)$, $V^{\varepsilon}_2(t)$ and $V^{\varepsilon}_3(t)$, respectively. By BDG inequality and Young's inequality, \eref{PhiApp1}, \eref{PhiApp2}, \eref{X} and \eref{barX}, we get
\begin{eqnarray}
\mathbb{E}\big(\sup_{0\leq t\leq T}\left|V^{\varepsilon}_1(t)\right|^{p/2}\big)\leq\!\!\!\!\!\!\!\!&&\varepsilon^{p/2} C_{p,T}\big[1+\mathbb{E}\big(\sup_{0\leq t\leq T}|X_{t}^{\varepsilon}|^{k_p}\big)+\mathbb{E}\big(\sup_{0\leq t\leq T}|\bar X_{t}|^{k_p}\big)\big]\nonumber\\
\leq\!\!\!\!\!\!\!\!&& C_{p,T}\varepsilon^{p/2}\big(1+|x|^{k_p}\big)\label{S5.5}
\end{eqnarray}
and
\begin{eqnarray}
\mathbb{E}\big(\sup_{0\leq t\leq T}\left|V^{\varepsilon}_2(t)\right|^{p/2}\big)
\le\!\!\!\!\!\!\!\!&& \frac14\mathbb{E}\big(\sup_{0\leq t\leq T}|X_t^\varepsilon\!-\!\bar{X_t}|^{p}\big)\!+\! C_{p,T}\varepsilon^p\mathbb{E}\int_0^T\!\!\|\partial_x^{2}\Phi_m(X_{s}^{\varepsilon},\alpha_s^{\varepsilon})\|^p \|\sigma\sigma^{\ast}(X_{s}^{\varepsilon})\|^{p} ds\nonumber\\
\le\!\!\!\!\!\!\!\!&& \frac{1}{4}\mathbb{E}\big(\sup_{0\leq t\leq T}|X_t^\varepsilon-\bar{X_t}|^{p}\big)+C_{p,T}m^p\varepsilon^p\mathbb{E}\int_0^T(1+|X_{s}^{\varepsilon}|^{k_p})ds\nonumber\\
\le\!\!\!\!\!\!\!\!&&\frac{1}{4}\mathbb{E}\big(\sup_{0\leq t\leq T}|X_t^\varepsilon-\bar{X_t}|^{p}\big)+C_{p,T}m^p\varepsilon^p(1+|x|^{k_p}).\label{S5.6}
\end{eqnarray}

Using Kunita's first inequality (see \cite[Theorem 4.4.23]{A}), \eref{PhiApp1} and condition \eref{Finte K}, we have for any $p\geq 4$,
\begin{eqnarray}
&&\mathbb{E}\big(\sup_{0\leq t\leq T}\left|V^{\varepsilon}_3(t)\right|^{p/2}\big)\nonumber\\
\le\!\!\!\!\!\!\!\!&&C\varepsilon^{p/2}\mathbb{E}\Big[\int^T_0\!\!\int_{[0,\infty)}|\Phi_m(X^{\varepsilon}_{s},\alpha^{\varepsilon}_{s}
\!+\!g^{\varepsilon}(X^{\varepsilon}_{s},\alpha^{\varepsilon}_{s},z))-\Phi_m(X^{\varepsilon}_{s},\alpha^{\varepsilon}_{s})|^2|X_s^\varepsilon-\bar{X_s}|^{2} dzds\Big]^{p/4}\nonumber\\
&&+C_p\varepsilon^{p/2}\mathbb{E}\int^T_0\int_{[0,\infty)}|\Phi_m(X^{\varepsilon}_{s},\alpha^{\varepsilon}_{s}+g^{\varepsilon}(X^{\varepsilon}_{s},\alpha^{\varepsilon}_{s},z))-\Phi_m(X^{\varepsilon}_{s},\alpha^{\varepsilon}_{s})|^{p/2}|X_s^\varepsilon-\bar{X_s}|^{p/2}  dzds\nonumber\\
\le\!\!\!\!\!\!\!\!&&\frac{1}{4}\mathbb{E}\big(\sup_{0\leq t\leq T}|X_t^\varepsilon-\bar{X_t}|^{p}\big)\nonumber\\
&&+C_p\varepsilon^{p}\mathbb{E}\Big[\int^T_0\!\!\int_{[0,\infty)}|\Phi_m(X^{\varepsilon}_{s},\alpha^{\varepsilon}_{s}\!+\!g^{\varepsilon}(X^{\varepsilon}_{s},\alpha^{\varepsilon}_{s},z))-\Phi_m(X^{\varepsilon}_{s},\alpha^{\varepsilon}_{s})|^2 dzds\Big]^{p/2}\nonumber\\
&&+C_p\varepsilon^{p}\mathbb{E}\Big[\int^T_0\int_{[0,\infty)}|\Phi_m(X^{\varepsilon}_{s},\alpha^{\varepsilon}_{s}+g^{\varepsilon}(X^{\varepsilon}_{s},\alpha^{\varepsilon}_{s},z))-\Phi_m(X^{\varepsilon}_{s},\alpha^{\varepsilon}_{s})|^{p/2} dzds\Big]^2\nonumber\\
\le\!\!\!\!\!\!\!\!&&\frac{1}{4}\mathbb{E}\big(\sup_{0\leq t\leq T}|X_t^\varepsilon-\bar{X_t}|^{p}\big)+C_p\varepsilon^{p}\mathbb{E}\Big[\int^T_0(1+|X^{\varepsilon}_{s}|^{2k})\int_{[0,K(X^{\varepsilon}_{s})\varepsilon^{-1}]} dzds\Big]^{p/2}\nonumber\\
&&+C_p\varepsilon^{p}\mathbb{E}\Big[\int^T_0(1+|X^{\varepsilon}_{s}|^{\frac{pk}{2}})\int_{[0,K(X^{\varepsilon}_{s})\varepsilon^{-1}]} dzds\Big]^2\nonumber\\
\le\!\!\!\!\!\!\!\!&&\frac{1}{4}\mathbb{E}\big(\sup_{0\leq t\leq T}|X_t^\varepsilon-\bar{X_t}|^{p}\big)+C_p\varepsilon^{p/2}\mathbb{E}\Big[\int^T_0(1+|X^{\varepsilon}_{s}|^{2k})(1+|X^{\varepsilon}_{s}|^k)ds\Big]^{p/2}\nonumber\\
&&+C_p\varepsilon^{p-2}\mathbb{E}\Big[\int^T_0(1+|X^{\varepsilon}_{s}|^{\frac{pk}{2}})(1+|X^{\varepsilon}_{s}|^k)ds\Big]^{2}\nonumber\\
\le\!\!\!\!\!\!\!\!&&\frac{1}{4}\mathbb{E}\big(\sup_{0\leq t\leq T}|X_t^\varepsilon-\bar{X_t}|^{p}\big)+C_{p,T}\left(1+|x|^{k_p}\right)\varepsilon^{p/2}.\label{F4.24}
\end{eqnarray}
Hence, by \eref{F4.9}-\eref{F4.24}, we obtain
\begin{eqnarray*}
\mathbb{E}\big(\sup_{0\leq t\leq T}|X_t^\varepsilon-\bar{X_t}|^{p}\big)\le C_{p,T}\left(1+|x|^{k_p}\right)\left(m^{-p}+m^p\varepsilon^p+\varepsilon^{p/2}\right).
\end{eqnarray*}

Finally, we have \eref{main 1}  by taking $m=[\varepsilon^{-1/2}]$, where $[s]$ means the maximum integral part of $s$. The proof is complete. $\square$

\subsection{Weak convergence of $X^{\varepsilon}$ in $C([0,T],\RR^n)$}

We will take the following steps: Firstly, we study the well-posedness of the corresponding averaged equation Eq.\eref{AR21}. Secondly, we establish the tightness of $X^{\varepsilon}$ in $C([0,T],\RR^n)$. Finally, we employ the martingale problem approach to determine the weak limiting process $X$.

\subsubsection{Well-posedness of equation \eref{AR21}}

\begin{lemma} \label{PMbarX2}
Suppose that \ref{A1}-\ref{A4} hold. Eq.\eref{AR21} has a unique solution $\bar{X}$. Moreover, for any $T>0$ and $p>0$, there exists $C_{p,T}>0$ such that	
\begin{eqnarray}
\mathbb{E}\big(\sup_{0\leq t\leq T}| \bar{X}_t|^{p}\big)\le C_{p,T}(1+|x|^p).  \label{barX2}
\end{eqnarray}
\end{lemma}

\begin{proof}
It is sufficient to prove
\begin{eqnarray}
\|\bar{\sigma}(x)-\bar{\sigma}(y)\|\leq C|x-y|.\label{Lipbarsigma}
\end{eqnarray}
\eref{Lipbarsigma} and \eref{barc1} imply that Eq.\eref{AR21} admits a unique solution. Moreover, \eref{barX2} holds by the same argument in Lemma \ref{PMX}. Hence, we only prove \eref{Lipbarsigma}.

Since
$$(\overline{\sigma\sigma^{\ast}})(x)
	=\sum_{i\in\mathbb{S}}[\sigma\sigma^{\ast}](x,i)\mu^x_{i}=\lim_{t\rightarrow \infty}P^x_t[\sigma\sigma^{\ast}](x,\cdot)(i),$$
$\sigma\in C^1_b(\RR^n\times\mathbb{S},\RR^{n}\otimes\RR^d)$ and \eref{Partialb} yield $\overline{\sigma\sigma^{\ast}}\in C^{1}_b(\RR^n,\RR^n\otimes\RR^n)$. And $\eref{NonD}$ implies that $\overline{\sigma\sigma^{\ast}}$ is non-degenerate, i.e.,
\begin{eqnarray}
			\inf_{x\in\RR^n,z\in \RR^n\backslash\{0\}}\frac{\langle \overline{\sigma\sigma^{\ast}}(x)\cdot z, z\rangle}{|z|^2}>0. \label{NonDbar}
		\end{eqnarray}
By a similar argument as used in \cite[Lemma A.7]{CLX}, we can prove \eref{Lipbarsigma} for $\bar{\sigma}(x)=
	(\overline{\sigma\sigma^{\ast}})^{1/2}(x)$. The proof is complete.	
\end{proof}

\begin{remark}
The non-degenerate condition \eref{NonD} on $\sigma\sigma^{\ast}$ is not necessary, that is $\sigma\sigma^{\ast}$ is allowed to degenerate. In fact, if $\sigma\in C^2_b(\RR^n\times\mathbb{S}, \RR^n\otimes\RR^d)$ and $Q\in C^2(\RR^n,\RR^{m_0}\otimes\RR^{m_0})$ with $\sup_{x\in\RR^n}\sum^2_{k=1}\|\nabla^k Q(x)\|_{\ell}<\infty$, we can prove $\overline{\sigma\sigma^{\ast}}\in C^{2}_b(\RR^n,\RR^n\otimes\RR^n)$, thus \eref{Lipbarsigma} holds by \cite[Theorem 5.2.3]{SV}.
\end{remark}

\begin{remark}\label{HighR}
The higher regularity of $\bar{b}$ and $\bar{\sigma}$ needs more regularity of $b$, $\sigma$ and $Q$. In fact, if ${\color{blue} {\bf H}_{k}}$ ($k=2,3,4$) holds, by using $\bar{b}(x)=\lim_{t\to \infty}P^x_t b(x,\cdot)(i)$, $b\in C^{k}_{p}(\RR^n\times\mathbb{S},\RR^n)$, $Q\in C^k_p(\RR^n,\RR^{m_0}\otimes\RR^{m_0})$ and \eref{PartialP}, we can easily obtain $\bar{b}\in C^{k}_p(\RR^n,\RR^n)$.
Similarly, note that
	$$(\overline{\sigma\sigma^{\ast}})(x)
	=\sum_{i\in\mathbb{S}}[\sigma\sigma^{\ast}](x,i)\mu^x_{i}=\lim_{t\rightarrow \infty}P^x_t[\sigma\sigma^{\ast}](x,\cdot)(i) .$$
$\sigma\sigma^{\ast}\in C^{k}_p(\RR^n\times\mathbb{S},\RR^n\otimes\RR^n)$ and \eref{PartialP} imply $\overline{\sigma\sigma^{\ast}}\in C^{k}_p(\RR^n,\RR^n)$, which together with \eref{NonDbar}, we can get $\bar{\sigma}=(\overline{\sigma\sigma^{\ast}})^{1/2}\in C^k_p(\RR^n,\RR^n\otimes \RR^n)$ by the same argument as in \cite[Lemma A.7]{CLX}.
\end{remark}

\subsubsection{Tightness of $\{X^{\varepsilon}\}$ in $C([0,T];\RR^{n})$ }

\begin{proposition}\label{pro4.6}
 Suppose that  \ref{A1}-\ref{A4} hold. Then for $T>0$, the distribution of process $\{X^\varepsilon\}_{\varepsilon\in(0,1]}$ is tight in $C([0,T];\RR^{n})$.
\end{proposition}

\begin{proof}
In order to prove  $\{X^\varepsilon\}_{\varepsilon\in(0,1]}$ being tight in $C([0,T];\RR^{n})$, by using the tight criterion in $C([0,T];\RR^{n})$ (cf.~\cite[Theorem 7.3]{B1999}),  we only check the following conditions.

(i) For any $\eta>0$, there exist $K>0$ and $\varepsilon_0>0$ such that
\begin{equation}\label{T1}
\sup_{0<\varepsilon<\varepsilon_0}\mathbb{P}\big(|X^\varepsilon_0|\geq K\big)\leq\eta.
\end{equation}

(ii) For any $\theta>0,\eta>0$, there exist constants $\varepsilon_0>0,\delta>0$ such that
\begin{equation}\label{T2}
\sup_{0<\varepsilon<\varepsilon_0}\mathbb{P}\big(\sup_{t_1,t_2\in[0,T], |t_1-t_2|<\delta}|X^\varepsilon_{t_1}-X^\varepsilon_{t_2}|\geq \theta\big)\leq \eta.
\end{equation}

Since $X^\varepsilon_0=x$, \eref{T1} holds obviously. By \cite[P83, (7.12)]{B1999}, a sufficient condition that guarantee \eref{T2} holds is following: If for any positive constants $\theta,\eta$, there exist constants $\varepsilon_0,\delta$ such that for any $0\leq t\leq t+\delta\leq T$,
\begin{equation}\label{T3}
\sup_{0<\varepsilon<\varepsilon_0}\delta^{-1}\mathbb{P}\big(\sup_{t\leq s\leq t+\delta}|X_{s}^{\varepsilon}-X_{t}^{\varepsilon}|\geq \theta\big)\leq \eta.
\end{equation}
Hence, we only prove \eref{T3}. For any $\delta>0$, by \ref{A1} and \eref{X}, we have
\begin{eqnarray*}
\mathbb{E}\big(\sup_{t\leq s\leq t+\delta}|X_{s}^{\varepsilon}\!-\!X_{t}^{\varepsilon}|^{4}\big)
\!\leq\!\!\!\!\!\!\!\!\!&&C\EE\Big[\sup_{t\leq s\leq t+\delta}\Big|\int_t^{s}b(X_{r}^{\varepsilon},\alpha^{\varepsilon}_{r})dr\Big|^4\Big]
\!+\!C\EE\Big[\sup_{t\leq s\leq t+\delta}\Big|\int_t^{s}\sigma(X_{r}^{\varepsilon},\alpha^{\varepsilon}_{r})dW_r\Big|^4\Big]\\
\leq\!\!\!\!\!\!\!\!&&C \delta^4\sup_{r\in[0,T]}\mathbb{E}|b(X_{r}^{\varepsilon},\alpha^{\varepsilon}_{r})|^{4}
+C\mathbb{E}\Big(\int_t^{t+\delta} \|\sigma(X_{r}^{\varepsilon},\alpha^{\varepsilon}_{r})\|^2 dr\Big)^{2}\nonumber\\
\leq\!\!\!\!\!\!\!\!&&C \delta^{4}\big(1+ \sup_{r\in[0,T]}\EE|X_{r}^{\varepsilon}|^{4(k+1)}\big)+C \delta^{2}\big(1+ \sup_{r\in[0,T]}\EE|X_{r}^{\varepsilon}|^{4}\big)\nonumber\\
\leq\!\!\!\!\!\!\!\!&& C_{T}\left(1+|x|^{4(k+1)}\right) \delta^{2}.
\end{eqnarray*}
Chebyshev's inequality implies that, for any $\theta>0$,
$$
\delta^{-1}\PP\big(\sup_{t\leq s\leq t+\delta}|X_{s}^{\varepsilon}-X_{t}^{\varepsilon}|\geq \theta\big)\leq \frac{C_T\left(1+|x|^{4(k+1)}\right) \delta}{\theta^4}.
$$
This show that \eref{T3} holds. The proof is complete.
\end{proof}

\subsubsection{Proof of Theorem \ref{main result 1.2}}

We prove our second result in the averaging principle.

\vspace{-3mm}
~~\\
\textbf{Proof of Theorem \ref{main result 1.2}}  Using Proposition \ref{pro4.6}, it's sufficient to prove that, for any sequence $\varepsilon_k\to 0$, there is a subsequence (we will still keep denoting by $\varepsilon_{k}$) such that $X^{\varepsilon_k}$ converges weakly to $X$ (short for $X^{\varepsilon_k}{\overset{W}\longrightarrow} X$) in $C([0,T];\RR^{n})$, which is the unique solution of Eq.\eref{AR21}. In order to do this, we shall apply the martingale problem approach to characterize $X$.

Let $\Psi_{t_0}(\cdot)$ be a $\sigma\{\varphi_t, \varphi\in C([0,T],\RR^{n}), t\leq t_0\}$-measurable bounded continuous function on $C([0,T],\RR^{n})$. It is sufficient to prove that for any $t_0 \geq 0$, any $\Psi_{t_0}(\cdot)$ and $U\in C^3_b(\RR^{n})$, the following assertion holds:
\begin{eqnarray}
\EE\left[\left(U(X_t)-U(X_{t_0})-\int^t_{t_0}\bar{\mathcal{L}}U(X_s)ds\right)\Psi_{t_0}(X)\right]=0,\quad t\leq t_0\leq T,\label{F3.33}
\end{eqnarray}
where $\bar{\mathcal{L}}$ is the generator of Eq. \eref{AR21}, i.e.,
\begin{eqnarray}
\bar{\mathcal{L}}U(x):=\sum^n_{i=1}\partial_{x_i}U(x)\bar{b}_i(x)+\frac{1}{2}\sum^{n}_{i=1}\sum^n_{j=1}\partial_{x_i}\partial_{x_j} U(x)(\bar{\sigma}\bar{\sigma})_{ij}(x).\label{LD}
\end{eqnarray}

Applying It\^{o}'s formula, we have
\begin{eqnarray*}
U(X^{\varepsilon_k}_t)=\!\!\!\!\!\!\!\!&&U(X^{\varepsilon_k}_{t_0})+\int^t_{t_0}\langle\partial_x U(X^{\varepsilon_k}_s),b(X^{\varepsilon_k}_s, \alpha^{\varepsilon_k}_s)\rangle ds+\int^t_{t_0}\langle\partial_x U(X^{\varepsilon_k}_s),\sigma(X^{\varepsilon_k}_s, \alpha^{\varepsilon_k}_s)d W_s\rangle\nonumber\\
&&+\frac{1}{2}\int^t_{t_0}\text{Tr}\left[(\sigma\sigma^{\ast})(X^{\varepsilon_k}_s,\alpha^{\varepsilon_k}_s)\cdot\partial^2_x U(X^{\varepsilon_k}_s)\right]ds.
\end{eqnarray*}
Using $X^{\varepsilon_k}{\overset{W}\longrightarrow} X$ in $C([0,T];\RR^{n})$, it is easy to see
$$\lim_{k\to\infty} \EE\left\{\left[U(X^{\varepsilon_k}_t)-U(X^{\varepsilon_k}_{t_0})\right]\Psi_{t_0}(X^{\vare_k})\right\}
=\EE\left\{\left[U(X_t)-U(X_{t_0})\right]\Psi_{t_0}(X)\right\}$$
$$\lim_{k\to \infty} \EE\Big\{\Big[\int^t_{t_0}\langle\partial_x U(X^{\varepsilon_k}_s),\bar{b}(X^{\varepsilon_k}_s)\rangle ds\Big]\Psi_{t_0}(X^{\varepsilon_k})\Big\}=\EE\Big\{\Big[\int^t_{t_0}\langle\partial_x U(X_s),\bar{b}(X_s)\rangle ds\Big]\Psi_{t_0}(X)\Big\},$$
\begin{eqnarray*}
&&\lim_{k\to \infty} \EE\Big\{\Big[\int^t_{t_0}\text{Tr}\left(\bar{\sigma}\bar{\sigma}(X^{\varepsilon_k}_s)\cdot\partial^2_x U(X^{\varepsilon_k}_s)\right)ds\Big]\Psi_{t_0}(X^{\varepsilon_k})\Big\}\\
=\!\!\!\!\!\!\!\!&&\EE\Big\{\Big[\int^t_{t_0}\text{Tr}\left((\bar{\sigma}\bar{\sigma})(X_s)\cdot\partial^2_x U(X_s)\right)ds\Big]\Psi_{t_0}(X)\Big\}
\end{eqnarray*}
and
\begin{eqnarray*}
\EE\left\{\left[\int^t_{t_0}\langle\partial_x U(X^{\varepsilon_k}_s),\sigma(X^{\varepsilon_k}_s, \alpha^{\varepsilon_k}_s)d W_s\rangle\right]\Psi_{t_0}(X^{\vare_k})\right\}=0.
\end{eqnarray*}

In order to prove \eref{F3.33}, it is sufficient to prove that
\begin{eqnarray}
&&\lim_{k\to \infty} \EE\Big|\int^t_{t_0}\left[\langle\partial_x U(X^{\varepsilon_k}_s),b(X^{\varepsilon_k}_s, \alpha^{\varepsilon_k}_s)\rangle +\text{Tr}\left(\sigma\sigma^{\ast}(X^{\varepsilon_k}_s,\alpha^{\varepsilon_k}_s)\cdot\partial^2_x U(X^{\varepsilon_k}_s)\right)\right]ds\nonumber\\
&&\quad\quad\quad\quad -\int^t_{t_0}\left[\langle\partial_x U(X^{\varepsilon_k}_s),\bar{b}(X^{\varepsilon_k}_s)\rangle+\text{Tr}\left(\bar{\sigma}\bar{\sigma}(X^{\varepsilon_k}_s)\cdot\partial^2_x U(X^{\varepsilon_k}_s)\right)\right]ds\Big|=0.\label{F3.34}
\end{eqnarray}
Put $H(x,i):=G(x,i)-\bar{G}(x)$ with
$$G(x,i):=\langle\partial_x U(x),b(x, i)\rangle+\text{Tr}\left[(\sigma\sigma^{\ast})(x,i)\cdot\partial^2_x U(x)\right]$$
and
$$\bar{G}(x):=\langle\partial_x U(x),\bar b(x)\rangle+\text{Tr}\left[\bar{\sigma}\bar{\sigma}(x)\cdot\partial^2_x U(x)\right].$$
Then $H$ satisfies the "\emph{central condition}" obviously. Using \ref{A1} and $U\in C^3_b(\RR^{n})$, it is easy to check $H\in C^1_p(\RR^n\times\mathbb{S},\RR^n)$.
Applying the Poisson equation and mollifying argument as in the proof of Th.\ref{main result 1}, it's easy to prove \eref{F3.34}.

Finally, Lemma \ref{PMbarX} implies that \eref{AR21} has a unique strong solution. Hence, the weak solution is also unique.
Since the uniqueness of solutions to the martingale problem with the operator $\bar{\mathcal{L}}$ is equivalent to the uniqueness of weak solution of \eref{AR21}, the limiting process $X$ satisfying \eref{F3.33} is the unique solution of \eref{AR21} in the weak sense. The proof is complete. $\square$

\begin{remark}
Th.\ref{main result 1.2} is equivalent to
\begin{eqnarray*}
\lim_{\varepsilon\rightarrow \infty}|\mathbb{E}f(X^{\varepsilon})-\mathbb{E}f(\bar{X})|=0,\quad \forall f\in\mathbb{C},
\end{eqnarray*}
where $\mathbb{C}$ is a class of bounded continuous functions on $C([0,T],\RR^n)$.
The martingale problem method appears to be ineffective in attaining an convergence order. Consequently, an intriguing inquiry arises as to how one can attain a satisfactory convergence order in such situation.
\end{remark}

\subsection{Weak convergence of $X^{\varepsilon}_t$ in $\RR^n$ }
\vspace{0.1cm}
In this subsection, we study the weak convergence \eref{WeakA2}. The corresponding averaged equation is still Eq. \eref{AR21}. In order to achieve a satisfactory convergence order, we require an additional ${\color{blue} {\bf H}_{4}}$ .

Now, we consider the following Kolmogorov equation:
\begin{equation}\left\{\begin{array}{l}\label{KE1}
			\displaystyle
			\partial_t u(t,x)=\bar{\mathcal{L}} u(t,x),\quad t\in[0, T], \\
			u(0, x)=\phi(x),
		\end{array}\right.
\end{equation}
where $\phi\in C^{4}_p(\RR^{n})$ and $\bar{\mathcal{L}}$ is the infinitesimal generator of Eq.\eref{AR21}, which is defined in \eref{LD}. Obviously, \eref{KE1} admits a unique solution $u$ given by
$$
u(t,x)=\EE\phi(\bar{X}^x_t),\quad t\geq 0.
$$
	
Following the same argument in \cite[Lemma 5.1]{SSWX2023},  by direct calculation, we can get the following regularity estimates of the solution $u$ of Eq.\eref{KE1}.

\begin{lemma} \label{Lemma 5.1}
For unit vectors $l_1,l_2,l_3,l_4\in \RR^n$ and $T>0$, there exist $C_T>0$, $k>0$ such that for $x\in\RR^n$,
$$\sup_{0\leq t\leq T}|\partial_{x} u(t,x)\cdot l_1|\leq C_{T}\left(1+|x|^{k}\right),~~~\sup_{0\leq t\leq T}|\partial^2_{x} u(t,x)\cdot (l_1,l_2)|\leq C_T\left(1+|x|^{k}\right),$$
$$\sup_{0\leq t\leq T}|\partial^3_{x} u(t,x)\cdot (l_1,l_2,l_3)|\leq C_{T}\left(1+|x|^{k}\right),~~~\sup_{0\leq t\leq T}|\partial^4_{x} u(t,x)\cdot (l_1,l_2,l_3,l_4)|\leq C_{T}\left(1+|x|^{k}\right), $$
$$\sup_{0\leq t\leq T}|\partial_t(\partial_x u(t,x))\cdot l_1|\leq C_T \left(1+|x|^{k}\right),~~~\sup_{0\leq t\leq T}|\partial_t(\partial^2_x u(t,x))\cdot (l_1,l_2)|\leq C_T\left(1+|x|^{k}\right).$$
\end{lemma}

Now, we can prove our third result in the averaging principle.

\vspace{-3mm}
~~\\
\textbf{Proof of Theorem \ref{main result 1.3}}:
Fixed $t\in (0,T]$, let $\tilde{u}^t(s,x):=u(t-s,x)$, $0\le s\le t$, It\^{o}'s formula implies
	\begin{eqnarray*}
		\tilde{u}^t(t, X^{\vare}_t)=\!\!\!\!\!\!\!\!&&\tilde{u}^t(0,x)+\int^t_0 \partial_s \tilde{u}^t(s, X^{\vare}_s )ds+\int^t_0 \mathcal{L}_{1}\tilde{u}^t(s, \cdot)(\alpha^{\vare}_s,X^{\vare}_s) ds+\tilde{M}_t,
	\end{eqnarray*}
where
	\begin{eqnarray*}
	\mathcal{L}_{1}\phi(\alpha,x):=\langle b(x,\alpha), \nabla  \phi(x)\rangle+\frac{1}{2}\text{Tr}[\sigma\sigma^{*}(x,\alpha)\nabla^2\phi(x)]
	\end{eqnarray*}
and $\tilde{M}_t:=\int^t_0 \langle \partial_{x}\tilde{u}^t(s,X_{s}^{\vare}),\sigma(X_{s}^{\vare}, \alpha_{s}^{\vare})dW_{s}^{1} \rangle$  is a $\mathcal{F}_{t}$-martingale.
Note that $\tilde{u}^t(t, X^{\vare}_t)=\phi(X^{\vare}_t)$, $\tilde{u}^t(0, x)=\EE\phi(\bar{X}^{x}_t)$ and $\partial_s \tilde{u}^t(s, X^{\vare}_s )=-\bar{\mathcal{L}} \tilde{u}^t(s,\cdot)(X^{\vare}_s)$, we have
	\begin{eqnarray}
		\left|\EE\phi(X^{\vare}_{t})-\EE\phi(\bar{X}_{t})\right|
 =\!\!\!\!\!\!\!\!&&\left|\EE\int^t_0 -\bar{\mathcal{L}} \tilde{u}^t(s, \cdot)(X^{\vare}_s )ds+\EE\int^t_0 \mathcal{L}_{1}\tilde{u}^t(s, \cdot)(\alpha^{\vare}_s,X^{\vare}_s)ds\right|\nonumber\\
=\!\!\!\!\!\!\!\!&&\Big|\EE\!\int^t_0 \!\langle b(X^{\vare}_s,\alpha^{\vare}_s)\!-\!\bar{b}(X^{\vare}_s),\partial_x \tilde{u}^t(s, X^{\vare}_s )\rangle\nonumber\\
		&&\quad\quad\quad\quad +\frac{1}{2}\text{Tr}\big[(\sigma\sigma^{\ast}(X^{\vare}_s,\alpha^{\vare}_s)\!-\!\bar{\sigma}\bar{\sigma}(X^{\vare}_s))\partial_x^{2}\tilde{u}^t(s, X^{\vare}_s)\big]  ds\Big| .  \label{F5.11}
	\end{eqnarray}
	
For $s\in [0,t], x\in\RR^{n},i\in\mathbb{S}$, put
\begin{eqnarray*}
F^t(s,x,i):=\langle b(x,i)-\bar{b}(x), \partial_x \tilde{u}^t(s, x)\rangle+\frac{1}{2}\text{Tr}\big[\left(\sigma\sigma^{*}(x,i)-\bar{\sigma}\bar{\sigma}(x)\right)\partial_x^{2}\tilde{u}^t(s,x)\big].
	\end{eqnarray*}
Then $F^t(s,\cdot,\cdot)$ satisfies the "\emph{central condition}". Lemma \ref{Lemma 5.1} and ${\color{blue} {\bf H}_{4}}$  imply $F^t(s,\cdot,\cdot) \in C^{2}_{p}(\RR^n\times\mathbb{S},\RR^n)$ and
$$\sup_{s\in [0,t],i\in\mathbb{S}}|\partial_sF^t(s,x,i)|\leq C_T(1+|x|^k).$$
By Th.\ref{Poisson}, the Poisson equation
	\begin{eqnarray}
		-Q(x)\tilde{\Phi}^t(s,x,\cdot)(i)=F^t(s,x,i)\label{WPE1}
	\end{eqnarray}
admits a solution $\tilde{\Phi}^t$. Moreover, for $T>0$, $t\in [0,T]$, there exist $C_T, k>0$ such that the following estimates hold:
	\begin{eqnarray}
	&&\sup_{0\leq s\leq t}\left\{\|\tilde{\Phi}^t(s,x,\cdot)\|_{\infty}+\|\partial_s \tilde{\Phi}^t(s,x,\cdot)\|_{\infty}+\|\partial_x \tilde{\Phi}^t(s,x,\cdot)\|_{\infty}+\| \partial_x^{2}\tilde{\Phi}^t(s,x, \cdot)\|_{\infty}\right\}\nonumber\\
	&&\leq C_{T}(1+|x|^{k}). \label{WR1}
	\end{eqnarray}
	
Using It\^o's formula and taking expectation on both sides, we get
	\begin{eqnarray*}
		\EE\tilde{\Phi}^t(t, X_{t}^{\vare},\alpha^{\vare}_{t})=\!\!\!\!\!\!\!\!&&\tilde \Phi^t(0, x,\alpha)+\EE\int^t_0 \partial_s \tilde{\Phi}^t(s, X_{s}^{\vare},\alpha^{\vare}_{s})ds\\
		&&	+\EE\int^t_0\mathcal{L}_{1}\tilde\Phi^t(s, \cdot,\alpha^{\vare}_{s})(\alpha^{\vare}_{s},X_{s}^{\vare})ds
		+\frac{1}{\vare}\EE\int^t_0 Q(X_{s}^{\vare})\tilde{\Phi}^t(s, X_{s}^{\vare},\cdot)(\alpha^{\vare}_{s})ds,
	\end{eqnarray*}
which implies
	\begin{eqnarray}
	-\EE\int^t_0 Q(X_{s}^{\vare})\tilde\Phi^t(s, X_{s}^{\vare},\cdot)(\alpha^{\vare}_{s})ds
		\!=\!\!\!\!\!\!\!\!\!&&\vare\Big[\tilde{\Phi}^t(0, x,\alpha)\!-\!\EE\tilde{\Phi}^t(t, X_{t}^{\vare},\alpha^{\vare}_{t})\!+\!\EE\!\int^t_0\!\! \partial_s \tilde{\Phi}^t(s, X_{s}^{\vare},\cdot)(\alpha^{\vare}_{s})ds\nonumber\\
		&&~~~+\EE\int^t_0\mathcal{L}_{1}\tilde{\Phi}^t(s, \cdot,\alpha^{\vare}_{s})(\alpha^{\vare}_{s},X_{s}^{\vare})ds\Big].\label{F3.391}
	\end{eqnarray}
\eref{F5.11}, \eref{WPE1} and \eref{F3.391} imply that
	\begin{eqnarray*}
		&&\sup_{0\leq t\leq T}\left|\EE\phi(X^{\vare}_{t})-\EE\phi(\bar{X}_{t})\right|\\
=\!\!\!\!\!\!\!\!&&\sup_{0\leq t\leq T}\left|\EE\int^t_0 \left[F^t(s,X_{s}^{\vare},\alpha^{\vare}_{s})\right]ds\right|\\
=\!\!\!\!\!\!\!\!&&\sup_{0\leq t\leq T}\left|\EE\int^t_0 Q(X_{s}^{\vare})\tilde{\Phi}^t(s, X_{s}^{\vare},\cdot)(\alpha^{\vare}_{s})ds\right|\\
\leq\!\!\!\!\!\!\!\!&&\vare\Big[\sup_{0\leq t\leq T}|\tilde{\Phi}^t(0, x,\alpha)|+\sup_{0\leq t\leq T}\left|\EE\tilde{\Phi}^t(t, X_{t}^{\vare},\alpha^{\vare}_{t})\right|+\sup_{0\leq t\leq T}\EE\int^t_0\left|\partial_s \tilde{\Phi}^t(s, X_{s}^{\vare},\alpha^{\vare}_{s})\right|ds.\\
		&&~~~~ +\sup_{0\leq t\leq T}\EE\int^t_0\left|\mathcal{L}_{1}\tilde{\Phi}^t(s, \cdot,\alpha^{\vare}_{s})(\alpha^{\vare}_{s},X_{s}^{\vare})\right|ds\Big].
	\end{eqnarray*}
Finally, using \eref{WR1}, we obtain that, for some $k>0$,
	\begin{eqnarray*}
		\sup_{0\leq t\leq T}\left|\EE\phi(X^{\vare}_{t})-\EE\phi(\bar{X}_{t})\right|\leq C_{T}\left(1+|x|^{k}\right)\vare.
	\end{eqnarray*}
	The proof is complete. $\square$

\section{Proofs of central limit theorem}

In this section, we will study CLT of Eq.\eref{e:eqn2}. Specifically, we will provide two convergences.  (1) weak convergence of $(X^{\varepsilon}-\bar{X})/\sqrt{\varepsilon}$ in $C([0,T],\RR^n)$. (2) weak convergence of $(X_t^{\varepsilon}-\bar{X}_t)/\sqrt{\varepsilon}$ in $\RR^n$ for any fixed $t>0$. These convergences are based on Th.\ref{main result 1}, which requires $\sigma(x,i)=\sigma(x)$.

\subsection{Weak convergence of $(X^{\varepsilon}-\bar{X})/\sqrt{\varepsilon}$ in $C([0,T],\RR^n)$}

In this subsection, we study the weak convergence \eref{WeakA1} of $X^{\varepsilon}$ in $C([0,T],\RR^n)$.

\subsubsection{Construction of the auxiliary process} \label{S4.1}
Eq.\eref{e:eqn2} and Eq.\eref{AR1} imply that  the deviation process $Z_t^\varepsilon:=(X^{\varepsilon}_t-\bar{X}_t)/{\sqrt{\varepsilon}}$ satisfies the equation:
\begin{equation*}
\left\{ \begin{aligned}
dZ_t^\varepsilon=&\frac{1}{\sqrt{\varepsilon}}\left[\left(b(X^{\varepsilon}_t,\alpha^{\varepsilon}_t)-\bar{b}(X^{\varepsilon}_t)\right)
+\left(\bar{b}(X^{\varepsilon}_t)-\bar{b}(\bar{X}_t)\right)\right]dt+\frac{1}{\sqrt{\varepsilon}}\left(\sigma(X^{\varepsilon}_t)-\sigma(\bar{X}_t)\right)dW_t,\\
Z_0^\varepsilon=&~0.
\end{aligned} \right.
\end{equation*}
For $p>0,T>0$,  (\ref{main 1}) implies that there exist constants $C_{p,T}>0, k_p>0$ such that
\begin{eqnarray}
\sup_{\epsilon\in(0,1]}\EE\big(\sup_{0\leq t\leq T}|Z_t^\varepsilon|^p\big)\leq C_{p,T}\left(1+|x|^{k_p}\right).\label{supZvare}
\end{eqnarray}

Let auxiliary process $\eta^\varepsilon$ satisfy the linear equation:
\begin{equation}\label{e6}
\left\{ \begin{aligned}
d\eta_t^\varepsilon=&~\frac{1}{\sqrt{\varepsilon}}\left(b(X^{\varepsilon}_t,\alpha^{\varepsilon}_t)-\bar{b}(X^{\varepsilon}_t)+\sqrt{\varepsilon}\nabla\bar{b}(\bar{X}_t)\cdot\eta_t^\varepsilon \right)dt+\left(\nabla\sigma(\bar{X}_t)\cdot\eta_t^\varepsilon\right) dW_t,\\
\eta_0^\varepsilon=&~0.
\end{aligned} \right.
\end{equation}
We will study the limitation of the difference process $Z^\varepsilon-\eta^\varepsilon$ in $C([0,T];\RR^n)$ for $T>0$.

\begin{proposition}\label{pro5.1}
Suppose that $\sigma(x,i)=\sigma(x)$ and \ref{A1}-\ref{A3} and ${\color{blue} {\bf H}_{2}}$  hold. Then for $T>0$, there exist $C_T,k>0$ such that
\begin{eqnarray}
\EE\big(\sup_{0\leq t\leq T}|Z_t^\varepsilon-\eta_t^\varepsilon|^2\big)\leq C_T(1+|x|^{k})\varepsilon.\label{F4.2}
\end{eqnarray}
\end{proposition}

\begin{proof} ~\eref{barc1} and \eref{ConA14} imply that
\begin{eqnarray}
\langle x,\nabla\bar{b}(x)\cdot x\rangle\leq C|x|^2,\quad \|\nabla\sigma(x)\cdot x\|\leq C|x|^2.\label{partialbarb}
\end{eqnarray}
Put $\rho^{\vare}_t:=Z_t^\varepsilon-\eta_t^\varepsilon$. We have
\begin{eqnarray*}
\rho^{\vare}_t=\!\!\!\!\!\!\!\!&&\frac{1}{\sqrt{\varepsilon}}\int_0^t\left[\bar{b}(X^{\varepsilon}_s)-\bar{b}(\bar{X}_s)-\nabla \bar{b}(\bar{X}_s)\cdot\sqrt{\varepsilon}Z_s^\varepsilon\right]ds+\int_0^t\nabla\bar{b}(\bar{X}_s)\cdot \rho^{\vare}_s ds\\
&&+\frac{1}{\sqrt{\varepsilon}}\int_0^t\left[\sigma(X^{\varepsilon}_s)-\sigma(\bar{X}_s)-\nabla \sigma(\bar{X}_s)\cdot\sqrt{\varepsilon}Z_s^\varepsilon\right]dW_s+\int_0^t\left(\nabla\sigma(\bar{X}_s)\cdot \rho^{\vare}_s \right)dW_s.
\end{eqnarray*}
By It\^{o}'s formula, BDG inequality and Young's inequality, we get
\begin{eqnarray*}
\EE\big(\sup_{0\leq t\leq T}|\rho^{\vare}_t|^2\big)\leq\!\!\!\!\!\!\!\!&&\frac{C_T}{\varepsilon}\int_0^T\EE\left|\bar{b}(X^{\varepsilon}_s)-\bar{b}(\bar{X}_s)-\nabla \bar{b}(\bar{X}_s)\cdot\sqrt{\varepsilon}Z_s^\varepsilon\right|^2ds\\
&&+\frac{C}{\varepsilon}\EE\int_0^T\left\|\sigma(X^{\varepsilon}_s)-\sigma(\bar{X}_s)-\nabla \sigma(\bar{X}_s)\cdot\sqrt{\varepsilon}Z_s^\varepsilon\right\|^2ds\\
&&+C\EE\int_0^T|\rho^{\vare}_s|^2ds+\frac{1}{2}\EE\big(\sup_{0\leq t\leq T}|\rho^{\vare}_t|^2\big)\\
=:\!\!\!\!\!\!\!\!&&J^{\varepsilon}_1(T)+J^{\varepsilon}_2(T)+C_T\int_0^T\EE|\rho^{\vare}_s|^2 ds+\frac{1}{2}\EE\big(\sup_{0\leq t\leq T}|\rho^{\vare}_t|^2\big).
\end{eqnarray*}
Gronwall's inequality gives that
\begin{eqnarray}
\EE\big(\sup_{0\leq t\leq T}|\rho^{\vare}_t|^2\big)\leq C_T\left(J^{\varepsilon}_1(T)+J^{\varepsilon}_2(T)\right).\label{J0}
\end{eqnarray}

For $J^{\varepsilon}_1(T)$,  by Remark \ref{HighR}, ${\color{blue} {\bf H}_{2}}$ implies $\bar{b}\in C^2_p(\RR^n,\RR^n)$. Using Taylor's formula, there exists $k>0$ and $r\in (0,1)$ such that
\begin{eqnarray}
J^{\varepsilon}_1(T)
=\!\!\!\!\!\!\!\!&&\frac{C_T}{\varepsilon}\int_0^T\EE\left|\bar{b}(\bar{X}_s+\sqrt{\varepsilon}Z_s^\varepsilon)-\bar{b}(\bar{X}_s)-\partial_x \bar{b}(\bar{X}_s)\cdot\sqrt{\varepsilon}Z_s^{\varepsilon}\right|^2ds\nonumber\\
=\!\!\!\!\!\!\!\!&&\frac{C_T}{\varepsilon}\int_0^T\EE\left|\int^1_0\nabla^2 \bar{b}(\bar{X}_s+r\sqrt{\varepsilon}Z_s^\varepsilon)\cdot(\sqrt{\varepsilon}Z_s^{\varepsilon},\sqrt{\varepsilon}Z_s^{\varepsilon})dr\right|^2ds\nonumber\\
\leq\!\!\!\!\!\!\!\!&&C_T\varepsilon\int_0^T\EE\left[(1+|\bar{X}_s|^{k}+|Z_s^{\varepsilon}|^{k})|Z_s^{\varepsilon}|^4\right]ds\nonumber\\
\leq\!\!\!\!\!\!\!\!&&C_T(1+|x|^{k})\varepsilon.\label{J1}
\end{eqnarray}
Similar argument for $J^{\varepsilon}_2(T)$, it follows
\begin{eqnarray}
J^{\varepsilon}_2(T)\leq C_T(1+|x|^{k})\varepsilon. \label{J2}
\end{eqnarray}
Combining \eref{J0}-\eref{J2}, we have
$$
\EE\big(\sup_{0\leq t\leq T}|\rho^{\vare}_t|^2\big)\leq C_T(1+|x|^{k})\varepsilon.
$$

The proof is complete. 
\end{proof}

\subsubsection{Tightness of $\eta^\varepsilon$}\label{S4.2}

In this subsection, we study the tightness of the solution $\eta^\varepsilon$ of Eq. (\ref{e6}) in $C([0,T],\RR^n)$ as $\varepsilon\to0$.

\vspace{0.2cm}
For $0\leq t\leq T$, put
$$\eta^\varepsilon(t)=:I_1^\varepsilon(t)+I_2^\varepsilon(t)+I_3^\varepsilon(t)$$
with
$$ I_1^\varepsilon(t)=\frac{1}{\sqrt{\varepsilon}}\!\int_{0}^t\!\left(b(X^{\varepsilon}_s,\alpha^{\varepsilon}_s)\!-\!\bar{b}(X^{\varepsilon}_s)\right)ds,~I_2^\varepsilon(t)=\!\int_{0}^t\!\nabla\bar{b}(\bar{X}_s)\cdot \eta_s^\varepsilon ds,~I_3^\varepsilon(t)=\!\int_{0}^t\!\nabla\sigma(\bar{X}_s)\cdot\eta_s^\varepsilon dW_s.$$

\begin{proposition}\label{pro3.2}
Suppose that $\sigma(x,i)=\sigma(x)$ and \ref{A1}-\ref{A3} and ${\color{blue} {\bf H}_{2}}$ hold. Then for $T>0$, $\eta^\varepsilon$ is tight in $C([0,T];\RR^{n})$ as $\varepsilon\to0$.
\end{proposition}

\begin{proof}
 For technical reasons, we can not use the criterion \eref{T1} and \eref{T2} of tightness in $C([0,T];\RR^{n})$ directly. We just have to prove the tightness of $\{\eta^{\varepsilon}\}_{\varepsilon\le 1}$ in $\mathbb{D}([0,T];\RR^{n})$ firstly, where $\mathbb{D}([0,T];\RR^{n})$ is the set of $c\grave{a}dl\grave{a}g$ $\RR^n$-valued functions on $[0,T]$. Path continulity  of $\eta^{\varepsilon}$ shows that $\{\eta^{\varepsilon}\}_{\varepsilon<1}$ is also tight in  $C([0,T];\RR^{n})$.

By the tight criterion in $\mathbb{D}([0,T];\RR^{n})$ (cf.~\cite[Chapter VI, Theorem 4.5]{JS}), we only prove that $\{\eta^\varepsilon\}_{\varepsilon<1}$ satisfies the followings.

(i) For $\delta>0$, there exist $\varepsilon_0>0$ and $\lambda>0$ such that
\begin{equation}\label{t1}
\sup_{0<\varepsilon\leq \varepsilon_0}\mathbb{P}\big(\sup_{0\leq t\leq T}|\eta^\varepsilon_t|\geq \lambda\big)\leq \delta.
\end{equation}

(ii) For $\delta>0$ and $K>0$, there exist $\theta>0$ and $\varepsilon_0>0$ such that
\begin{equation}\label{t2}
\sup_{0<\varepsilon\leq \varepsilon_0}\sup_{S_1,S_2\in \mathcal{T}^{T} , S_1\leq S_2\leq S_1+\theta}\mathbb{P}\big(|\eta^\varepsilon_{S_1}-\eta^\varepsilon_{S_2}|\geq K\big)\leq \delta,
\end{equation}
where $ \mathcal{T}^{T}=\{S:S{\rm \text ~is}~\mathcal{F}_t-{\rm \text stopping~time~with}~S\le T\}$.

\textbf{Step 1:}  We check \eref{t1}. Since $F=b(x,i)-\bar{b}(x)$ satisfies the "\emph{central condition}" \eref{CenCon},  ${\color{blue} {\bf H}_{2}}$ implies that $F\in C^2_p(\RR^n\times\mathbb{S},\RR^n)$. Th.\ref{Poisson} yields that \eref{PEQ} admits a solution $\Phi(x,i)$. Moreover, there exist constants $C>0,k>0$ such that, for $x\in \RR^n$,
\begin{eqnarray}
\|\partial_{x}\Phi(x,\cdot)\|_{\infty}+\|\partial^2_{x}\Phi(x,\cdot)\|_{\infty}\le C(1+|x|^k).\label{E1}
\end{eqnarray}
By It\^{o}'s formula, we obtain
\begin{eqnarray*}
\Phi(X_t^{\varepsilon},\alpha_t^{\varepsilon})=\!\!\!\!\!\!\!\!\!&&\Phi(x,\alpha)+\int^t_0 \mathcal{L}_{1}\Phi(\cdot,\alpha^{\varepsilon}_{s})(\alpha^{\vare}_{s},X_{s}^{\varepsilon})ds\\
&&+\frac{1}{\varepsilon}\int_0^tQ(X_{s}^{\varepsilon})\Phi(X_{s}^{\varepsilon},\cdot)(\alpha_s^{\varepsilon})ds+M^{1,\varepsilon}_t+M^{2,\varepsilon}_t,
\end{eqnarray*}
where $\mathcal{L}_{1}\Phi(\cdot,\alpha)(\alpha,x):=(\mathcal{L}_{1}\Phi_1(\cdot,\alpha)(\alpha,x),\ldots, \mathcal{L}_{1}\Phi_n(\cdot,\alpha)(\alpha,x))$
with
\begin{eqnarray*}\label{inf2}
 \mathcal{L}_{1}\Phi_k(x,\alpha):=\left\langle b(x,\alpha), \partial_x \Phi_k(x,\alpha)\right \rangle+\frac{1}{2}\text{Tr}\left[\partial^2_{x} \Phi_k(x,\alpha)\sigma\sigma^{*}(x)\right ],\quad k=1,\ldots, n
\end{eqnarray*}
and $M^{1,\varepsilon}_t=\int_0^t\!\int_{[0,\infty)}[ \Phi(X_{s-}^{\varepsilon},\alpha^{\varepsilon}_{s-}+g^{\varepsilon}(X_{s-}^{\varepsilon},\alpha_{s-}^{\varepsilon},z))-\Phi(X_{s-}^{\varepsilon}\!,\alpha_{s-}^{\varepsilon})]\tilde{N}(ds,dz).$ and $M^{2,\varepsilon}_t=\int_0^t \left[\partial_x\Phi(X_{s}^{\varepsilon},\alpha_s^{\varepsilon})\cdot \sigma(X_{s}^{\varepsilon})\right]dW_s$ are local martingales.
Then it follows that
\begin{eqnarray*}\label{4}
I_1^\varepsilon(t)=\!\!\!\!\!\!\!\!&&\sqrt{\varepsilon}\Big\{-\Phi(X_{t}^{\varepsilon},\alpha^{\varepsilon}_{t})+\Phi(x,\alpha)+\int^t_0 \mathcal{L}_{1}\Phi(\cdot,\alpha^{\varepsilon}_{s})(\alpha^{\varepsilon}_{s},X_{s}^{\varepsilon})ds+M^{2,\varepsilon}_t\Big\}+\sqrt{\varepsilon}M^{1,\varepsilon}_t\nonumber\\
=:\!\!\!\!\!\!\!\!&&I_{11}^\varepsilon(t)+I_{12}^\varepsilon(t).
\end{eqnarray*}

For any $p>0$, BDG inequality, (\ref{E1}) and (\ref{X}) imply that there exists $k_p>0$ such that
\begin{eqnarray}
\EE\big(\sup_{t\in[0,T]}|I_{11}^\varepsilon(t)|^p\big)\leq C_T\varepsilon^{p/2}\Big[1+\EE\big(\sup_{t\in[0,T]}|X^{\varepsilon}_t|^{k_p}\big)\Big]\le C_T\varepsilon(1+|x|^{k_p}).\label{Vanish}
\end{eqnarray}
Using Kunita's first inequality (see \cite[Theorem 4.4.23]{A}), \eref{E1} and condition \eref{Finte K}, we have, for  $p\geq 2$,
\begin{eqnarray}
&&\mathbb{E}\big(\sup_{0\leq t\leq T}\left|I_{12}^\varepsilon(t)\right|^p\big)\nonumber\\
\le\!\!\!\!\!\!\!\!&&C_p\varepsilon^{p/2}\mathbb{E}\Big[\int^T_0\int_{[0,\infty)}\left|\Phi(X^{\varepsilon}_{s},\alpha^{\varepsilon}_{s}+g^{\varepsilon}(X^{\varepsilon}_{s},\alpha^{\varepsilon}_{s},z))-\Phi(X^{\varepsilon}_{s},\alpha^{\varepsilon}_{s})\right|^2 dzds\Big]^{p/2}\nonumber\\
&&+C_p\varepsilon^{p/2}\mathbb{E}\int^T_0\int_{[0,\infty)}\left|\Phi(X^{\varepsilon}_{s},\alpha^{\varepsilon}_{s}+g^{\varepsilon}(X^{\varepsilon}_{s},\alpha^{\varepsilon}_{s},z))-\Phi(X^{\varepsilon}_{s},\alpha^{\varepsilon}_{s})\right|^p dzds\nonumber\\
\le\!\!\!\!\!\!\!\!&&C_p\varepsilon^{p/2}\mathbb{E}\Big[\int^T_0\!\!\left(1\!+\!|X^{\varepsilon}_{s}|^{k_p}\right)\int_{[0,K(X^{\varepsilon}_{s})\varepsilon^{-1}]} \!\!\!dzds\Big]^{p/2}\!\!\!+\!C_p\varepsilon^{p/2}\mathbb{E}\int^T_0\!\!\left(1+|X^{\varepsilon}_{s}|^{k_p}\right)\!\!\int_{[0,K(X^{\varepsilon}_{s})\varepsilon^{-1}]}\!\! \!dzds\nonumber\\
\le\!\!\!\!\!\!\!\!&&C_p\mathbb{E}\Big[\int^T_0\left(1+|X^{\varepsilon}_{s}|^{k_p})(1+|X^{\varepsilon}_{s}|^{k_p}\right)ds\Big]^{p/2}\!\!\!+C_p\mathbb{E}\int^T_0\left(1+|X^{\varepsilon}_{s}|^{k_p}\right)\left(1+|X^{\varepsilon}_{s}|^{k_p}\right)ds\nonumber\\
\le\!\!\!\!\!\!\!\!&&C_{p,T}\left(1+|x|^{k_p}\right).\label{S5.7}
\end{eqnarray}

By \eref{Vanish} and \eref{S5.7}, we get
\begin{eqnarray}\label{6}
\sup_{\varepsilon\in(0,1)}\EE\big(\sup_{t\in[0,T]}|I_1^\varepsilon(t)|^p\big)\leq C_T\left(1+|x|^{k_p}\right).
\end{eqnarray}

It\^{o}'s formula implies that
\begin{eqnarray*}
|\eta^\varepsilon_t|^p=\!\!\!\!\!\!&&\frac{p}{\sqrt{\varepsilon}}\int_0^t|\eta^\varepsilon_s|^{p-2}\left\langle \eta^\varepsilon_s,b(X^{\varepsilon}_s,\alpha^{\varepsilon}_s)-\bar{b}(\bar{X}_s)\right\rangle ds+p\int_0^t|\eta^\varepsilon_s|^{p-2}\left\langle \eta^\varepsilon_s,\nabla\bar{b}(\bar{X}_s)\cdot\eta_s^\varepsilon\right\rangle ds\\
&&+p\int_0^t|\eta^\varepsilon_s|^{p-2}\left\langle \eta^\varepsilon_s,\nabla\sigma(\bar{X}_s)\cdot\eta_s^\varepsilon dW_s\right\rangle+{\frac{p}{2}}\int_0^t|\eta^\varepsilon_s|^{p-2} \|\nabla\sigma(\bar X_s)\cdot \eta^\varepsilon_s\|^2ds\\
&&+{\frac{p(p-2)}{2}}\int_0^t|\eta^\varepsilon_s|^{p-4}\big|(\nabla\sigma(\bar X_s)\cdot \eta^\varepsilon_s)^{*}\cdot\eta^\varepsilon_s \big|^{2}ds.
\end{eqnarray*}
Then by \eref{partialbarb}, \eref{6}, BDG inequality and Young's inequality, we get
\begin{eqnarray*}
\EE\big(\sup_{0\leq t\leq T}|\eta^\varepsilon_t|^p\big)\leq\!\!\!\!\!\!\!\!&&C_p\EE\big(\sup_{0\leq t\leq T}|I_1^\varepsilon(t)|^{k_p}\big)+C_{p}\EE\int_{0}^T|\eta_s^\varepsilon|^p ds
+\frac{1}{2}\EE\big(\sup_{0\leq t\leq T}|\eta^\varepsilon_t|^p\big)\\
\leq\!\!\!\!\!\!\!\!&&C_{p,T}\left(1+|x|^{k_p}\right)+C_{p}\EE\int_{0}^T|\eta_s^\varepsilon|^p ds+\frac{1}{2}\EE\big(\sup_{0\leq t\leq T}|\eta^\varepsilon_t|^p\big).
\end{eqnarray*}
This implies that
\begin{eqnarray}\label{5}
\EE\big(\sup_{0\leq t\leq T}|\eta^\varepsilon_t|^p\big)\leq C_{p,T}\left(1+|x|^{k_p}\right).
\end{eqnarray}
Hence, it is easy to get that $\{\eta^\varepsilon\}$ satisfies condition \eref{t1} by Chebyshev's inequality.

\textbf{Step 2:} We check \eref{t2}. Because it's easy to deduce that $I^\varepsilon_{11}$ satisfies \eref{t2} from \eref{Vanish}, we just have to prove that $J^{\varepsilon}_t:=I^\varepsilon_{12}(t)+I^\varepsilon_{2}(t)+I^\varepsilon_{3}(t)$ satisfies condition \eref{t2}.

In fact, for any $S_1,S_2\in \mathcal{T}^{T}$ with $S_1\leq S_2\leq S_1+\theta$, $\bar{b}\in C^{1}_{p}(\RR^n,\RR^n)$, \eref{E1} and \eref{5} yield
\begin{eqnarray*}
\EE|J^{\varepsilon}_{S_1}-J^{\varepsilon}_{S_2}|^2
\leq\!\!\!\!\!\!\!\!&&C\varepsilon\mathbb{E}\Big|\int^{S_2}_{S_1}\int_{[0,\infty)}\Phi(X^{\varepsilon}_{s},\alpha^{\varepsilon}_{s}+g^{\varepsilon}(X^{\varepsilon}_{s},\alpha^{\varepsilon}_{s},z))-\Phi(X^{\varepsilon}_{s},\alpha^{\varepsilon}_{s}) \tilde{N}(ds,dz)\Big|^2\nonumber\\
&&+C\mathbb{E}\Big|\int^{S_2}_{S_1}\nabla\bar{b}(\bar{X}_s)\cdot \eta_s^\varepsilon ds\Big|^2+C\mathbb{E}\Big|\int_{S_1}^{S_2}\nabla\sigma(\bar{X}_s)\cdot\eta_s^\varepsilon dW_s\Big|^2\\
\leq\!\!\!\!\!\!\!\!&&C\varepsilon\EE\int^{S_2}_{S_1}\int_{[0,\infty)}|\Phi(X^{\varepsilon}_{s},\alpha^{\varepsilon}_{s}+g^{\varepsilon}(X^{\varepsilon}_{s},\alpha^{\varepsilon}_{s},z))-\Phi(X^{\varepsilon}_{s},\alpha^{\varepsilon}_{s})|^2 dzds\nonumber\\
&&+C\theta\mathbb{E}\int^{S_2}_{S_1}|\nabla\bar{b}(\bar{X}_s)\cdot \eta_s^\varepsilon|^2 ds+C\mathbb{E}\int_{S_1}^{S_2}\|\nabla\sigma(\bar{X}_s)\cdot\eta_s^\varepsilon\|^2 ds\\
\leq\!\!\!\!\!\!\!\!&&C\mathbb{E}\int^{S_2}_{S_1}(1+|X^{\varepsilon}_{s}|^k)ds+C\theta\mathbb{E}\int^{S_2}_{S_1}(1+|\bar{X}_s|^k)| \eta_s^\varepsilon|^2 ds+C\mathbb{E}\int_{S_1}^{S_2}|\eta_s^\varepsilon|^2 ds\nonumber\\
\leq\!\!\!\!\!\!\!\!&&C\theta\EE\big(\sup_{s\in [0,T]}(1\!+\!|X^{\varepsilon}_{s}|^k)\big)\!+\!C(\theta^2\!+\!\theta)\!\Big[\EE\big(\sup_{0\leq s\leq T}|\bar{X}_s|^{2k}\big)\Big]^{1/2}\!\Big[\EE\big(\sup_{0\leq s\leq T}|\eta_s^\varepsilon|^4\big)\Big]^{1/2}\\
\leq\!\!\!\!\!\!\!\!&&C_T(\theta^2+\theta)\left(1+|x|^{k}\right).
\end{eqnarray*}
Chebyshev's inequality implies that $\{J^{\varepsilon}\}$ satisfies \eref{t2}. The proof is complete.
\end{proof}

Now we prove our first result in the central limit theorem.

\vspace{-3mm}
~~\\
\textbf{Proof of Theorem \ref{main result 4}:}
Note that \eref{main 1} gives that $X^{\varepsilon}\to\bar{X}$ (as $\varepsilon\to0$) in the space $C([0,T];\RR^{n})$  $\PP$-a.s., then by Proposition \ref{pro3.2}, the family $\{(\bar{X},X^{\varepsilon},\eta^{\varepsilon})\}_{0<\varepsilon\leq 1}$ is tight $C([0,T];\RR^{3n})$ for any $T>0$ (see \cite[Chapter VI, Corollary 3.33]{JS}). Hence, we just have to show that for any sequence $\varepsilon_k\to0$, there is the subsequence (for the sake of writing, we still write as $\varepsilon_{k}$) such that $\{(\bar{X},X^{\varepsilon_{k}},\eta^{\varepsilon_{k}})\}_{k\ge 1}$ weakly converges to the unique solution of equations $(\bar{X}, \bar{X}, Z)$ of \eref{AR1} and \eref{limE}  in $C([0,T];\RR^{3n})$ as $k\to\infty$. To to this, we will use the martingale problem approach to determine the limiting process $(\bar{X}, Z)$. The detailed proof is divided into three steps.

\textbf{Step 1:} Let $\Psi_{t_0}(\cdot)$ be a $\sigma\{\varphi_t, \varphi\in C([0,T],\RR^{2n}), t\leq t_0\}$-measurable bounded continuous function on $C([0,T],\RR^{2n})$. It is sufficient to prove that, for $0\leq t_0 \leq T$, $\Psi_{t_0}(\cdot)$ and $U\in C^4_b(\RR^{2n})$, the following assertion holds:
\begin{eqnarray}
\EE\Big[\Big(U(\bar X_t, Z_t)-U(\bar X_{t_0}, Z_{t_0})-\int^t_{t_0}\mathcal{L}U(\bar X_s,Z_s)ds\Big)\Psi_{t_0}(\bar X,Z)\Big]=0,\quad t_0\leq t\leq T,\label{Mar11}
\end{eqnarray}
where $\mathcal{L}$ is the generator of Eq. \eref{AR1} and Eq. \eref{limE}, i.e.,
\begin{eqnarray}
\mathcal{L}U(x,z):=\sum^4_{i=1}\mathcal{L}_iU(x,z),\quad x,z\in\RR^n,\label{Generator}
\end{eqnarray}
with
$$\mathcal{L}_1U(x,z):=\sum^n_{i=1}\partial_{x_i}U(x,z)\bar{b}_i(x)+\frac{1}{2}\sum^{n}_{i=1}\sum^n_{j=1}\partial_{x_i}\partial_{x_j} U(x,z)(\sigma\sigma^{\ast})_{ij}(x),$$
$$\mathcal{L}_2U(x,z)\!:=\!\sum^n_{i=1}\!\partial_{z_i}U(x,z)(\nabla\bar{b}_i(x)\cdot z)\!+\!\frac12\!\sum^{n}_{i=1}\sum^n_{j=1}\!\partial_{z_i}\partial_{z_j} U(x,z)\left[(\nabla\sigma\cdot z)(\nabla\sigma\cdot z)^{\ast}\right]_{ij}(x),$$
$$\mathcal{L}_3U(x,z):=\sum^{n}_{i=1}\sum^n_{j=1}\partial_{x_i}\partial_{z_j} U(x,z)\left[\sigma(\nabla\sigma\cdot z)^{\ast}\right]_{ij}(x),$$
$$\mathcal{L}_4U(x,z):=\frac{1}{2}\sum^{n}_{i=1}\sum^n_{j=1}\partial_{z_i}\partial_{z_j} U(x,z)(\Theta\Theta^{\ast})_{ij}(x).$$

For the solution $\Phi$ of Poisson equation (\ref{PEQ}), Th.\ref{Poisson} and ${\color{blue} {\bf H}_{2}}$ imply $\Phi\in C^2_p(\RR^n\times\mathbb{S},\RR^n)$. It\^{o}'s formula gives
\begin{eqnarray*}
d\Phi(X_{t}^{\varepsilon_k},\alpha_{t}^{\varepsilon_k})=\!\!\!\!\!\!\!\!\!&&\partial_x\Phi(X_{t}^{\varepsilon_k},\alpha_{t}^{\varepsilon_k})\cdot b(X_{t}^{\varepsilon_k},\alpha_{t}^{\varepsilon_k})dt+\partial_x\Phi(X_{t}^{\varepsilon_k},\alpha_{t}^{\varepsilon_k})\cdot \sigma(X_{t}^{\varepsilon_k})dW_t\\
&&+\frac{1}{\varepsilon_k}Q(X_{t}^{\varepsilon_k})\Phi(X_{t}^{\varepsilon_k},\cdot)( \alpha_t^{\varepsilon_k})dt+\frac{1}{2}\text{Tr}[\partial_x^{2}\Phi(X_{t}^{\varepsilon_k},\alpha_{t}^{\varepsilon_k})\cdot \left(\sigma\sigma^{*}\right)(X_{t}^{\varepsilon_k})]dt\\
&&+\int_{[0,\infty)}\left( \Phi(X_{t}^{\varepsilon_k},\alpha_{t-}^{\varepsilon_k}+\!g^{\varepsilon}(X_{t}^{\varepsilon_k},\alpha_{t-}^{\varepsilon_k},z))-\Phi(X_{t}^{\varepsilon_k},\alpha_{t-}^{\varepsilon_k})\right)\tilde{N}(dt,dz)
\end{eqnarray*}
and
\begin{eqnarray*}
d\partial_z U(\bar{X}_{t},\eta^{\varepsilon_k}_t)=\!\!\!\!\!\!\!\!\!&&(\mathcal{L}_1+\mathcal{L}_2+\mathcal{L}_3) (\partial_z U)(\bar X_t, \eta^{\varepsilon_k}_t)dt\\
&&+\partial^2_z U(\bar X_t, \eta^{\varepsilon_k}_t)\cdot\frac{1}{\sqrt{\varepsilon_k}}\left[b(X^{\vare_k}_t, \alpha^{\varepsilon_k}_t)-\bar{b}(X^{\vare_k}_t)\right] dt\\
&&+ \partial_x \partial_z U(\bar X_t, \eta^{\varepsilon_k}_t)\cdot \sigma(\bar X_t)d W_t+\partial^2_z U(\bar X_t, \eta^{\varepsilon_k}_t)\cdot\left(\nabla\sigma(\bar{X}_t)\cdot\eta_t^{\varepsilon_k}\right)d W_t.
\end{eqnarray*}
Then using integration by parts formula, for any $t\geq t_0$,
\begin{eqnarray*}
&&\big\langle\Phi(X_{t}^{\varepsilon_k},\alpha_{t}^{\varepsilon_k}), \partial_z U(\bar{X}_{t},\eta^{\varepsilon_k}_t)\big\rangle\\
=\!\!\!\!\!\!\!\!\!&&\left\langle\Phi(X_{t_0}^{\varepsilon_k},\alpha^{\varepsilon_k}_{t_0}), \partial_z U(\bar{X}_{t_0},\eta^{\varepsilon_k}_{t_0})\right\rangle+\int^t_{t_0}\left\langle\Phi(X_{s}^{\varepsilon_k},\alpha_{s}^{\varepsilon_k}), d\partial_z U(\bar{X}_{s},\eta^{\varepsilon_k}_s)\right\rangle\\
&&+\int^t_{t_0}\left\langle\partial_z U(\bar{X}_{s},\eta^{\varepsilon_k}_s), d\Phi(X_{s}^{\varepsilon_k},\alpha_{s}^{\varepsilon_k})\right\rangle\!+\!\int^t_{t_0} d[\Phi(X^{\varepsilon_k},\alpha^{\varepsilon_k}), \partial_z U(\bar{X},\eta^{\varepsilon_k})]_s\\
=\!\!\!\!\!\!\!\!\!&&\left\langle\!\Phi(X_{t_0}^{\varepsilon_k}\!,\alpha^{\varepsilon_k}_{t_0}), \partial_z U(\bar{X}_{t_0},\eta^{\varepsilon_k}_{t_0})\!\right\rangle
\!+\!\frac{1}{\sqrt{\varepsilon_k}}\!\int^t_{t_0}\!\!\left\langle\!\Phi(X_{s}^{\varepsilon_k}\!,\alpha_{s}^{\varepsilon_k}),\partial^2_z U(\bar X_s, \eta^{\varepsilon_k}_s)\!\cdot\!\left[b(X^{\vare_k}_s, \alpha^{\varepsilon_k}_s)-\bar{b}( X^{\vare_k}_s)\right] \right\rangle ds\\
&&+\frac{1}{\varepsilon_k}\int^{t}_{t_0}\left\langle Q(X_{s}^{\varepsilon_k})\Phi(X_{s}^{\varepsilon_k},\cdot)(\alpha^{\varepsilon_k}_{s}), \partial_z U(\bar X_s,\eta_{s}^{\varepsilon_k})\right\rangle ds+R^{\vare_k}_{\partial_zU,\Phi}(t,t_0),
\end{eqnarray*}
where
\begin{eqnarray}
&&R^{\vare_k}_{\partial_zU,\Phi}(t,t_0)   \nonumber\\
:=\!\!\!\!\!\!\!\!&&\int^t_{t_0}\left\langle\Phi(X_{s}^{\varepsilon_k},\alpha_{s}^{\varepsilon_k}),[(\mathcal{L}_1+\mathcal{L}_2+\mathcal{L}_3) (\partial_z U)(\bar X_s, \eta^{\varepsilon_k}_s)ds\right.\nonumber\\
&&~~~+ \left.\partial_x \partial_z U(\bar X_s, \eta^{\varepsilon_k}_s)\cdot \sigma(\bar X_s)d W_s+\partial^2_z U(\bar X_s, \eta^{\varepsilon_k}_s)\cdot\left(\partial_x\sigma(\bar{X}_s)\cdot\eta_s^{\varepsilon_k}\right) d W_s]\right\rangle \nonumber\\
&&+\int^t_{t_0}\big\langle\partial_z U(\bar X_{s},\eta^{\varepsilon_k}_{s}),\big[\partial_x\Phi(X_{s}^{\varepsilon_k},\alpha_{s}^{\varepsilon_k})\cdot b(X_{s}^{\varepsilon_k},\alpha_{s}^{\varepsilon_k})ds+\partial_x\Phi(X_{s}^{\varepsilon_k},\alpha_{s}^{\varepsilon_k})\cdot \sigma(X_{s}^{\varepsilon_k})dW_s\nonumber\\
&&\quad\quad\quad\quad+\frac12
\text{Tr}\left(\partial_x^{2}\Phi(X_{s}^{\varepsilon_k},\alpha_{s}^{\varepsilon_k})\cdot \left(\sigma\sigma^{*}\right)(X_{s}^{\varepsilon_k})\right)ds\big]\big\rangle \nonumber\\
&&+ \int^t_{t_0}\!\text{Tr}\! \left[\left(\partial_x\Phi(X_{s}^{\varepsilon_k},\alpha_{s}^{\varepsilon_k})
 \sigma(X_{s}^{\varepsilon_k})\right)^{\ast}
 \left(\partial_x \partial_z U(\bar X_s, \eta^{\varepsilon_k}_s)
\sigma(\bar X^{k}_s)\!+\!\partial^2_z U(\bar X_s, \eta^{\varepsilon_k}_s)
(\nabla\sigma(\bar{X}_s)\cdot\eta_s^{\varepsilon_k})\right)\right]ds\nonumber\\
&&+\int^t_{t_0}\!\int_{[0,\infty)}\!\!\left\langle\partial_z U(\bar X_{s},\eta^{\varepsilon_k}_{s}),\left( \Phi(X_{s}^{\varepsilon_k},\alpha_{s-}^{\varepsilon_k}\!+\!g^{\varepsilon}(X_{s}^{\varepsilon_k},\alpha_{s-}^{\varepsilon_k},z))
\!-\!\Phi(X_{s}^{\varepsilon_k},\alpha_{s-}^{\varepsilon_k})\right)\right\rangle\tilde{N}(ds,dz).\label{DRPhi}
\end{eqnarray}

Poisson equation \eref{PEQ} shows that
\begin{eqnarray}
&&\frac{1}{\sqrt{\varepsilon_k}}\int^t_{t_0}\left \langle\partial_z U(\bar X_s,\eta_{s}^{\varepsilon_k}),b(X^{\vare_k}_s, \alpha^{\varepsilon_k}_s)-\bar{b}(X^{\vare_k}_s)\right\rangle ds\nonumber\\
=\!\!\!\!\!\!\!\!&&-\frac{1}{\sqrt{\varepsilon_k}}\int^t_{t_0} \left\langle \partial_z U(\bar X_s,\eta_{s}^{\varepsilon_k}),Q(X_{s}^{\varepsilon_k})\Phi(X_{s}^{\varepsilon_k},\cdot)(\alpha^{\varepsilon_k}_{s})\right\rangle ds\nonumber\\
=\!\!\!\!\!\!\!\!&&\sqrt{\vare_k}\left[R^{\vare_k}_{\partial_zU,\Phi}(t,t_0)+\langle\Phi(X_{t_0}^{\varepsilon_k},\alpha^{\varepsilon_k}_{t_0}), \partial_z U(\bar{X}_{t_0},\eta^{\varepsilon_k}_{t_0})\rangle-\langle\Phi(X_{t}^{\varepsilon_k},\alpha_{t}^{\varepsilon_k}), \partial_z U(\bar{X}_{t},\eta^{\varepsilon_k}_t)\rangle\right]\nonumber\\
&&+\int^t_{t_0}\left\langle\Phi(X_{s}^{\varepsilon_k},\alpha_{s}^{\varepsilon_k}),\partial^2_z U(\bar X_s, \eta^{\varepsilon_k}_s)\cdot F(X^{\varepsilon_k}_s,\alpha^{\varepsilon_k}_s)\right\rangle ds.\label{Fo4.21}
\end{eqnarray}
Using
\begin{eqnarray*}
&&\int^t_{t_0}\left\langle\Phi(X_{s}^{\varepsilon_k},\alpha_{s}^{\varepsilon_k}),\partial^2_z U(\bar X_s, \eta^{\varepsilon_k}_s)\cdot F(X^{\varepsilon_k}_s,\alpha^{\varepsilon_k}_s)\right\rangle ds\\
=\!\!\!\!\!\!\!\!&&\sum^n_{i=1}\sum^n_{j=1}\int^t_{t_0} \partial _{z_i}\partial_{z_j} U(\bar X_s, \eta^{\varepsilon_k}_s) (F\otimes\Phi)_{ij}(X_{s}^{\varepsilon_k},\alpha^{\varepsilon_k}_{s})ds\\
=\!\!\!\!\!\!\!\!&&\frac{1}{2}\sum^n_{i=1}\sum^n_{j=1}\int^t_{t_0} \partial _{z_i}\partial_{z_j} U(\bar X_s, \eta^{\varepsilon_k}_s) \left[(F\otimes\Phi)+(F\otimes\Phi)^{\ast}\right]_{ij}(X_{s}^{\varepsilon_k},\alpha^{\varepsilon_k}_{s})ds
\end{eqnarray*}
and It\^{o}'s formula, we get
\begin{eqnarray}
U(\bar X_t, \eta^{\varepsilon_k}_t)
=\!\!\!\!\!\!\!\!&& U(\bar X_{t_0}, \eta^{\varepsilon_k}_{t_0})+\int^t_{t_0}(\mathcal{L}_1+\mathcal{L}_2+\mathcal{L}_3) U(\bar X_s, \eta^{\varepsilon_k}_s)ds\nonumber\\
&&+\int^t_{t_0}\big\langle\partial_z U(\bar X_s, \eta^{\varepsilon_k}_s),\frac{1}{\sqrt{\varepsilon_k}}[b(X^{\vare_k}_s, \alpha^{\varepsilon_k}_s)-\bar{b}(X^{\vare_k}_s)]\big\rangle ds\nonumber\\
&&+ \int^t_{t_0}\big\langle\partial_x U(\bar X_s, \eta^{\varepsilon_k}_s),\sigma(\bar X_s)d W_s\big\rangle+ \int^t_{t_0}\big\langle\partial_z U(\bar X_s, \eta^{\varepsilon_k}_s),[\partial_x\sigma(\bar{X}_s)\cdot\eta_s^{\varepsilon_k}] d W_s\big\rangle\nonumber\\
=\!\!\!\!\!\!\!\!&& U(\bar X_{t_0}, \eta^{\varepsilon_k}_{t_0})+\int^t_{t_0}(\mathcal{L}_1+\mathcal{L}_2+\mathcal{L}_3) U(\bar X_s, \eta^{\varepsilon_k}_s)ds\nonumber\\
&&+\frac12\sum^n_{i=1}\sum^n_{j=1}\int^t_{t_0} \partial_{x_i}\partial_{x_j} U(\bar X_s,\eta^{\vare_k}_s) \left[F\otimes\Phi+(F\otimes\Phi)^{\ast}\right]_{ij}(X_{s}^{\varepsilon_k},\alpha^{\varepsilon_k}_{s})ds\nonumber\\
&&+\sqrt{\vare_k}\big[R^{\vare_k}_{\partial_zU,\Phi}(t,t_0)\!+\!\big\langle\Phi(X_{t_0}^{\varepsilon_k},\alpha^{\varepsilon_k}_{t_0}), \partial_z U(\bar{X}_{t_0},\eta^{\varepsilon_k}_{t_0})\rangle\!-\!\langle\Phi(X_{t}^{\varepsilon_k},\alpha_{t}^{\varepsilon_k}), \partial_z U(\bar{X}_{t},\eta^{\varepsilon_k}_t)\big\rangle\big]\nonumber\\
&&+ \int^t_{t_0}\!\!\left\langle\partial_x U(\bar X_s, \eta^{\varepsilon_k}_s),\sigma(\bar X_s)d W_s\right\rangle\!+\! \int^t_{t_0}\!\!\left\langle\partial_z U(\bar X_s, \eta^{\varepsilon_k}_s),\nabla\sigma(\bar{X}_s)\cdot\eta_s^{\varepsilon_k} d W_s\right\rangle.\label{F4.21}
\end{eqnarray}
By \eref{E1}, \eref{X}, \eref{barX} and $U\in C^3_b(\RR^{2n})$, it is easy to see that
\begin{eqnarray}
&&\lim_{k\to\infty}\sqrt{\vare_k} \EE\big\{\big[R^{\vare_k}_{\partial_zU,\Phi}(t,t_0)+\left\langle\Phi(X_{t_0}^{\varepsilon_k},\alpha^{\varepsilon_k}_{t_0}), \partial_z U(\bar{X}_{t_0},\eta^{\varepsilon_k}_{t_0})\right\rangle\nonumber\\
&&\quad\quad\quad\quad\quad-\left\langle\Phi(X_{t}^{\varepsilon_k},\alpha_{t}^{\varepsilon_k}), \partial_z U(\bar{X}_{t},\eta^{\varepsilon_k}_t)\right\rangle\big]\Psi_{t_0}(\bar X,\eta^{\vare_k})\big\}=0.\label{F4.22}
\end{eqnarray}
Meanwhile,
\begin{eqnarray}
&&\EE\Big\{\Big[\int^t_{t_0}\left\langle\partial_x U(\bar X_s, \eta^{\varepsilon_k}_s),\sigma(\bar X_s)d W_s\right\rangle\nonumber\\
&&\quad\quad\quad\quad\quad+ \int^t_{t_0}\left\langle\partial_z U(\bar X_s, \eta^{\varepsilon_k}_s),\left(\partial_x\sigma(\bar{X}_s)\cdot\eta_s^{\varepsilon_k}\right) d W_s\right\rangle\Big]\Psi_{t_0}(\bar X,\eta^{\vare_k})\Big\}=0.\label{F4.23}
\end{eqnarray}
Thus it remains to prove
\begin{eqnarray}
\lim_{k\rightarrow \infty}\EE\left[\Gamma^{\vare_k}\Psi_{t_0}(\bar X,\eta^{\vare_k})\right]=\EE \left[\Gamma\Psi_{t_0}(\bar X, Z)\right],\label{Step2}
\end{eqnarray}
where
\begin{eqnarray*}
\Gamma^{\vare_k}:=\!\!\!\!\!\!\!\!&&U(\bar X_t, \eta^{\varepsilon_k}_t)-U(\bar X_{t_0}, \eta^{\varepsilon_k}_{t_0})-\int^t_{t_0}(\mathcal{L}_1+\mathcal{L}_2+\mathcal{L}_3) U(\bar X_s, \eta^{\varepsilon_k}_s)ds\\
&&-\frac{1}{2}\sum^n_{i=1}\sum^n_{j=1}\int^t_{t_0} \partial_{z_i}\partial_{z_j} U(\bar X_s,\eta^{\vare_k}_s) \left[F\otimes\Phi+(F\otimes\Phi)^{\ast}\right]_{ij}(X_{s}^{\varepsilon_k},\alpha^{\varepsilon_k}_{s})ds
\end{eqnarray*}
and
\begin{eqnarray*}
\Gamma:=U(\bar X_t, Z_t)-U(\bar X_{t_0}, Z_{t_0})-\int^t_{t_0}\mathcal{L}U(\bar X_s, Z_s)ds.
\end{eqnarray*}
Consequently, combining \eref{F4.21}-\eref{F4.23}, it is easy to see \eref{Mar11} holds.

\textbf{Step 2}: The proof of \eref{Step2}. Since $(\bar{X},X^{\varepsilon_k},\eta^{\varepsilon_{k}}){\overset{W}\longrightarrow} (\bar{X}, \bar{X}, Z)$ in $C([0,T];\RR^{3n})$ for any $T>0$, we have
\begin{eqnarray*}
&&\lim_{k\to \infty} \Big|\EE\Big[\Big(U(\bar X_t, \eta^{\varepsilon_k}_t)-U(\bar X_{t_0}, \eta^{\varepsilon_k}_{t_0})-\int^t_{t_0}(\mathcal{L}_1+\mathcal{L}_2+\mathcal{L}_3) U(\bar X_s, \eta^{\varepsilon_k}_s)ds\Big)\Psi_{t_0}(\bar X,\eta^{\vare_k})\Big]\\
&&\quad\quad-\EE\Big[\left(U(\bar X_t, Z_t)-U(\bar X_{t_0}, Z_{t_0})-\int^t_{t_0}(\mathcal{L}_1+\mathcal{L}_2+\mathcal{L}_3) U(\bar X_s, Z_s)ds\right)\Psi_{t_0}(\bar X,Z)\Big]\Big|=0
\end{eqnarray*}
and, for any $1\leq i,j\leq n$,
\begin{eqnarray*}
&&\lim_{k\rightarrow \infty} \Big|\EE\left[\left(\int^t_{t_0} \partial_{z_i}\partial_{z_j} U(\bar X_s,\eta^{\vare_k}_s) \overline{[F\otimes\Phi+(F\otimes\Phi)^{\ast}]}_{ij}(X_{s}^{\varepsilon_k})ds\right)\Psi_{t_0}(\bar X,\eta^{\vare_k})\right]\\
&&\quad\quad-\EE\Big[\left(\int^t_{t_0} \partial_{z_i}\partial_{z_j} U(\bar X_s,Z_s) \overline{[F\otimes\Phi+(F\otimes\Phi)^{\ast}]}_{ij}(\bar X_{s})ds\right)\Psi_{t_0}(\bar X,Z)\Big]\Big|=0.
\end{eqnarray*}
Hence, in order to prove \eref{Step2}, we just have to prove the following, for any $1\leq i,j\leq n$,
\begin{eqnarray}
\lim_{k\rightarrow \infty} \EE\Big|\int^t_{t_0}\partial_{z_i}\partial_{z_j} U(\bar X_s,\eta^{\vare_k}_s) \left[(F\otimes\Phi)_{ij}(X_{s}^{\varepsilon_k},\alpha^{\varepsilon_k}_{s})-\overline{(F\otimes\Phi)}_{ij}(X^{\varepsilon_k}_{s})\right]ds\Big|=0.\label{MP2}
\end{eqnarray}
It is easy to check $\tilde{H}_{ij}(x,\alpha):=(F\otimes\Phi)_{ij}(x,\alpha)-\overline{(F\otimes\Phi)}_{ij}(x)\in C^2_{p}(\RR^n\times \mathbb{S}, \RR)$ and satisfies the "central condition" \eref{CenCon}. Th.\ref{Poisson} gives that the following Poisson equation
\begin{eqnarray*}
-Q(x)\tilde{\Phi}_{ij}(x,\cdot)(\alpha)=\tilde H_{ij}(x,\alpha)
\end{eqnarray*}
has a solution $\tilde{\Phi}_{ij}$. And there exists $k>0$ such that
\begin{eqnarray}
\sum^{2}_{l=0}\|\partial^l_{x}\tilde \Phi_{ij}(x,\cdot)\|_{\infty}\le  C(1+|x|^{k})\label{F04.30}
\end{eqnarray}
Next, by a similar argument in \eref{Fo4.21}, we have
\begin{eqnarray*}
&&\int^t_{t_0} \partial_{z_i}\partial_{z_j} U(\bar X_s,\eta^{\vare_k}_s)\tilde{H}_{ij}(X^{\varepsilon}_s,\alpha^{\vare_k}_s) ds\nonumber\\
=\!\!\!\!\!\!\!\!&&-\int^t_{t_0} \partial_{z_i}\partial_{z_j} U(\bar X_s,\eta^{\vare_k}_s)Q(X^{\varepsilon}_s)\tilde{\Phi}_{ij}(X^{\varepsilon}_s,\cdot)(\alpha^{\vare_k}_s) ds\nonumber\\
=\!\!\!\!\!\!\!\!&&\vare_k\big[ R^{\vare_k}_{\partial_{z_i}\partial_{z_j} U, \tilde{\Phi}_{ij}}(t,t_0)+\tilde{\Phi}_{ij}(X_{t_0}^{\varepsilon_k},\alpha^{\varepsilon_k}_{t_0}) \partial_{z_i}\partial_{z_j} U(\bar{X}_{t_0},\eta^{\varepsilon_k}_{t_0})-\tilde \Phi_{ij}(X_{t}^{\varepsilon_k},\alpha_{t}^{\varepsilon_k}) \partial_{z_i}\partial_{z_j} U(\bar{X}_{t},\eta^{\varepsilon_k}_t)\rangle\big]\nonumber\\
&&+\sqrt{\varepsilon_k}\int^t_{t_0}\tilde{\Phi}_{ij}(X_{s}^{\varepsilon_k},\alpha_{s}^{\varepsilon_k})\partial_z\partial_{z_i}\partial_{z_j} U(\bar X_s, \eta^{\varepsilon_k}_s)\cdot\left[b(X^{\vare_k}_s, \alpha^{\varepsilon_k}_s)-\bar{b}( X^{\vare_k}_s)\right] ds,\label{Fo4.22}
\end{eqnarray*}
where $R^{\vare_k}_{\partial_{z_i}\partial_{z_j} U, \tilde{\Phi}_{ij}}(t,t_0)$ is defined as in \eref{DRPhi} with $\partial_z U$ and $\Phi$ being replaced by $\partial_{z_i}\partial_{z_j} U$ and $\tilde{\Phi}_{ij}$ respectively.
Then $U\in C^4_b(\RR^{2n})$ and \eref{F04.30} imply that \eref{MP2} holds.

\textbf{Step 3}:
In order to complete the proof, it is necessary to demonstrate the uniqueness of solutions to the martingale problem with the operator $\mathcal{L}$, which is equivalent to the uniqueness of weak solutions of  \eref{AR1} and \eref{limE}. It should be noted that \eref{AR1} exhibits well-posedness, while \eref{limE} is a linear equation. Consequently, it can be easily verified that \eref{AR1} and \eref{limE} have a unique strong solution.
The proof is complete. $\square$

\begin{remark}
In order to establish the validity of the square root of the diffusion matrix $\overline{F\otimes \Phi+(F\otimes \Phi)^*}(x)$, it is necessary to prove the nonnegativity. It should be noted that $Q(x)$ is negative under the invariant measure $\mu^x$, that is, for any function $f$ on $\mathbb{S}$ and $x$ in $\RR^n$, the following inequality holds:
$$\sum_{i\in \mathbb{S}}[Q(x)f](i)f(i)\mu^x_{i}\leq 0.$$
Based on this, it can be concluded that the diffusion matrix $\Theta(x)\Theta(x)\!=\!\overline{F\otimes \Phi\!+\!(F\otimes \Phi)^*}(x)$ is nonnegative, as demonstrated in \cite[Lemmq 4.2.4]{PTW2012}.
\end{remark}

\subsection{Weak convergence of $(X^{\varepsilon}_t-\bar{X}_t)/\sqrt{\varepsilon}$ in $\RR^n$}

In this subsection, we study the weak convergence \eref{WeakA2}. In order to achieve a satisfactory order of convergence, we require the condition ${\color{blue} {\bf H}_{3}}$. In fact, using Remark \ref{HighR}, ${\color{blue} {\bf H}_{3}}$ ensures that $\bar b\in C^{3}_p(\RR^n,\RR^n)$.

We consider the following Kolmogorov equation:
	\begin{equation}\left\{\begin{array}{l}\label{KE}
			\displaystyle
			\partial_t u(t,x,z)=\mathcal{L} u(t,x,z),\quad t\in[0, T], \\
			u(0, x,z)=\phi(z),
		\end{array}\right.
	\end{equation}
where $\phi\in C^{4}_p(\RR^{n})$ and $\mathcal{L}$ is the infinitesimal generator of Eq.\eref{AR1} and Eq.\eref{limE}, which is defined in \eref{Generator}. \eref{KE} admits a unique solution $u$ given by
	$$
	u(t,x,z)=\EE\phi(Z^{x,z}_t),\quad t\geq 0,
	$$
where $Z^{x,z}$ is the unique solution of Eq. \eref{limE} with $\bar{X}_0=x$ and $Z_0=z$.

\vspace{2mm}
For $i,j\in \mathbb{N}_{+}$ and vectors $l_1,\ldots,l_{i+j}\in \RR^n$, denote $\partial^i_{x}\partial^j_z u(t,x,z)$ by the $(i+j)$-order partial derivative of $u(t,x,z)$ with respect to $x$ $i$-times and $z$ $j$-times. For instance, the chain rule we have
$$\partial_x u(t,x,z)\cdot l=\EE[\nabla\phi(Z^{x,z}_t)\cdot (\partial_x Z^{x,z}_t\cdot l)],~~~\partial_{z} u(t,x,z)\cdot l=\EE[\nabla\phi(Z^{x,z}_t)\cdot (\partial_z Z^{x,z}_t\cdot l)],$$
where $\partial_x Z^{x,z}_t\cdot l$ and $\partial_z Z^{x,z}_t\cdot l$ are directional derivatives of $Z^{x,z}_t$ with respect to $x$ and $z$ in the direction $l$ respectively, which satisfy
\begin{equation*}\left\{\begin{array}{l}
				\displaystyle\label{partialxZ}
d \left[\partial_x Z^{x,z}_t\cdot l\right]=  \left[\nabla^2\bar{b}(\bar{X}^x_t)\cdot (Z^{x,z}_t,\partial_x\bar{X}^x_t\cdot l)+\nabla\bar{b}(\bar{X}^x_t)\cdot (\partial_x Z^{x,z}_t\cdot l)\right] dt\\
\quad\quad\quad\quad\quad\quad\quad+\left[\nabla^2\sigma(\bar{X}^x_t)\cdot (Z^{x,z}_t, \partial_x\bar{X}^x_t\cdot l)+\nabla\sigma(\bar{X}^x_t)\cdot (\partial_x Z^{x,z}_t\cdot l)\right]dW_t\\
\quad\quad\quad\quad\quad\quad\quad+\nabla\Theta(\bar{X}^x_t)\cdot (\partial_x\bar{X}^x_t\cdot l) d\tilde{W}_t,\nonumber\\
\partial_x Z^{x,z}_0\cdot l=0,\nonumber
			\end{array}\right.
		\end{equation*}
and
\begin{equation*}\left\{\begin{array}{l}
				\displaystyle\label{partialzZ}
d \left(\partial_z Z^{x,z}_t\cdot l\right)=  \nabla\bar{b}(\bar{X}^x_t)\cdot (\partial_z Z^{x,z}_t\cdot l) dt+\nabla\sigma(\bar{X}^x_t)\cdot (\partial_z Z^{x,z}_t\cdot l)dW_t,\nonumber\\
\partial_z Z^{x,z}_0\cdot l=l,\nonumber
			\end{array}\right.
		\end{equation*}
where $\partial_x \bar{X}^{x}_t\cdot l$ is directional derivative of $\bar{X}^{x}_t$ with respect to $x$ in the direction $l$, which satisfies
		\begin{equation*}\left\{\begin{array}{l}
				\displaystyle
				d[\partial_x \bar{X}^{x}_t\cdot l]=\nabla\bar{b}(\bar{X}^{x}_t)\cdot (\partial_x \bar{X}^{x}_t\cdot l)dt+\nabla\sigma(\bar{X}^{x}_t)\cdot (\partial_x \bar{X}^{x}_t\cdot l)d W_t,\nonumber\\
				\partial_x \bar{X}^{x}_0\cdot l=l.\nonumber
			\end{array}\right.
		\end{equation*}
The other notations can be interpreted similarly.

Now, we present the regularity estimates of the solution $u$ of Eq.\eref{KE}.
	
\begin{lemma} \label{Lemma 5.3}
		For $0\leq i\leq 2$ and $0\leq j\leq 4$ with $0\leq i+j\leq 4$, unit vectors $l_1,\ldots,l_{i+j}\in \RR^n$ and $T>0$, there exist $C_T>0$ and $k>0$ such that for $x\in\RR^n$,
		\begin{eqnarray}
		&&\sup_{t\in[0,T],z\in\mathbb{R}^n}\!\!\left\{\left|\partial^{i}_{x}\partial^{j}_z u(t,x,z)\cdot(l_1,\ldots,l_{i+j})\right|\!+\!|\partial_t(\partial_z u(t,x,z))\cdot l_1|\!+\!\left.|\partial_t(\partial^2_z u(t,x,z))\cdot (l_1,l_2)\right|\right\}\nonumber\\
		&&\leq C_T (1+|x|^{k}).\label{UE1}
\end{eqnarray}			
	\end{lemma}
	
\begin{proof}
By \eref{barc1}, we can obtain for any unit vector $l\in\RR^n$,
\begin{eqnarray}
\langle\nabla \bar{b}(x)\cdot l,l\rangle\leq C|x|^2.\label{fourthpartial}
\end{eqnarray}
Using $\bar{b}\in C^{3}_{p}(\RR^n,\RR^n)$ and $\sigma\in C^{3}_{p}(\RR^n,\RR^{n}\otimes\RR^{d})$ and \eref{fourthpartial}, we calculate directly and get that, for $0\leq i\leq 2$ and $0\leq j\leq 4$ with $0\leq i+j\leq 4$, unit vectors $l_1,\ldots,l_{i+j}\in \RR^n$ and $p,T>0$, there exist $C_T>0$ and $k_p>0$ such that for $x\in\RR^n$,
		\begin{eqnarray}
			\sup_{0\leq t\leq T,z\in\RR^n}\EE|\partial^i_x\partial^j_z Z^{x,z}_t\cdot (l_1,\ldots,l_{i+j})|^{p}\leq C_T(1+|x|^{k_p}).\label{EZ}
		\end{eqnarray}
Using \eref{EZ}, \eref{UE1} is easily derived from $\phi\in C^{4}_p(\RR^{n})$. The proof is complete.
\end{proof}

Finally, we prove our second result in the central limit theorem.
\vspace{-3mm}

~~\\
\textbf{Proof of Theorem \ref{main result 5}:}
Fxed $t>0$, let $\tilde{u}^t(s,x,z):=u(t-s,x,z)$, $s\in [0,t]$. By It\^{o}'s formula and following a similar argument in \eref{Fo4.21}, we have
\begin{eqnarray*}
\tilde{u}^t(t, \bar{X}_t,\eta^{\vare}_t)=\!\!\!\!\!\!\!\!&&\tilde{u}^t(0,x,0)+\int^t_0 \partial_s \tilde{u}^t(s, \bar{X}_s,\eta^{\vare}_s)ds+\int^t_{0}(\mathcal{L}_1+\mathcal{L}_2+\mathcal{L}_3) \tilde{u}^t(s, \bar{X}_s,\eta^{\vare}_s)ds\\
&&+\frac{1}{\sqrt{\varepsilon}}\int^t_{0}\left\langle\partial_z \tilde{u}^t(s, \bar{X}_s,\eta^{\vare}_s),\left(b(X^{\vare}_s, \alpha^{\vare}_s)-\bar{b}( X^{\vare}_s)\right)\right\rangle ds\\
&&+\int^t_{0}\!\!\left\langle\partial_x \tilde{u}^t(s, \bar{X}_s,\eta^{\vare}_s),\sigma(\bar X_s)d W_s\right\rangle\!+\! \int^t_{0}\big\langle\partial_z \tilde{u}^t(s, \bar{X}_s,\eta^{\vare}_s),\partial_x\sigma(\bar{X}_t)\cdot\eta_s^{\varepsilon} dW_s\big\rangle\\
=\!\!\!\!\!\!\!\!&&\tilde{u}^t(0,x,0)+\int^t_0 \partial_s \tilde{u}^t(s, \bar{X}_s,\eta^{\vare}_s)ds+\int^t_{0}(\mathcal{L}_1+\mathcal{L}_2+\mathcal{L}_3)\tilde{u}^t(s, \bar{X}_s,\eta^{\vare}_s)ds\\
&&+\frac12\sum^n_{i=1}\sum^n_{j=1}\int^t_{0} \partial_{z_i}\partial_{z_j} \tilde{u}^t(s, \bar{X}_s,\eta^{\vare}_s) \left[F\otimes\Phi+(F\otimes\Phi)^{\ast}\right]_{ij}(X_{s}^{\varepsilon},\alpha^{\varepsilon}_{s})ds\\
&&+\sqrt{\vare_k}\Big[R^{\vare}_{\partial_z\tilde{u},\Phi}(t,0)+\int^t_0\langle \Phi(X_{s}^{\varepsilon},\alpha^{\varepsilon}_{s}),\partial_s\partial_z\tilde{u}^t(s, \bar{X}_s,\eta^{\vare}_s)\rangle ds\\
&&\quad\quad\quad\quad+ \langle\Phi(x,\alpha), \partial_z \tilde{u}^t(0,x,0)\rangle-\langle\Phi(X_{t}^{\varepsilon},\alpha_{t}^{\varepsilon}), \partial_z \tilde{u}^t(t, \bar{X}_t,\eta^{\vare}_t)\rangle\Big]\\
&&+ \int^t_{0}\!\left\langle\partial_x \tilde{u}^t(s, \bar{X}_s,\eta^{\vare}_s),\sigma(\bar X_s)d W_s\right\rangle\!+\!\! \int^t_{0}\big\langle\partial_z \tilde{u}^t(s, \bar{X}_s,\eta^{\vare}_s),\partial_x\sigma(\bar{X}_t)\cdot\eta_s^{\varepsilon}dW_s\big\rangle,
	\end{eqnarray*}
where $R^{\vare}_{\partial_z\tilde{u},\Phi}(t,0)$ is defined in \eref{DRPhi} with $\vare_k$ and $\partial_z U$ are replaced by $\vare$ and $\partial_z\tilde{u}$ respectively.	
	
Note that
$$\tilde{u}^t(t, \bar{X}_t,\eta^{\vare}_t)=\phi(\eta^{\vare}_t),\tilde{u}^t(0, x,0)=\EE\phi(Z^{x,0}_t)$$
 and
 $$\partial_s \tilde{u}^t(s, \bar{X}_s,\eta^{\vare}_s )=-\mathcal{L} \tilde{u}^t(s, \bar{X}_s,\eta^{\vare}_s ),$$
 we have
	\begin{eqnarray}
		\left|\EE\phi(\eta^{\vare}_t)-\EE\phi(Z^{x,0}_t)\right|
=\!\!\!\!\!\!\!\!&&\Big|\EE\int^t_0 -\mathcal{L}\tilde{u}^t(s, \bar{X}_s,\eta^{\vare}_s )ds+\EE\int^t_{0}(\mathcal{L}_1+\mathcal{L}_2+\mathcal{L}_3)\tilde{u}^t(s, \bar{X}_s,\eta^{\vare}_s)ds\nonumber\\
&&+\!\frac12\sum^n_{i=1}\sum^n_{j=1}\EE\!\int^t_{0} \!\partial_{z_i}\partial_{z_j} \tilde{u}^t(s, \bar{X}_s,\eta^{\vare}_s) \left[F\otimes\Phi\!+\!(F\otimes\Phi)^{\ast}\right]_{ij}(X_{s}^{\varepsilon},\alpha^{\varepsilon}_{s})ds\nonumber\\
&&+\sqrt{\vare}\EE\Big[R^{\vare}_{\partial_z\tilde{u},\Phi}(t,0)+\int^t_0\langle \Phi(X_{s}^{\varepsilon},\alpha^{\varepsilon}_{s}),\partial_s\partial_z\tilde{u}^t(s, \bar{X}_s,\eta^{\vare}_s)\rangle ds\nonumber\\
&&\quad\quad\quad+\langle\Phi(x,\alpha), \partial_z \tilde{u}^t(0,x,0)\rangle-\langle\Phi(X_{t}^{\varepsilon},\alpha_{t}^{\varepsilon}), \partial_z \tilde{u}^t(t, \bar{X}_t,\eta^{\vare}_t)\rangle\Big]\Big|\nonumber\\
\leq\!\!\!\!\!\!\!\!&&\frac{1}{2}\EE\Big|\int^t_{0} \text{Tr}\big[\partial^2_{z} \tilde{u}^t(s, \bar{X}_s,\eta^{\vare}_s) \left[F\otimes\Phi+(F\otimes\Phi)^{\ast}\right](X_{s}^{\varepsilon},\alpha_{s}^{\varepsilon})\big]\nonumber\\
&&\quad\quad\quad-\text{Tr}\left[\partial^2_{z} \tilde{u}^t(s, \bar{X}_s,\eta^{\vare}_s)\overline{F\otimes\Phi+(F\otimes\Phi)^{\ast}}(\bar X_{s})\right]ds\Big|\nonumber\\
&&+\sqrt{\vare}\Big|\EE\big[R^{\vare}_{\partial_z\tilde{u},\Phi}(t,0)+\int^t_0\langle \Phi(X_{s}^{\varepsilon},\alpha^{\varepsilon}_{s}),\partial_s\partial_z\tilde{u}^t(s, \bar{X}_s,\eta^{\vare}_s)\rangle ds\nonumber\\
&&\quad\quad+\langle\Phi(x,\alpha), \partial_z \tilde{u}^t(0,x,0)\rangle\!-\!\langle\Phi(X_{t}^{\varepsilon},\alpha_{t}^{\varepsilon}), \partial_z \tilde{u}^t(t, \bar{X}_t,\eta^{\vare}_t)\rangle\big]\Big|.
\label{F11}
	\end{eqnarray}
	
For any $s\in [0,t], x\in\RR^{n},y\in\RR^{m}$, define
\begin{eqnarray*}
F^t(s,x,z,\alpha):=\!\!\!\!\!\!\!\!&&\frac{1}{2}\text{Tr}\big[\partial^2_{z} \tilde{u}^t(s, x,z)[F\otimes\Phi+(F\otimes\Phi)^{\ast}](x,\alpha)\big]\nonumber\\
&&-\frac{1}{2}\text{Tr}\left[\partial^2_{z} \tilde{u}^t(s, x,z)\overline{F\otimes\Phi+(F\otimes\Phi)^{\ast}}(x)\right].
\end{eqnarray*}
By Lemma \ref{Lemma 5.3} and $F\otimes\Phi-\overline{(F\otimes\Phi)}\in C^{2}_p(\RR^n\times\mathbb{S},\RR^n\otimes\RR^n)$, we can obtain that for any $i+j\leq 2$,
 $$ \sup_{s\in [0, t],z\in\RR^n}\left[\left\|\partial^i_x\partial^j_z F^t(s,x,z,\cdot)\right\|_{\infty}+\left\|\partial_s F^t(s,x,z,\cdot)\right\|_{\infty}\right]\leq C(1+|x|^{k}).$$
 Then by Th.\ref{Poisson}, the following Poisson equation
	\begin{eqnarray*}
		-Q(x)\tilde{\Phi}^t(s,x,z,\cdot)(\alpha)=F^t(s,x,z,\alpha),\quad s\in [0,t],\alpha\in \mathbb{S}\label{WPE}
	\end{eqnarray*}
admits a solution $\tilde{\Phi}^t(s,x,z,i)$. Moreover, for any $T>0$, $t\in [0,T]$, there exist $C_T, k>0$ such that the following estimate holds for any $0\leq i+j\leq 2$
	\begin{eqnarray}
		\sup_{s\in [0, t],z\in\RR^n}\left[\left\|\partial^i_x\partial^j_z \tilde{\Phi}^t(s,x,z,\cdot)\right\|_{\infty}+\left\|\partial_s \tilde{\Phi}^t(s,x,z,\cdot)\right\|_{\infty}\right]\leq C_T(1+|x|^{k}). \label{WR2}
	\end{eqnarray}
Using It\^o's formula and taking expectation on both sides, we get
\begin{eqnarray*}
\EE\tilde{\Phi}^t(t,\bar{X}_t,\eta^{\vare}_t,\alpha_t^{\varepsilon})=\!\!\!\!\!\!\!\!\!&&\tilde{\Phi}^t(0,x,0,\alpha)+\EE\int^t_0 \partial_s \tilde{\Phi}^t(s, \bar{X}_s,\eta^{\vare}_s,\alpha_s^{\varepsilon})ds\\
&&+\EE\int^t_0 \!\sum_{k=1}^3\mathcal{L}_k\tilde{\Phi}^t(s,\bar{X}_s,\eta^{\vare}_s,\alpha_s^{\varepsilon})ds\!+\!\frac{1}{\varepsilon}\EE\int_0^tQ(X_{s}^{\varepsilon})\tilde{\Phi}^t(s,\bar{X}_s,\eta^{\vare}_s,\cdot)(\alpha_s^{\varepsilon})ds\\
&&+\int^t_{0}\Big\langle\partial_z \tilde{\Phi}^t(s,\bar{X}_s,\eta^{\vare}_s,\alpha_s^{\varepsilon}),\frac{1}{\sqrt{\varepsilon}}\big[b( X^{\vare}_s,  \alpha^{\varepsilon}_s)-\bar{b}(X^{\vare}_s)\big]\Big\rangle ds.
\end{eqnarray*}
This and \eref{WR2}, it follows
\begin{eqnarray}
&&\Big|\EE\int^t_0 F^t(s,\bar{X}_s,\eta^{\vare}_s,\alpha_s^{\varepsilon})ds\Big|= \Big|\EE\int_0^tQ(\bar{X}_s)\tilde{\Phi}^t(s,\bar{X}_s,\eta^{\vare}_s,\cdot)(\alpha_s^{\varepsilon})ds\Big|\nonumber \\
\leq\!\!\!\!\!\!\!\!&& \Big|\EE\int_0^t[Q(\bar{X}_s)\!-\!Q(X^{\vare}_s)]\tilde{\Phi}^t(s,\bar{X}_s,\eta^{\vare}_s,\cdot)(\alpha_s^{\varepsilon})ds\Big|\!+\!\Big|\EE\int_0^tQ(X^{\vare}_s)\tilde{\Phi}^t(s,\bar{X}_s,\eta^{\vare}_s,\cdot)(\alpha_s^{\varepsilon})ds\Big|\nonumber \\
\leq\!\!\!\!\!\!\!\!&&\Big|\EE\int_0^t[Q(\bar{X}_s)-Q(X^{\vare}_s)]\tilde{\Phi}^t(s,\bar{X}_s,\eta^{\vare}_s,\cdot)(\alpha_s^{\varepsilon})ds\Big|\nonumber \\
&&+\!\vare\Big|\tilde{\Phi}^t(0,\!x,\!0,\!\alpha)\!-\!\EE\tilde{\Phi}^t(t,\!\bar{X}_t,\!\eta^{\vare}_t,\!\alpha_t^{\varepsilon})\!+\!\EE\!\!\int^t_0\!\! \partial_s \tilde{u}^t(s, \!\bar{X}_s,\!\eta^{\vare}_s)ds\!+\!\sum_{k=1}^3\!\EE\!\!\int^t_0\!\!\mathcal{L}_k\!\tilde{\Phi}^t(s,\!\bar{X}_s,\!\eta^{\vare}_s,\!\alpha_s^{\varepsilon})ds\Big|\nonumber\\
&&+\sqrt{\vare}\Big|\EE\int^t_{0}\left\langle\partial_z \tilde{\Phi}^t(s,\bar{X}_s,\eta^{\vare}_s,\alpha_s^{\varepsilon}),\left[b( X^{\vare}_s,  \alpha^{\varepsilon}_s)-\bar{b}(X^{\vare}_s)\right]\right\rangle ds\Big|\nonumber\\
\leq\!\!\!\!\!\!\!\!&&\EE\!\int_0^t\!\!\|Q(\bar{X}_s)\!-\!Q(X^{\vare}_s)\|_{\ell}
\|\tilde{\Phi}^t(s,\bar{X}_s,\eta^{\vare}_s,\cdot)\|_{\infty}ds\!+\!C_{T}\sqrt{\vare}\sup_{s\in[0,t]}\EE\left(1\!+\!|\bar{X}_s|^k\!+\!|X^{\vare}_s|^k\!+\!|\eta^{\vare}_s|^k\right)\nonumber\\
\leq\!\!\!\!\!\!\!\!&&C_{T}(1+|x|^k)\sqrt{\vare}.\label{F3.39}
	\end{eqnarray}
	
	Combining \eref{F11}, \eref{WR2} and \eref{F3.39}, we get
	\begin{eqnarray*}
		\sup_{0\leq t\leq T}\left|\EE\phi(\eta^{\vare}_t)-\EE\phi(Z^{x,0}_t)\right|
\leq\!\!\!\!\!\!\!\!&&\sup_{0\leq t\leq T}\Big|\EE\int^t_0 F^t(s,\bar{X}_s,\eta^{\vare}_s,\alpha_s^{\varepsilon})ds\Big|\\
&&+\sqrt{\vare}\sup_{0\leq t\leq T}\Big|\EE\big[R^{\vare}_{\partial_z\tilde{u},\Phi}(t,0)+\int^t_0\langle \Phi(X_{s}^{\varepsilon},\alpha^{\varepsilon}_{s}),\partial_s\partial_z\tilde{u}^t(s, \bar{X}_s,\eta^{\vare}_s)\rangle ds\\
&&\quad\quad\quad+\langle\Phi(x,\alpha), \partial_z \tilde{u}^t(0,x,0)\rangle-\langle\Phi(X_{t}^{\varepsilon},\alpha_{t}^{\varepsilon}), \partial_z \tilde{u}^t(t, \bar{X}_t,\eta^{\vare}_t)\rangle\big]\Big|\\
\leq\!\!\!\!\!\!\!\!&&C_{T}\sqrt{\vare}(1+|x|^k).
	\end{eqnarray*}
	Finally, using \eref{5}, \eref{supZvare} and \eref{F4.2}, we obtain
	\begin{eqnarray*}
		\sup_{0\leq t\leq T}\left|\EE\phi(Z^{\vare}_{t})-\EE\phi(Z_{t})\right|\leq\!\!\!\!\!\!\!\!&&\sup_{0\leq t\leq T}\left|\EE\phi(Z^{\vare}_{t})-\EE\phi(\eta^{\vare}_{t})\right|+\sup_{0\leq t\leq T}\left|\EE\phi(\eta^{\vare}_{t})-\EE\phi(Z^{x,0}_{t})\right|\\
\leq\!\!\!\!\!\!\!\!&&\sup_{0\leq t\leq T}\EE\left[(1+|Z^{\vare}_{t}|^k+|\eta^{\vare}_{t}|^k)|Z^{\vare}_{t}-\eta^{\vare}_{t}|\right]+\sup_{0\leq t\leq T}\left|\EE\phi(\eta^{\vare}_{t})-\EE\phi(Z^{x,0}_{t})\right|\\
\leq\!\!\!\!\!\!\!\!&& C_{T}\sqrt{\vare}(1+|x|^{k}).
	\end{eqnarray*}
	The proof is complete. $\square$

\vspace{0.3cm}
\textbf{Acknowledgment}. We would like to thank Professor Renming Song for useful discussion. This work is supported by the National Natural Science Foundation of China (Nos.  12271219, 11931004, 12090010, 12090011),  the QingLan Project of Jiangsu Province and the Priority Academic Program Development of Jiangsu Higher Education Institutions.

\end{document}